\documentclass{amsart}

\usepackage[hmargin=30mm, vmargin=30mm, includefoot, twoside]{geometry}
\usepackage[english]{babel}
\usepackage{amsmath,amssymb}
\usepackage{mathrsfs}
\usepackage{dsfont}
\usepackage{lmodern}
\usepackage{mathabx}
\usepackage{hyperref}
\usepackage{color}

\newtheorem{thmintro}{Theorem}

\newtheorem{propintro}[thmintro]{Proposition}
\newtheorem{theorem}{Theorem}[section]
\newtheorem{corollary}[theorem]{Corollary}
\newtheorem{lemma}[theorem]{Lemma}
\newtheorem{prop}[theorem]{Proposition}
\newtheorem{problem}{Problem}
\theoremstyle{definition}
\newtheorem{remark}[theorem]{Remark}
\newtheorem{example}[theorem]{Example}
\newtheorem{definition}[theorem]{Definition}

\newcommand{\CC}{\mathbb{C}}
\newcommand{\GG}{\mathbb{G}}
\newcommand{\FF}{\mathbb{F}}
\newcommand{\NN}{\mathbb{N}}
\newcommand{\ZZ}{\mathbb{Z}}

\newcommand{\BBBB}{\mathcal{B}}
\newcommand{\DDD}{\mathcal{D}}
\newcommand{\UU}{\mathcal{U}}

\newcommand{\B}{\mathfrak{B}}
\newcommand{\g}{\mathfrak{g}}
\newcommand{\G}{\mathfrak{G}}
\newcommand{\hh}{\mathfrak{h}}
\newcommand{\ii}{\mathfrak{i}}
\newcommand{\nn}{\mathfrak{n}}
\newcommand{\N}{\mathfrak{N}}
\newcommand{\PP}{\mathfrak{P}}
\newcommand{\T}{\mathfrak{T}}
\newcommand{\U}{\mathfrak{U}}

\newcommand{\inv}{^{-1}}
\newcommand{\la}{\langle}
\newcommand{\ra}{\rangle}
\newcommand{\co}{\colon\thinspace}
\newcommand{\Stt}{\mathfrak{St}}

\DeclareMathOperator{\ad}{ad}

\DeclareMathOperator{\Aut}{Aut}
\DeclareMathOperator{\charact}{char}
\DeclareMathOperator{\GL}{GL}
\DeclareMathOperator{\height}{ht}
\DeclareMathOperator{\im}{im}
\DeclareMathOperator{\modulo}{mod}
\DeclareMathOperator{\re}{re}
\DeclareMathOperator{\SL}{SL}
\DeclareMathOperator{\Hom}{Hom}
\DeclareMathOperator{\gr}{gr}
\DeclareMathOperator{\mm}{{\boldsymbol m}}
\DeclareMathOperator{\rr}{{\boldsymbol r}}
\DeclareMathOperator{\s}{{\boldsymbol s}}

\numberwithin{equation}{section}

\begin{document}

\renewcommand{\proofname}{{\bf Proof}}

\title[{Lie correspondence for complete Kac--Moody groups}]{Around the Lie correspondence for complete Kac--Moody groups and Gabber--Kac simplicity}
\author[Timoth\'ee Marquis]{Timoth\'ee \textsc{Marquis}$^*$}
\address{Department Mathematik, FAU Erlangen-Nuernberg, Cauerstrasse 11, 91058 Erlangen, Germany}
\email{marquis@math.fau.de}
\thanks{$^*$Supported by a Marie Curie Intra-European Fellowship}
\subjclass[2010]{20G44, 20E42 (primary), and 20E18 (secondary)}

\begin{abstract}
Let $k$ be a field and $A$ be a generalised Cartan matrix, and let ${\mathfrak G}_A(k)$ be the corresponding minimal Kac--Moody group of simply connected type over $k$. Consider the completion ${\mathfrak G}_A^{pma}(k)$ of ${\mathfrak G}_A(k)$ introduced by O. Mathieu and G. Rousseau, and let ${\mathfrak U}_A^{ma+}(k)$ denote the unipotent radical of the positive Borel subgroup of ${\mathfrak G}_A^{pma}(k)$. In this paper, we exhibit some functorial dependence of the groups ${\mathfrak U}_A^{ma+}(k)$ and ${\mathfrak G}_A^{pma}(k)$ on their Lie algebra. 
We also produce a large class of examples of minimal Kac--Moody groups ${\mathfrak G}_A(k)$ that are not dense in their Mathieu--Rousseau completion ${\mathfrak G}_A^{pma}(k)$.
In addition, we explain how the problematic of providing a unified theory of complete Kac--Moody groups is related to the problem of Gabber--Kac simplicity of ${\mathfrak G}_A^{pma}(k)$, asking whether every normal subgroup of ${\mathfrak G}_A^{pma}(k)$ that is contained in ${\mathfrak U}_A^{ma+}(k)$ must be trivial. We contribute to this problem by giving the first counter-examples to Gabber--Kac simplicity.
We further present several motivations for the study of this problem, as well as several applications of our functoriality theorem, with contributions to the question of (non-)linearity of ${\mathfrak U}_A^{ma+}(k)$, and to the isomorphism problem for complete Kac--Moody groups over finite fields. For $k$ finite, we also make some observations on the structure of ${\mathfrak U}_A^{ma+}(k)$ in the light of some important concepts from the theory of pro-$p$ groups.
\end{abstract}

\maketitle

\section{Introduction}
The main theme of this paper is the correspondence between the properties of a complete Kac--Moody group and its Lie algebra over an arbitrary field, with a special emphasis on the case of finite ground fields.

Let $A=(a_{ij})_{i,j\in I}$ be a generalised Cartan matrix (GCM) and let $\g=\g(A)$ be the associated Kac--Moody algebra (\cite{Kac}). 
Let also $\G_A$ denote the corresponding Tits functor of simply connected type, as defined by J.~Tits (\cite{Tits87}). 
Given a field $k$, the value of $\G_A$ over $k$ is called a \emph{minimal Kac--Moody group}. This terminology is justified by the existence of larger groups, called \emph{maximal} or \emph{complete Kac--Moody groups}, which can be constructed as completions $\widehat{\G}_A(k)$ of $\G_A(k)$ with respect to some suitable topology. 
For instance, the completion of the affine Kac--Moody group $\SL_n(k[t,t\inv])$ of type $\widetilde{A}_{n-1}$ is the maximal Kac--Moody group $\SL_n(k(\!(t)\!))$.

Roughly speaking, a minimal Kac--Moody group $\G_A(k)$ is obtained by ``exponentiating" the real root spaces of the Kac--Moody algebra $\g$, while completions $\widehat{\G}_A(k)$ of $\G_A(k)$ are obtained by exponentiating both real and imaginary root spaces of $\g$. As a result, it becomes easier to make computations in $\widehat{\G}_A(k)$ rather than in $\G_A(k)$ (see e.g. \cite[Remark~2.8]{RCap}).
Another motivation to consider maximal Kac--Moody groups rather than minimal ones is the fact that, when $k$ is a finite field, the groups $\widehat{\G}_A(k)$ form a prominent family of simple, compactly generated totally disconnected locally compact groups. Such groups have received considerable attention in the past years (see \cite{CRWpart2} for a current state of the art). 

Unlike minimal Kac--Moody groups, whose definition is somehow ``canonical" (in the sense that the Tits functor $\G_A$ over the category of fields is uniquely determined by a small number of axioms generalising in a natural way properties of semi-simple algebraic groups), maximal Kac--Moody groups have been constructed in the literature using different approaches. There are essentially three such constructions of completions of a minimal Kac--Moody group $\G_A(k)$, which we now briefly review.

The first approach is geometric. The \emph{R\'emy--Ronan completion} $\G_A^{rr}(k)$ of $\G_A(k)$ (\cite{ReRo}) is the completion of the image of $\G_A(k)$ in the automorphism group $\Aut(X_+)$ of its associated positive building, where $\Aut(X_+)$ is equipped with the topology of uniform convergence on bounded sets. A slight variant of this construction was introduced by P-E.~Caprace and B.~R\'emy (\cite[\S 1.2]{CaRe}): the resulting group $\G_A^{crr}(k)$ admits $\G_A(k)$ as a dense subgroup and $\G_A^{rr}(k)$ as a quotient.

The second approach is representation-theoretic. The \emph{Carbone--Garland completion} $\G_A^{cg\lambda}(k)$ with dominant integral weight $\lambda$ (\cite{CarboneGarland}) is the completion of the image of $\G_A(k)$ in the automorphism group $\Aut(L_k(\lambda))$ of an irreducible $\lambda$-highest-weight module $L_k(\lambda)$ over $k$. 
Again, as for $\G_A^{rr}(k)$, this construction can be slightly modified to produce a group $\G_A^{cgr}(k)$ containing $\G_A(k)$ as a dense subgroup, rather than a quotient of $\G_A(k)$ (\cite[6.2]{Rousseau}).

The third approach is algebraic. It is closer in spirit to the construction of the Tits functor $\G_A$, and produces a (topological) group functor over the category of $\ZZ$-algebras, denoted $\G_A^{pma}$, such that $\G_A(k)$ canonically embeds in $\G_A^{pma}(k)$ for any field $k$. The group $\G_A^{pma}(k)$ was first introduced by O.~Mathieu (\cite{M88b}) and further developed by G.~Rousseau (\cite{Rousseau}), and will be called the \emph{Mathieu--Rousseau completion} of $\G_A(k)$. Over $k=\CC$, the group $\G_A^{pma}(k)$ coincides with the maximal Kac--Moody group constructed by S.~Kumar (\cite[\S 6.1.6]{Kumar}).

The Mathieu--Rousseau completion of $\G_A(k)$, which will be used in this paper, is better suited to the study of finer algebraic properties of Kac--Moody groups, as for instance illustrated in \cite{simpleKM} and \cite{RCap}. The reason for this is that the relation between $\G_A^{pma}(k)$ and its Kac--Moody algebra $\g$ is more transparent than for the other completions. Our first theorem further illustrates this statement.

Let $\g=\hh\oplus\bigoplus_{\alpha\in\Delta(A)}{\g_{\alpha}}$ be the root decomposition of $\g$ with respect to its Cartan subalgebra $\hh$, with corresponding set of roots $\Delta(A)$ (resp. of positive roots $\Delta_+(A)$, of positive real roots $\Delta_+^{\re}(A)$). Let $\g_{\ZZ}$ denote the standard $\ZZ$-form of $\g$ introduced by J.~Tits (\cite[\S 4]{Tits87}) and set $\g_k:=\g_{\ZZ}\otimes_{\ZZ} k$. Set also $\nn^+(A):=\bigoplus_{\alpha\in\Delta_+(A)}{\g_{\alpha}}$ and $\nn_k^+(A):=(\nn^+(A)\cap\g_{\ZZ})\otimes_{\ZZ}k$. Finally, let $\U_A^{ma+}(k)$ denote the unipotent radical of the positive Borel subgroup of $\G_A^{pma}(k)$: the Lie algebra corresponding to $\U_A^{ma+}(k)$ is then some completion of $\nn_k^+(A)$.
Our first theorem exhibits some ``functorial dependence" of the group $\U_A^{ma+}(k)$ on its Lie algebra.

\begin{thmintro}\label{thmintro:funct1}
Let $k$ be a field, and let $A=(a_{ij})_{i,j\in I}$ and $B=(b_{ij})_{i,j\in I}$ be two GCM such that $|b_{ij}|\leq |a_{ij}|$ for all $i,j\in I$. Then the following assertions hold:
\begin{itemize}
\item[(i)]
There exists a surjective Lie algebra morphism $\pi\co \nn^+(A)\to\nn^+(B)$.
\item[(ii)]
$\pi$ gives rise to a surjective, continuous and open group homomorphism $$\widehat{\pi}\co\U_A^{ma+}(k)\to\U_B^{ma+}(k).$$
\end{itemize}
\end{thmintro}
\noindent
A more precise version of Theorem~\ref{thmintro:funct1} is given in \S\ref{subsection:SZRM} below (see Theorem~\ref{thm:construction_pi}).

Now that we have introduced the three constructions of maximal Kac--Moody groups that can be found in the literature, a very natural question arises: how do these constructions compare to one another? Or, more optimistically stated: \emph{do the geometric, representation-theoretic and algebraic completions of $\G_A(k)$ yield isomorphic topological groups?} Surprisingly, the answer to this question is yes in many cases, and conjecturally yes in \emph{almost} all cases. However, when the field has positive characteristic $p$ smaller than $M_A:=\max_{i\neq j}{|a_{ij}|}$, things become more subtle.

One obstruction to an affirmative answer in all cases is the fact that the closure $\overline{\G_A}(k)$ of $\G_A(k)$ in its Mathieu--Rousseau completion $\G_A^{pma}(k)$ might be proper: in \cite{simpleKM}, we gave for each finite field $k$ an infinite family of GCM $A$ such that $\overline{\G_A}(k)\neq \G_A^{pma}(k)$ (see also \cite[\S 6.10]{Rousseau} for an example over $k=\FF_2$). Here, we exhibit a much wider class of examples.

\begin{propintro}\label{thmintro:nondensity}
Let $k=\FF_q$ be a finite field, and let $A=(a_{ij})_{i,j\in I}$ be a GCM. Assume that there exist indices $i,j\in I$ such that $|a_{ij}|\geq q+1$ and $|a_{ji}|\geq 2$. 
Then $\G_A(k)$ is not dense in $\G_A^{pma}(k)$.
\end{propintro}
\noindent
We give two completely different proofs of this theorem. The first relies on Theorem~\ref{thmintro:funct1}. The second is more constructive, and provides another perspective on this non-density phenomenon. The proof of Proposition~\ref{thmintro:nondensity} can be found in \S\ref{section:ND} below. Note that $\overline{\G_A}(k)=\G_A^{pma}(k)$ as soon as the characteristic of $k$ is zero or bigger than $M_A$ (see \cite[6.11]{Rousseau}). 

On the other hand, G.~Rousseau proved that there always exist continuous group homomorphisms $\overline{\G_A}(k)\to \G_A^{cgr}(k)$ and $\G_A^{cgr}(k)\to\G_A^{crr}(k)$, which are moreover isomorphisms as soon as $\charact k=0$ and $A$ is symmetrisable (see \cite[6.3 and 6.7]{Rousseau}). When $k$ is finite, these homomorphisms are surjective, but the question of their injectivity is open. 

Assume that $\overline{\G_A}(k)=\G_A^{pma}(k)$ and denote by $\phi\co \G_A^{pma}(k)\to \G_A^{crr}(k)$ the composition of the two above homomorphisms. The kernel of $\phi$ then coincides with $Z'_A\cap \U_A^{ma+}(k)$, where $Z'_A$ denotes the kernel of the $\G_A^{pma}(k)$-action on its associated building $X_+$. The injectivity of $\phi$ thus amounts to $Z'_A\cap \U_A^{ma+}(k)$ being trivial or, equivalently, to the statement that every normal subgroup of $\G_A^{pma}(k)$ that is contained in $\U_A^{ma+}(k)$ must be trivial. If this is the case, we call $\G_A^{pma}(k)$ \emph{simple in the sense of the Gabber--Kac theorem}, or simply \emph{GK-simple}. This terminology is motivated by its Lie algebra counterpart, stating that, at least in the symmetrisable case, every (graded) ideal of the Kac--Moody algebra $\g$ that is contained in $\nn^+(A)$ must be trivial: this is an equivalent formulation of the Gabber--Kac theorem\footnote{Here, we define a Kac--Moody algebra using the Serre relations, as in \cite[\S 5.12]{Kac}} (\cite[Theorem~9.11]{Kac}). When $\charact k=0$ and $A$ is symmetrisable, $\G_A^{pma}(k)$ is known to be GK-simple (see \cite[Remarque~6.9.1]{Rousseau}). However, in the other cases, the following problem is widely open:
\begin{problem}[GK-simplicity problem]\label{problem:GKsimplicity}
Let $A$ be a GCM and let $k$ be a field. Determine when $\G_A^{pma}(k)$ is GK-simple.
\end{problem}
To give a feeling for the difficulty of Problem~\ref{problem:GKsimplicity}, note that in characteristic zero (say $k=\CC$), the GK-simplicity of $\G_A^{pma}(k)$ is equivalent to the Gabber--Kac theorem for $\g_k$ (see \cite[Remark~8.104(1)]{KMGbook}); when $A$ is not symmetrisable, this latter problem remains, decades after it was first considered, completely open.
As a second application of Theorem~\ref{thmintro:funct1}, we give the first (negative) contribution to Problem~\ref{problem:GKsimplicity} over finite fields.
\begin{propintro}\label{thmintro:nonGK-simple}
Let $k=\FF_q$ be a finite field. Consider the GCM $A=(\begin{smallmatrix}2 & -m\\ -n & 2\end{smallmatrix})$ with $m,n\geq 2$ and $mn>4$. Assume that $m\equiv n\equiv 2 \ (\modulo q-1)$. If $\charact k=2$, we moreover assume that at least one of $m$ and $n$ is odd. Then $\G^{pma}_{A}(k)$ and $\overline{\G_{A}}(k)$ are not GK-simple, that is, $Z'_A\cap\overline{U_A^+}(k)\neq\{1\}$.
\end{propintro}
The proof of Proposition~\ref{thmintro:nonGK-simple} is given in \S\ref{section:ND} (see Proposition~\ref{prop:noGK}).
Note that the above counter-examples to GK-simplicity all occur for $\charact k<M_A$; the hope is that for $\charact k>M_A$, Problem~\ref{problem:GKsimplicity} has a positive answer. 

To illustrate why the Lie correspondence is better behaved when $\charact k>M_A$, we make the following observations on the pro-$p$ group $\U_A^{ma+}(k)$ (for $k$ a finite field of characteristic $p$) in the light of some important pro-$p$ group concepts, such as the \emph{Zassenhaus--Jennings--Lazard (ZJL) series} (also known as the series of \emph{dimension subgroups}, see \cite[\S 11.1]{padicanalytic}). Given a pro-$p$ group $G$ with ZJL series $(D_n)_{n\geq 1}$, the space $L=\bigoplus_{n\geq 1}{D_n/D_{n+1}}$ has the structure of a graded Lie algebra over $\FF_p$, called the \emph{ZJL Lie algebra} of $G$ (see \cite[page 280]{padicanalytic}).

\begin{propintro}\label{thmintro:ZJLseries}
Let $A$ be a GCM and let $k$ be a finite field of characteristic $p>M_A$. Then the following assertions hold:
\begin{enumerate}
\item The ZJL series of $\U_A^{ma+}(k)$ coincides with its lower central series.
\item The ZJL Lie algebra of $\U_A^{ma+}(k)$ is isomorphic to $\nn^+_k(A)$, viewed as a Lie algebra over $\FF_p$.
\end{enumerate}
\end{propintro}
\noindent
The proof of Proposition~\ref{thmintro:ZJLseries} is given in \S\ref{section:MDS} below.

We now present a few more functoriality results, as well as results that are either applications of Theorem~\ref{thmintro:funct1} or provide motivations for the study of Problem~\ref{problem:GKsimplicity} (or both) -- besides the motivation to clarify the relations between the different completions of $\G_A(k)$, and hence to provide a unified theory of complete Kac--Moody groups.

For each positive real root $\alpha\in \Delta_+^{\re}(A)$, we let $e_{\alpha}$ be a $\ZZ$-basis element of $\g_{\alpha}\cap \g_{\ZZ}$, and we let $e_i=e_{\alpha_i}$, $i\in I$, be the Chevalley generators of $\nn^+(A)$.
\begin{thmintro}\label{thmintro:funct2}
Let $k$ be a field, $B$ a GCM, and let $\{\beta_i \ | \ i\in I\}$ be a linearly independent finite subset of $\Delta^{\re}_+(B)$ such that $\beta_i-\beta_j\notin\Delta(B)$ for all $i,j\in I$. Then the following assertions hold:
\begin{enumerate}
\item
The matrix $A:=(\beta_j(\beta_i^{\vee}))_{i,j\in I}$ is a GCM and the map $\pi\co\nn^+(A)\to\nn^+(B):e_i\mapsto e_{\beta_i}$ is a Lie algebra morphism.
\item
$\pi$ gives rise to a continuous group homomorphism $$\widehat{\pi}\co\U_A^{ma+}(k)\to\U_B^{ma+}(k)$$ whose kernel is normal in $\G_A^{pma}(k)$. In particular, if $\G_A^{pma}(k)$ is GK-simple, then $\widehat{\pi}$ is injective.
\item The restriction of $\widehat{\pi}$ to $\U_A^{ma+}(k)\cap \G_A(k)$ extends to continuous group homomorphisms
$$\G_A(k)\to\G_B(k)\quad\textrm{and}\quad \overline{\G_A}(k)\to\overline{\G_B}(k)$$ with kernels contained in $Z'_A$.
\end{enumerate}
\end{thmintro}
\noindent
Here, we view $\G_A(k)$ and $\G_B(k)$ as subgroups of their Mathieu--Rousseau completion (with the induced topology).
Note that, in contrast to Theorem~\ref{thmintro:funct2}, the surjective map $\widehat{\pi}\co\U_A^{ma+}(k)\to\U_B^{ma+}(k)$ provided by Theorem~\ref{thmintro:funct1} can typically \emph{not} be extended to the whole group $\G_A^{ma+}(k)$ (or even to $\G_A(k)$) as soon as $A\neq B$: this is a consequence of the simplicity results for these groups (see \cite{Moodyunpublished}, \cite{simpleKM} and \cite[\S 6.13]{Rousseau}).
A more precise version of Theorem~\ref{thmintro:funct2} is given in \S\ref{subsection:IZRM} below (see Theorem~\ref{thm:injective_standard_maps}).

As a third instance of functoriality properties of Kac--Moody groups, we also establish that every symmetrisable Kac--Moody group $\G_A(k)$ can be embedded into some \emph{simply laced} Kac--Moody group $\G_B(k)$, that is, such that the off-diagonal entries of $B$ are either $0$ or $-1$. It is known that any symmetrisable GCM $A$ admits a \emph{simply laced cover}, which is a simply laced GCM $B$ for which there is an embedding $\g(A)\to\g(B)$ (see \cite[\S 2.4]{HKL15}).  
\begin{thmintro}\label{thmintro:funct3}
Let $k$ be a field and $A$ be a GCM. Let $B$ be a simply laced cover of $A$, and consider the associated embedding $\pi\co\g(A)\to\g(B)$. Then $\pi$ gives rise to continuous group homomorphisms
$$\G_A(k)\to\G_B(k)\quad\textrm{and}\quad \overline{\G_A}(k)\to\overline{\G_B}(k)$$ with kernels contained in $Z'_A$. 
\end{thmintro}
\noindent
Note that the embeddings of the minimal Kac--Moody groups (modulo center) provided by Theorem~\ref{thmintro:funct3} preserve the corresponding twin BN-pairs and hence induce embeddings of the corresponding twin buildings. As pointed out to us by B. M\"uhlherr, similar embeddings can be obtained with a totally different approach (not relying on the Lie algebra), using the techniques developed in \cite{BMCrelle} (see also \cite{BMTwinbuild}).
A more precise version of Theorem~\ref{thmintro:funct3} is given in \S\ref{subsection:SLC} below (see Theorem~\ref{thm:SLC}).

As a second motivation for the study of Problem~\ref{problem:GKsimplicity} (besides Theorem~\ref{thmintro:funct2}(2)), as well as a third application of Theorem~\ref{thmintro:funct1}, we present a contribution to the linearity question of $\U_A^{ma+}(k)$ for $k$ a finite field. The long-standing question whether $\U_A^{ma+}(k)$ is linear over some field $k'$ is still open (see \cite[\S 4.2]{RCap}). Caprace and Stulemeijer \cite{CS14} proved that, within the class of non-discrete, compactly generated, topologically simple totally disconnected locally compact groups $G$ (of which the simple Kac--Moody groups $\G_A^{pma}(k)/Z_A'$ for $k$ a finite field are examples), the existence of a linear open subgroup $U$ of $G$ (in the sense that $U$ has a continuous faithful finite-dimensional linear representation over a local field) is equivalent to the linearity of $G$ itself (even more: $G$ is in that case a simple algebraic group over a local field). The following theorem extends this result in the Kac--Moody setting, and addresses the above-mentioned linearity problem for continuous representations over local fields, provided the group $\G_A^{pma}(k)$ is GK-simple (actually, an \emph{a priori} much weaker version of the GK-simplicity of $\G_A^{pma}(k)$ would be sufficient in this case, see Remark~\ref{remark:KZ'} below).

\begin{thmintro}\label{thmintro:nonlinearity}
Let $A$ be an indecomposable GCM of non-finite type and let $k$ be a finite field. Assume that $\G_A^{pma}(k)$ is GK-simple and set $G:=\G_A^{pma}(k)/Z_A'$. Then the following assertions are equivalent:
\begin{enumerate}
\item
Every compact open subgroup of $G$ is just-infinite (i.e. possesses only finite proper quotients).
\item
$\U_A^{ma+}(k)$ is linear over a local field.
\item
$G$ is a simple algebraic group over a local field.
\item
The matrix $A$ is of affine type.
\end{enumerate}
\end{thmintro}
\noindent
The proof of Theorem~\ref{thmintro:nonlinearity} relies on the paper \cite{CS14} (which already contains the implications (2)$\Leftrightarrow$(3) and (3)$\Rightarrow$(1)), and is given in \S\ref{section:nonlinearity} below.

As a third motivation for the study of Problem~\ref{problem:GKsimplicity}, we also present a contribution to the isomorphism problem for complete Kac--Moody groups over finite fields. The isomorphism problem for minimal Kac--Moody groups has been addressed by P-E.~Caprace (\cite[Theorem~A]{thesePE}). When $k$ is a finite field, the group $\G_A(k)$ turns out to contain, in general, very little information about $A$ (see \cite[Lemma~4.3]{thesePE}). The situation for $\G_A^{pma}(k)$ is completely different (see \cite[Theorem~E]{simpleKM}), and we expect it to be possible to recover $A$ from $\G_A^{pma}(k)$ in all cases.
This difference between $\G_A(k)$ and $\G_A^{pma}(k)$ is in fact related to the non-density of $\G_A(k)$ in $\G_A^{pma}(k)$ (see the proof of Proposition~\ref{thmintro:nondensity}).

Given a GCM $A=(a_{ij})_{i,j\in I}$ and a subset $J\subseteq I$, we define the GCM $A|_{J}:=(a_{ij})_{i,j\in J}$.
\begin{propintro}\label{thmintro:isomproblem}
Let $k,k'$ be finite fields, and let $A=(a_{ij})_{i,j\in I}$ and $B=(b_{ij})_{i,j\in J}$ be GCM. Assume that $p=\charact k>M_A,M_B$ and that all rank $2$ subgroups of $\G_{A}^{pma}(k)$ and $\G_{B}^{pma}(k')$ are GK-simple. 

If $\alpha\co\G_{A}^{pma}(k)/Z'_{A}\to\G_{B}^{pma}(k')/Z'_{B}$ is an isomorphism of topological groups, then $k\cong k'$, and there exist an inner automorphism $\gamma$ of $\G_{B}^{pma}(k')/Z'_{B}$ and a bijection $\sigma\co I\to J$ such that 
\begin{enumerate}
\item
$\gamma\alpha(\U_{A|_{\{i,j\}}}^{ma+}(k))=\U_{B|_{\{\sigma(i),\sigma(j)\}}}^{ma+}(k')$ for all distinct $i,j\in I$. 
\item
$B|_{\{\sigma(i),\sigma(j)\}}\in \big\{(\begin{smallmatrix}2 & a_{ij} \\ a_{ji} & 2 \end{smallmatrix}),(\begin{smallmatrix}2 & a_{ji} \\ a_{ij} & 2 \end{smallmatrix})\big\}$ for all distinct $i,j\in I$.
\end{enumerate}
\end{propintro}
\noindent
The proof of Proposition~\ref{thmintro:isomproblem} can be found in \S\ref{section:isomproblem} below.
Note that if $\G_{A}^{pma}(k)$ is of rank $2$ and if $\alpha$ lifts to an isomorphism $\alpha\co\G_{A}^{pma}(k)\to\G_{B}^{pma}(k')$, then the conclusion of the theorem holds without any GK-simplicity assumption (see Remark~\ref{remark:rank2isom} below).

\subsubsection*{Acknowledgement}
I am very grateful to Pierre-Emmanuel Caprace for triggering the research presented in this paper, as well as for his useful comments. I would also like to thank the anonymous referee for his/her detailed comments.

\section{Preliminaries}
Throughout this paper, $\NN$ denotes the set of nonnegative integers.

\subsection{Generalised Cartan matrices} 
An integral matrix $A=(a_{ij})_{i,j\in I}$ indexed by some finite set $I$ is called a {\bf generalised Cartan matrix} (GCM) if it satisfies the following conditions:
\begin{enumerate}
\item[(C1)] $a_{ii}=2$ for all $i\in I$;
\item[(C2)] $a_{ij}\leq 0$ for all $i,j\in I$ with $i\neq j$;
\item[(C3)] $a_{ij}=0$ if and only if $a_{ji}=0$.
\end{enumerate}
Given two GCM $A=(a_{ij})_{i,j\in I}$ and $B=(b_{ij})_{i,j\in J}$, we write $B\leq A$ if $J\subseteq I$ and $|b_{ij}|\leq |a_{ij}|$ for all $i,j\in J$. 

\subsection{Kac--Moody algebras}\label{subsection:KMA}
The general reference for this paragraph is \cite[Chapters~1--5]{Kac} (see also \cite[Part~II]{KMGbook}).

Let $A=(a_{ij})_{i,j\in I}$ be a GCM and let $(\hh,\Pi=\{\alpha_i \ | \ i\in I\},\Pi^{\vee}=\{\alpha_i^{\vee} \ | \ i\in I\})$ denote a realisation of $A$, as in \cite[\S 1.1]{Kac}. Define $\tilde{\g}(A)$ to be the complex Lie algebra with generators $e_i,f_i$ ($i\in I$) and $\hh$, and with the following defining relations:
\begin{equation*}
\left\{
\begin{array}{rcll}
[e_i,f_j]&=&-\delta_{ij}\alpha_i^{\vee}  \qquad $ $ & \textrm{($i,j\in I$),}\\
$$ [h,h']&=&0  \qquad $ $ & \textrm{($h,h'\in\hh$),}\\
$$ [h,e_i]&=&\la\alpha_i,h\ra e_i,  \qquad $ $ & \textrm{($i\in I$, $h\in\hh$),}\\
$$ [h,f_i]&=&-\la\alpha_i,h\ra f_i  \qquad $ $ & \textrm{($i\in I$, $h\in\hh$).}\\
\end{array}\right.
\end{equation*}
Denote by $\tilde{\nn}^+=\tilde{\nn}^+(A)$ (respectively, $\tilde{\nn}^-=\tilde{\nn}^-(A)$) the subalgebra of $\tilde{\g}(A)$ generated by $e_i$, $i\in I$ (respectively, $f_i$, $i\in I$). Then $\tilde{\nn}^{+}$ (respectively, $\tilde{\nn}^{-}$) is freely generated by $e_i$, $i\in I$ (respectively, $f_i$, $i\in I$), and one has a decomposition
$$\tilde{\g}(A)=\tilde{\nn}^-\oplus\hh\oplus\tilde{\nn}^+\quad\textrm{(direct sum of vector spaces).}$$
Moreover, there is a unique maximal ideal $\ii'$ of $\tilde{\g}(A)$ intersecting $\hh$ trivially. It decomposes as $$\ii'=(\ii'\cap \tilde{\nn}^-)\oplus (\ii'\cap \tilde{\nn}^+)\quad\textrm{(direct sum of ideals),}$$
and contains the ideal $\ii$ of $\tilde{\g}(A)$ generated by the elements
$$x_{ij}^+:=\ad(e_i)^{1+|a_{ij}|}e_j\in \tilde{\nn}^+\quad\textrm{and}\quad x_{ij}^-:=\ad(f_i)^{1+|a_{ij}|}f_j\in \tilde{\nn}^-$$
for all $i,j\in I$ with $i\neq j$. The {\bf Kac--Moody algebra} with GCM $A$ is then the complex Lie algebra $$\g(A):=\tilde{\g}(A)/\ii.$$
We keep the same notation for the images of $e_i,f_i,\hh$ in $\g(A)$. The subalgebra $\hh$ of $\g(A)$ is called its {\bf Cartan subalgebra}. The elements $e_i,f_i$ ($i\in I$) are called the {\bf Chevalley generators} of $\g(A)$. They respectively generate the images $\nn^+=\nn^+(A)$ and $\nn^-=\nn^-(A)$ of $\tilde{\nn}^+$ and $\tilde{\nn}^-$ in $\g(A)$. The {\bf derived Kac--Moody algebra} $\g_A:=[\g(A),\g(A)]$ is generated by the Chevalley generators of $\g(A)$.

Let $Q=Q(A):=\sum_{i\in I}{\ZZ\alpha_i}$ denote the free abelian group generated by the {\bf simple roots} $\alpha_1,\dots,\alpha_n$, and set $Q_+=Q_+(A):=\sum_{i\in I}{\NN\alpha_i}$ and $Q_-=Q_-(A):=-Q_+$. Then $\g(A)$ admits a $Q$-gradation. More precisely, 
$$\g(A)=\nn^-\oplus\hh\oplus\nn^+=\bigoplus_{\alpha\in Q_-\setminus\{0\}}{\g_{\alpha}}\oplus\hh\oplus \bigoplus_{\alpha\in Q_+\setminus\{0\}}{\g_{\alpha}},$$
where for $\alpha\in Q_+\setminus\{0\}$ (respectively, $\alpha\in Q_-\setminus\{0\}$), the {\bf root space} $\g_{\alpha}$ is the linear span of all elements of the form $[e_{i_1},\dots,e_{i_s}]$ (respectively, $[f_{i_1},\dots,f_{i_s}]$) such that $\alpha_{i_1}+\dots+\alpha_{i_s}=\alpha$ (respectively, $=-\alpha$). Here we follow the standard notation
$$[x_1,x_2,\dots,x_s]:=\ad(x_1)\ad(x_2)\dots\ad(x_{s-1})(x_s).$$
The set of {\bf roots} of $\g(A)$ is $\Delta=\Delta(A):=\big\{\alpha\in Q\setminus\{0\} \ | \ \g_{\alpha}\neq \{0\}\big\}$. It decomposes as $\Delta=\Delta_+\cup\Delta_-$, where $\Delta_{\pm}=\Delta_{\pm}(A):=\Delta\cap Q_{\pm}$ is the set of {\bf positive/negative roots}. The subgroup $W=W(A)$ of $\GL(Q)$ generated by the reflections
$$s_{i}\co Q\to Q: \alpha_j\mapsto \alpha_j-a_{ij}\alpha_i$$
for $i\in I$ stabilises $\Delta$. The $W$-orbit $W.\{\alpha_i \ | \ i\in I\}\subseteq \Delta$ is called the set of {\bf real roots} and is denoted $\Delta^{\re}=\Delta^{\re}(A)$. Its complement $\Delta^{\im}=\Delta^{\im}(A):=\Delta\setminus\Delta^{\re}$ is the set of {\bf imaginary roots}. We furthermore set $\Delta^{\re}_{\pm}=\Delta^{\re}_{\pm}(A):=\Delta^{\re}\cap\Delta_{\pm}$ and $\Delta^{\im}_{\pm}=\Delta^{\im}_{\pm}(A):=\Delta^{\im}\cap\Delta_{\pm}$. Given $\alpha=\sum_{i\in I}{n_i\alpha_i}\in Q$, we call $\height(\alpha):=\sum_{i\in I}{n_i}\in\ZZ$ the {\bf height} of $\alpha$. The group $W$ also acts linearly on $Q^{\vee}=Q^{\vee}(A):=\sum_{i\in I}{\ZZ\alpha_i^{\vee}}$ by 
$$s_{i}(\alpha_j^{\vee})=\alpha_j^{\vee}-a_{ji}\alpha_i^{\vee}.$$
Given a real root $\alpha=w\alpha_i$ ($w\in W$, $i\in I$), we define the {\bf coroot} of $\alpha$ as $\alpha^{\vee}:=w\alpha_i^{\vee}\in Q^{\vee}$. Alternatively, $\alpha^{\vee}$ is the unique element of $[\g_{\alpha},\g_{-\alpha}]$ with $\alpha(\alpha^{\vee})=2$.

\subsection{Integral enveloping algebra}\label{subsection:IEA}
The general references for this paragraph are \cite{Tits87} and \cite[Section~2]{Rousseau} (see also \cite[Chapter 7]{KMGbook}).

Let $A=(a_{ij})_{i,j\in I}$ be a GCM, and consider the corresponding derived Kac--Moody algebra $\g=\g_A$. For an element $u$ of the enveloping algebra $\UU_{\CC}(\g)$ of $\g$ and an $s\in\NN$, we write
$$u^{(s)}:=\frac{u}{s!}, \quad (\ad u)^{(s)}:=\frac{1}{s!}(\ad u)^s\quad\textrm{and}\quad \binom{u}{s}:=\frac{1}{s!}u(u-1)\dots (u-s+1).$$
Let $\UU^+$, $\UU^-$ and $\UU^0$ be the $\ZZ$-subalgebras of $\UU_{\CC}(\g)$ respectively generated by the elements $e_i^{(s)}$ ($i\in I$, $s\in\NN$), $f_i^{(s)}$ ($i\in I$, $s\in\NN$) and $\binom{h}{s}$ ($h\in\sum_{i\in I}{\ZZ\alpha_i^{\vee}}$, $s\in\NN$). Then the $\ZZ$-subalgebra $\UU=\UU(A)$ of $\UU_{\CC}(\g)$ generated by $\UU^+$, $\UU^-$ and $\UU^0$ is a $\ZZ$-form of $\UU_{\CC}(\g)$, called the {\bf integral enveloping algebra} of $\g$. It has the structure of a co-invertible $\ZZ$-bialgebra with respect to the coproduct $\nabla$, co-unit $\epsilon$, and co-inverse $\tau$, whose restrictions to $\UU^+=\UU^+(A)$ are respectively given by
$$\nabla e_i^{(m)}=\sum_{k+l=m}{e_i^{(k)}\otimes e_i^{(l)}}, \quad \textrm{$\epsilon e_i^{(m)}=0$ for $m>0$} \quad\textrm{and}\quad \tau e_i^{(m)}=(-1)^me_i^{(m)}.$$
The $\ZZ$-algebra $\UU$ inherits from $\UU_{\CC}(\g)$ a natural filtration, as well as a $Q$-gradation $\UU=\bigoplus_{\alpha\in Q}{\UU_{\alpha}}$.
We set $\g_{\ZZ}:=\g\cap\UU$ and $\nn^+_{\ZZ}=\nn^+_{\ZZ}(A):=\nn^+\cap\UU$. For $\alpha\in Q_+$, we also set $(\nn^+_{\ZZ})_{\alpha}:=\nn^+_{\ZZ}\cap \g_{\alpha}$. For a field $k$, we similarly write $\g_{k}:=\g_{\ZZ}\otimes_{\ZZ}k$, $\nn^+_k=\nn^+_k(A):=\nn^+_{\ZZ}\otimes_{\ZZ}k$ and $(\nn^+_{k})_{\alpha}:=(\nn^+_{\ZZ})_{\alpha}\otimes_{\ZZ}k$, as well as $\UU_k:=\UU\otimes_{\ZZ}k$.

A set of roots $\Psi\subseteq\Delta_+$ is called {\bf closed} if for all $\alpha,\beta\in\Psi$: $\alpha+\beta\in\Delta_+\implies\alpha+\beta\in\Psi$. 
For a closed set $\Psi\subseteq\Delta_+$, we let $\UU(\Psi)$ denote the $\ZZ$-subalgebra of $\UU^+$ generated by all $\UU^{\alpha}:=\UU_{\CC}(\oplus_{n\geq 1}\g_{n\alpha})\cap\UU^+$ for $\alpha\in\Psi$. Given a field $k$, we define the completion $\widehat{\UU}_k(\Psi)$ of $\UU(\Psi)$ over $k$ with respect to the $Q_+$-gradation as $$\widehat{\UU}_k(\Psi)=\prod_{\alpha\in Q_+}{(\UU(\Psi)_{\alpha}\otimes_{\ZZ}k)},$$
where $\UU(\Psi)_{\alpha}:=\UU(\Psi)\cap\UU_\alpha$. For $\Psi=\Delta^+$, we also write $\widehat{\UU}_k^+=\widehat{\UU}_k^+(A):=\widehat{\UU}_k(\Delta^+)$, as well as $\UU^+_{\alpha}:=\UU(\Delta_+)_{\alpha}$.

For each $i\in I$, the element $$s_i^*:=\exp(\ad e_i)\exp(\ad f_i)\exp(\ad e_i)\in \Aut(\UU)$$ satisfies $s_i^*(\UU_{\alpha})=\UU_{s_i(\alpha)}$ for all $\alpha\in Q$. We denote by $W^*=W^*(A)$ the subgroup of $\Aut(\UU)$ generated by the $s_i^*$, $i\in I$. There is a surjective group homomorphism $$\pi_W\co W^*\to W:s_i^*\mapsto s_i$$ such that for any $w^*\in W^*$ and any $i\in I$, the pair $E_{\alpha}:=\{w^*e_i,-w^*e_i\}$ only depends on the root $\alpha:=\pi_W(w^*)\alpha_i\in\Delta^{\re}$, that is, it is the same for any decomposition $\alpha=\pi_W(v^*)\alpha_j$. Moreover, for any $w\in W$ and any reduced decomposition $w=s_{i_1}s_{i_2}\dots s_{i_k}$ for $w$, the element $w^*:=s_{i_1}^*s_{i_2}^*\dots s_{i_k}^*\in W^*$ only depends on $w$, and not on the choice of the reduced decomposition for $w$.
For each $\alpha\in\Delta^{\re}$, we make some choice of an element $e_{\alpha}\in E_{\alpha}$ (with $e_{\alpha_i}:=e_i$ and $e_{-\alpha_i}:=f_i$ for $i\in I$), so that $e_{\alpha}=w^*e_i$ for some $w^*\in W^*$ and $i\in I$ with $\alpha=\pi_W(w^*)\alpha_i$. Then $\{e_{\alpha}\}$ is a $\ZZ$-basis for $\g_{\alpha}\cap\UU$, and we set $$s_{\alpha}^*:=\exp(\ad e_{\alpha})\exp(\ad e_{-\alpha})\exp(\ad e_{\alpha})=w^*s_i^*(w^*)\inv\in W^*.$$

\begin{lemma}\label{lemma:W_bialgebra_morphism}
The group $W$ acts on $\UU$ by bialgebras morphisms.
\end{lemma}
\begin{proof}
Let $u\in\UU$ and $i\in I$. Since the coproduct $\nabla$ is an algebra morphism, we have
\begin{equation*}
\begin{aligned}
\nabla(s_i^*u)&=\nabla\Big(\sum_{n_1,n_2,n_3\geq 0}{\!\!\! (\ad e_i)^{(n_1)}(\ad f_i)^{(n_2)}(\ad e_i)^{(n_3)}u}\Big)\\
&=\sum_{n_1,n_2,n_3\geq 0}{\!\!\! (\ad e_i\otimes {\mathbf 1}+{\mathbf 1}\otimes \ad e_i)^{(n_1)}(\ad f_i\otimes {\mathbf 1}+{\mathbf 1}\otimes \ad f_i)^{(n_2)}(\ad e_i\otimes {\mathbf 1}+{\mathbf 1}\otimes \ad e_i)^{(n_3)}\nabla(u)}\\
&=\sum_{\begin{smallmatrix}r_1,r_2,r_3\geq 0 \\ s_1,s_2,s_3\geq 0\end{smallmatrix}}{\!\!\! \big((\ad e_i)^{(r_1)}(\ad f_i)^{(r_2)}(\ad e_i)^{(r_3)}\otimes (\ad e_i)^{(s_1)}(\ad f_i)^{(s_2)}(\ad e_i)^{(s_3)}\big)\nabla(u)}\\
&=(s_i^*\otimes s_i^*)\nabla(u),
\end{aligned}
\end{equation*}
and hence $\nabla s_i^*=(s_i^*\otimes s_i^*)\nabla$. Since clearly $\epsilon s_i^*=\epsilon$ for all $i\in I$, the lemma follows.
\end{proof}

\subsection{Minimal Kac--Moody groups}\label{subsection:MKG}

The general references for this paragraph are \cite{Tits87} and \cite[Chapter~9]{theseBR} (see also \cite[Chapter~7]{KMGbook}).

Given a GCM $A=(a_{ij})_{i,j\in I}$, we denote by $\G_A$ the corresponding {\bf Tits functor} of simply connected type. As a group functor over the category of fields, it is characterised by a small number of properties; one of them ensures that the complex group $\G_A(\CC)$ admits an adjoint action by automorphisms on the corresponding derived Kac--Moody algebra $\g=\g_A$. {\bf Minimal Kac--Moody groups} are by definition the groups obtained by evaluating such Tits functors over a field $k$.

The minimal Kac--Moody group $\G_A(k)$ can be constructed by generators and relations, as follows. For each real root $\alpha\in\Delta^{\re}$, we let $U_{\alpha}$ denote the affine group scheme over $\ZZ$ with Lie algebra $\g_{\alpha}\cap\g_{\ZZ}=\ZZ e_{\alpha}$, and we denote by
$x_{\alpha}\co \GG_a\stackrel{\sim}{\to}U_{\alpha}$
the isomorphism from the additive group scheme $\GG_a$ to $U_{\alpha}$ determined by the choice of $e_{\alpha}\in E_{\alpha}$ as a $\ZZ$-basis element, that is, 
$$x_{\alpha}\co \GG_a(k)=(k,+)\stackrel{\sim}{\to}U_{\alpha}(k): r\mapsto \exp(re_{\alpha})\quad\textrm{for any field $k$.}$$
A pair of roots $\{\alpha,\beta\}\subseteq \Delta^{\re}$ is called {\bf prenilpotent} if there exist some $w,w'\in W$ such that $\{w\alpha,w\beta\}\subseteq \Delta_+^{\re}$ and $\{w'\alpha,w'\beta\}\subseteq \Delta_-^{\re}$. In this case, the interval $$[\alpha,\beta]_{\NN}:=(\NN\alpha+\NN\beta)\cap \Delta^{\re}$$ is finite. One then defines a group functor $\Stt_A$, called the {\bf Steinberg functor} associated to $A$, such that for any field $k$, the group $\Stt_A(k)$ is the quotient of the free product of the {\bf real root groups} $U_{\gamma}(k)$, $\gamma\in\Delta^{\re}$, by the relations
\begin{equation}\label{R0}
[x_{\alpha}(r),x_{\beta}(s)]=\prod_{\gamma}{x_{\gamma}(C^{\alpha\beta}_{ij}r^is^j)}\quad\textrm{for any $r,s\in k$ and any prenilpotent pair $\{\alpha,\beta\}\subseteq\Delta^{\re}$},
\end{equation}
where $\gamma=i\alpha+j\beta$ runs through $]\alpha,\beta[_{\NN}:=[\alpha,\beta]_{\NN}\setminus\{\alpha,\beta\}$ in some prescribed order, and where the $C^{\alpha\beta}_{ij}$ are integral constants (that can be computed) depending on $\alpha,\beta$ and on the chosen order on $]\alpha,\beta[_{\NN}$ (see \cite[9.2.2]{theseBR}). The canonical homomorphisms $U_{\gamma}(k)\to\Stt_A(k)$ turn out to be injective, and we may thus identify each $U_{\gamma}(k)$ with its image in $\Stt_A(k)$. There is a $W^*$-action on $\Stt_A(k)$, defined for any $w^*\in W^*$, $r\in k$ and $\gamma\in\Delta^{\re}$ by 
$$w^*(x_{\gamma}(r))=w^*(\exp(re_{\gamma})):=\exp(rw^*e_{\gamma})=x_{w\gamma}(\epsilon r),$$
where $w:=\pi_W(w^*)\in W$ and where $\epsilon\in\{\pm 1\}$ corresponds to the choice $e_{w\gamma}=\epsilon w^*e_{\gamma}\in E_{w\gamma}$. For any $i\in I$ and $r\in k^{\times}$, we define the element
$$\widetilde{s}_i(r):=x_{\alpha_i}(r)x_{-\alpha_i}(r\inv)x_{\alpha_i}(r)$$
of $\Stt_A(k)$ and we set $\widetilde{s}_i:=\widetilde{s}_i(1)$. 

The second step of the construction is to define the {\bf split torus scheme} $\T=\T_{A}$. Let $\Lambda$ be the free $\ZZ$-module whose $\ZZ$-dual $\Lambda^{\vee}$ is freely generated by $\{\alpha_i^{\vee} \ | \ i\in I\}$. In particular, $\{\alpha_i \ | \ i\in I\}\subseteq \Lambda$, where we view each simple root $\alpha_i$ as a linear functional on $\sum_{i\in I}{\CC\alpha_i^{\vee}}$.
For any field $k$, we set
$$\T(k):=\Hom_{\gr}(\Lambda,k^{\times})\cong (k^{\times})^{|I|}.$$
The torus $\T(k)$ is then generated by the elements
$$r^{\alpha_i^{\vee}}\co \Lambda\to k^{\times}: \lambda\mapsto r^{\alpha_i^{\vee}}(\lambda):=r^{\langle \lambda,\alpha_i^{\vee}\rangle}$$
for $r\in k^{\times}$ and $i\in I$. There is a $W$-action on $\T(k)$, defined for any $i,j\in I$ and $r\in k^{\times}$ by
$$s_i(r^{\alpha_j^{\vee}})=r^{s_i(\alpha_j^{\vee})}=r^{\alpha_j^{\vee}-a_{ji}\alpha_i^{\vee}}.$$

For any field $k$, the {\bf minimal Kac--Moody group $\G_A(k)$ of simply connected type} is now defined as the quotient of the free product $\Stt_A(k)*\T(k)$ by the following relations, where $i\in I$, $r\in k$ and $t\in\T(k)$:
\begin{align}
&t\cdot x_{\alpha_i}(r)\cdot t\inv = x_{\alpha_i}(t(\alpha_i)r), \label{R1} \\ 
&\widetilde{s}_i\cdot t\cdot \widetilde{s}_i^{\thinspace -1} = s_i(t), \label{R2} \\ 
&\widetilde{s}_i(r\inv)=\widetilde{s}_i\cdot r^{\alpha_i^{\vee}} \qquad\textrm{for $r\neq 0$,} \label{R3} \\ 
&\widetilde{s}_i\cdot u\cdot \widetilde{s}_i^{\thinspace -1} = s_i^*(u) \quad\textrm{for $u\in U_{\gamma}(k)$,}\quad \gamma\in\Delta^{\re}. \label{R4}
\end{align}

We let $U^+(k)=U_A^+(k)$ denote the subgroup of $\G_{A}(k)$ generated by all $U_{\alpha}(k)$ with $\alpha\in\Delta_+^{\re}$. The normaliser of $U^+(k)$ in $\G_{A}(k)$ is the {\bf standard Borel subgroup} $\B^+(k)=\T(k)\ltimes U^+(k)$. The center ${\mathcal Z}_A(k)$ of $\G_A(k)$ is given by 
\begin{equation}\label{eqn:center}
{\mathcal Z}_A(k)=\bigcap_{g\in\G_A(k)}{g\B^+(k)g\inv}=\{t\in \T(k) \ | \ t(\alpha_i)=1 \ \forall i\in I\}.
\end{equation}
We let $\N(k)=\N_A(k)$ denote the subgroup of $\G_{A}(k)$ generated by $\T(k)$ and by the elements
$\tilde{s}_i$ for $i\in I$. Then the assignment $\widetilde{s}_i\mapsto s_i$ for $i\in I$ induces an isomorphism $\N(k)/\T(k)\cong W$, and $(\B^+(k),\N(k))$ is a BN-pair for $\G_{A}(k)$.

\subsection{Mathieu--Rousseau completions}\label{subsection:MR completions}
The general reference for this paragraph is \cite{Rousseau} (see also \cite[\S 8.5]{KMGbook}).

Let $A=(a_{ij})_{i,j\in I}$ be a GCM. For each closed set $\Psi\subseteq\Delta_+(A)$ of positive roots, we let $\U^{ma}_{\Psi}$ denote the affine group scheme (viewed as a group functor) whose algebra is the restricted dual $\ZZ[\U^{ma}_{\Psi}]:=\bigoplus_{\alpha\in\NN\Psi}{\UU(\Psi)^*_{\alpha}}$ of $\UU(\Psi)$. 
One can then define {\bf real} and {\bf imaginary root groups} $\U_{(\alpha)}=\U^A_{(\alpha)}$ in $$\U_A^{ma+}:=\U^{ma}_{\Delta_+}$$ by setting $\U_{(\alpha)}:=\U^{ma}_{\{\alpha\}}$ for $\alpha\in\Delta_+^{\re}$ and $\U_{(\alpha)}:=\U^{ma}_{\ZZ_{>0}\alpha}$ for $\alpha\in\Delta_+^{\im}$. 

The {\bf Mathieu--Rousseau completion} $\G^{pma}_A$ of the Tits functor $\G_A$ is a group functor, with the following properties. It contains the split torus scheme $\T_A$, as well as the group functors $\U_A^{ma+}$ and $\N_A$ as subfunctors. Over a field $k$, the identification of the real root groups $U_{\alpha}(k)$ ($\alpha\in\Delta^{\re}_+$) of $\G_{A}(k)$ with the corresponding real root groups $\U_{(\alpha)}(k)$ in $\G^{pma}_A(k)$ produces injections of $U_A^+(k)$ in $\U_A^{ma+}(k)$ and of $\G_A(k)$ in $\G^{pma}_A(k)$. Again, the normaliser of $\U_A^{ma+}(k)$ in $\G^{pma}_A(k)$ is the {\bf standard Borel subgroup} $\B^{ma+}(k)=\T(k)\ltimes \U_A^{ma+}(k)$, and $(\B^{ma+}(k),\N(k))$ is a BN-pair for $\G^{pma}_{A}(k)$. 

The group $\G^{pma}_{A}(k)$ is a Hausdorff topological group, with basis of neighbourhoods of the identity the normal subgroups $\U^{ma}_n(k)$ ($n\in\NN$) of $\U_A^{ma+}(k)$ defined by
$$\U^{ma}_n=\U^{ma}_{A,n}:=\U^{ma}_{\Psi(n)}\quad\textrm{where $\Psi(n)=\{\alpha\in\Delta^+ \ | \ \height(\alpha)\geq n\}$.}$$
It is topologically generated by $\G_A(k)$, together with the imaginary root groups $\U_{(\alpha)}(k)$, $\alpha\in\Delta^{\im}_+$. Unlike the minimal Kac--Moody group $\G_A(k)$, the Mathieu--Rousseau completion $\G^{pma}_{A}(k)$ of $\G_A(k)$ is thus obtained by not only ``exponentiating" the real root spaces of the derived Kac--Moody algebra $\g$, but also the imaginary root spaces. 

The group functor $\U_A^{ma+}$ admits a more tractable description in terms of root groups, which we now briefly review. We call an element $x\in \nn^+_{\ZZ}$ {\bf homogeneous} if $x\in (\nn^+_{\ZZ})_{\alpha}$ for some $\alpha\in\Delta_+$. In this case, we call $\deg(x):=\alpha$ the {\bf degree} of $x$. Given an homogeneous element $x\in \nn^+_{\ZZ}$ with $\deg(x)=\alpha$, we call a sequence $(x^{[n]})_{n\in\NN}$ an {\bf exponential sequence} for $x$ if it satisfies the following conditions:
\begin{enumerate}
\item[(ES1)]
$x^{[0]}=1$, $x^{[1]}=x$, and $x^{[n]}\in\UU_{n\alpha}$ for all $n\in\NN$.
\item[(ES2)]
$x^{[n]}-x^{(n)}$ has filtration less than $n$ in $\UU_{\CC}(\g(A))$ for all $n>0$.
\item[(ES3)]
$\nabla(x^{[n]})=\sum_{k+l=n}{x^{[k]}\otimes x^{[l]}}$ and $\epsilon(x^{[n]})=0$ for all $n>0$.
\end{enumerate}
For a field $k$ and an element $\lambda\in k$, one can then define the {\bf twisted exponential} $$[\exp]\lambda x:=\sum_{n\geq 0}{\lambda^nx^{[n]}}\in\widehat{\UU}^+_k.$$
Note that an exponential sequence $(x^{[n]})_{n\in \NN}$ for $x$ always exists and is essentially unique, in the following sense (see \cite[\S 2.9]{Rousseau} or \cite[Proposition~8.50]{KMGbook}): if $(x^{\{n\}})_{n\in \NN}$ is another exponential sequence for $x$ with associated twisted exponential $\{\exp\}x$, then for any given choice of exponential sequences $(y^{[m]})_{m\in\NN}$ for the homogeneous elements of $\bigoplus_{r\geq 2}\g_{r\alpha\ZZ}$, there exist (uniquely determined) elements $x_m\in\g_{m\alpha\ZZ}$ ($m\geq 2$) such that $$\{\exp\}x=[\exp]x \cdot \prod_{m\geq 2}[\exp]x_m.$$
In particular, when $\alpha\in\Delta_+^{\re}$, one has $x^{[n]}=x^{(n)}=x^n/n!$ for all $n\in\NN$. 
The element $[\exp]\lambda x$ satisfies $\epsilon([\exp]\lambda x)=1$ and is group-like, that is, $\nabla [\exp]\lambda x=[\exp]\lambda x\widehat{\otimes}[\exp]\lambda x$. It is moreover invertible in $\widehat{\UU}^+_k$, with inverse $\tau [\exp]\lambda x=\sum_{n\geq 0}{\lambda^n\tau x^{[n]}}$.

\begin{remark}\label{remark:ES_of_0}
As we will see, it is also convenient to allow $x\in\g_{\alpha\ZZ}$ in the definition of an exponential sequence to be zero, in which case one has to specify $\alpha$, so as to make sense of (ES1). In other words, the sequence $(x^{[n]})_{n\in\NN}$ is an exponential sequence for $x=0$ \emph{viewed as an element of $\g_{\alpha\ZZ}$} if it satisfies the conditions (ES1)--(ES3) for that $\alpha$. Of course, in that case, one should rather call $\sum_{n\geq 0}x^{[n]}\otimes r^n$ for $r$ in some ring $k$ the \emph{twisted exponential of $x$ associated to $(x^{[n]})_{n\in\NN}$ and to $r$}, and one should replace the notation $[\exp](rx)$ by some other notation, such as $[\exp](r,x)$ (keeping the notation $[\exp](x)$ for $\sum_{n\geq 0}x^{[n]}$).

By the above uniqueness statement, the exponential sequences $(x^{[n]})_{n\in\NN}$ for $x=0\in\g_{\alpha\ZZ}$ ($\alpha\in\Delta_+$) can be described as follows. Fix a choice of exponential sequences for the homogeneous elements of $\bigoplus_{r\geq 2}\g_{r\alpha\ZZ}$; in particular, for $y=0$ viewed as a homogeneous element of $\g_{r\alpha\ZZ}$ ($r\geq 2$), one could take $[\exp] y:=1$. Then there exist (uniquely determined) $x_r\in\g_{r\alpha\ZZ}$ ($r\geq 2$) such that $[\exp](x):=\sum_{n\geq 0}x^{[n]}=\prod_{r\geq 2}[\exp](x_r)$; conversely, any such product defines an exponential sequence for $x=0\in\g_{\alpha\ZZ}$.
\end{remark}

For each $\alpha\in\Delta_+$, let $\BBBB_{\alpha}$ be a $\ZZ$-basis of $(\nn^+_{\ZZ})_{\alpha}$. For $\alpha\in\Delta_+^{\re}$, we choose $\BBBB_{\alpha}=\{e_{\alpha}\}$. For a closed subset $\Psi\subseteq \Delta_+$, we then call $\BBBB_\Psi=\BBBB_\Psi(A):=\bigcup_{\alpha\in\Psi}{\BBBB_{\alpha}}$ a {\bf standard $\ZZ$-basis} of $\nn^+_{\ZZ}\cap\UU(\Psi)$. The announced description of $\U_A^{ma+}$ is provided by the following proposition.
\begin{prop}[{\cite[Proposition~3.2]{Rousseau}}] \label{prop formal sum Rousseau}
Let $\Psi\subseteq\Delta_+$ be closed and let $k$ be a field. Then the following hold:
\begin{enumerate}
\item
$\U^{ma}_{\Psi}(k)$ can be identified to the multiplicative subgroup of $\widehat{\UU}_k(\Psi)$ consisting of all group-like elements of $\widehat{\UU}_k(\Psi)$ of constant term $1$.
\item
Let $\BBBB_\Psi$ be a standard $\ZZ$-basis of $\nn^+_{\ZZ}\cap\UU(\Psi)$, and choose for each $x\in \BBBB_\Psi$ an exponential sequence. Then $\U^{ma}_{\Psi}(k)\subseteq \widehat{\UU}_k(\Psi)$ consists of the products $$\prod_{x\in\BBBB_{\Psi}}{[\exp]\lambda_xx}$$ for $\lambda_x\in k$, where the product is taken in any (arbitrary) chosen order on $\BBBB_\Psi$. The expression of an element of $\U^{ma}_{\Psi}(k)$ in the form of such a product is unique.
\end{enumerate}
\end{prop}
In this paper, we will always identify $\U_A^{ma+}(k)$ with a subset of $\widehat{\UU}^+_k$, as in Proposition~\ref{prop formal sum Rousseau}(1).
The conjugation action of the torus $\T(k)$ on $\U_A^{ma+}(k)$ is then given by 
\begin{equation}\label{eqn:torus_action}
t([\exp]x)t\inv=[\exp]t(\alpha)x\quad\textrm{for all $t\in\T(k)$ and $x\in (\nn^+_{k})_{\alpha}$, $\alpha\in\Delta_+$.}
\end{equation}
Given $i\in I$, $\lambda\in k$ and $\alpha\in\{\pm\alpha_i\}$, we also have a conjugation action of $\exp \lambda e_{\alpha}$ on $\U_{\Delta_+\setminus\{\alpha\}}^{ma}(k)$ given by
\begin{equation}\label{eqn:ei_action}
\exp(\lambda e_{\alpha})u\exp(-\lambda e_{\alpha})=\sum_{n\geq 0}{(\ad \lambda e_{\alpha})^{(n)}u}\quad\textrm{for all $u\in \U_{\Delta_+\setminus\{\alpha\}}^{ma}(k)$.}
\end{equation}

\begin{lemma}\label{lemma:Winv_twisted_exp}
Let $i\in I$, and let $x\in \nn^+_{\ZZ}$ be an homogeneous element of degree $\alpha\in \Delta_+\setminus\{\alpha_i\}$. Then for any choice of exponential sequence $(x^{[n]})_{n\in\NN}$ for $x$, the sequence $(s_i^*x^{[n]})_{n\in\NN}$ is an exponential sequence for $s_i^*x$, and we have
$$\tilde{s}_i([\exp] x)\tilde{s}_i\inv=[\exp](s_i^*x)\in\U_A^{ma+}(k)$$
for the corresponding twisted exponentials.
\end{lemma}
\begin{proof}
We first prove that $(s_i^*x^{[n]})_{n\in\NN}$ is an exponential sequence for $s_i^*x$. Since $s_i^*$ preserves the natural gradation and filtration on $\UU=\UU(A)$ and maps $\UU_{n\alpha}$ to $\UU_{ns_i(\alpha)}$ ($n\in\NN$), the axioms (ES1) and (ES2) are clearly satisfied. 
Since moreover $s_i^*$ acts on $\UU$ by bialgebra morphisms by Lemma~\ref{lemma:W_bialgebra_morphism}, the axiom (ES3) is also satisfied, as desired. The second statement of the lemma follows from (\ref{eqn:ei_action}).
\end{proof}

\subsection{Gabber--Kac kernel and non-density}\label{subsection:GKKND}
The general reference for this paragraph is \cite[Section~6]{Rousseau} (see also \cite[\S 8.5--\S 8.6]{KMGbook}).

Let $A=(a_{ij})_{i,j\in I}$ be a GCM and $k$ be a field. The minimal Kac--Moody group $\G_A(k)$ acts strongly transitively by simplicial automorphisms on its positive building $X_+$, associated to the BN-pair $(\B^+(k),\N(k))$ of $\G_A(k)$. (For general background on buildings and BN-pairs, we refer the reader to \cite[Chapter~6]{BrownAbr}).

The {\bf R\'emy--Ronan completion} $\G_A^{rr}(k)$ of $\G_A(k)$ (see \cite{ReRo}) is the completion of the image of $\G_A(k)$ in the automorphism group $\Aut(X_+)$ of $X_+$, where $\Aut(X_+)$ is equipped with the topology of uniform convergence on bounded sets. The BN-pair $(\B^{ma+}(k),\N(k))$ of the Mathieu--Rousseau completion $\G^{pma}_A(k)$ of $\G_A(k)$ yields the same building $X_+$ (possibly with a larger apartment system). The kernel of the action of $\G^{pma}_A(k)$ on $X_+$ is given by
$$Z'_A:=\bigcap_{g\in \G^{pma}_A(k)}{g\B^{ma+}(k)g\inv}$$ and decomposes as $Z'_A=Z_A\cdot (Z'_A\cap \U_A^{ma+}(k))$, where $Z_A={\mathcal Z}_A(k)$ is the center of $\G_A(k)$. We call the intersection $Z'_A\cap \U_A^{ma+}(k)$ the {\bf Gabber-Kac kernel} of $\G^{pma}_A(k)$, for reasons that will become clear in \S\ref{subsection:GKS} below. It can also be described as $$Z'_A\cap \U_A^{ma+}(k)=\bigcap_{u\in\U_A^{ma+}(k)}{uU^{im+}u\inv},$$
where $U^{im+}:=\U^{ma}_{\Delta_+^{\im}}(k)$ is the {\bf imaginary subgroup} of $\U_A^{ma+}(k)$.
Note that if $A$ is of indefinite type and $k$ is of characteristic zero or is finite, the quotient $\G^{pma}_A(k)/Z'_A$ is simple (see \cite{simpleKM} and \cite[Thm.6.19]{Rousseau}).

Unlike the R\'emy--Ronan completion, the Mathieu--Rousseau completion $\G^{pma}_A(k)$ of $\G_A(k)$ is, in general, not the completion of $\G_A(k)$ (in its own topology). Note however that $\G_A(k)$ is dense in $\G^{pma}_A(k)$ as soon as the characteristic of $k$ is either zero or bigger than 
$$M_A:=\max_{i\neq j}{|a_{ij}|}.$$
We denote by $\overline{U^+_A}(k)$ (respectively, $\overline{\G_A}(k)$) the completion of $U^+_A(k)$ (respectively, $\G_A(k)$) in $\G^{pma}_A(k)$. The completions $\overline{\G_A}(k)$ and $\G_A^{rr}(k)$ of $\G_A(k)$ are strongly related: there is a continuous homomorphism
$$\varphi_{A}\co \overline{\G_A}(k)\to \G_A^{rr}(k)$$
with kernel $Z'_A\cap \overline{\G_A}(k)=Z_A\cdot (Z'_A\cap \overline{U^+_A}(k))$, which is moreover surjective if $k$ is finite. 

\subsection{GK-simplicity}\label{subsection:GKS}
The general reference for this paragraph is \cite[6.5]{Rousseau} (see also \cite[\S 8.6]{KMGbook}).

Let $A=(a_{ij})_{i,j\in I}$ be a GCM and $k$ be a field. 
By a theorem of Gabber--Kac (see \cite[Proposition~1.7 and Theorem~9.11]{Kac}), every ideal of the derived Kac--Moody algebra $\g=\g_A$ intersecting the Cartan subalgebra $\hh$ trivially is reduced to $\{0\}$ (at least when $A$ is symmetrisable). Equivalently, every graded sub-$\g$-module of $\g$ that is contained in $\nn^+$ is reduced to $\{0\}$. The Lie algebra $\g_k$ is called \emph{simple in the sense of the Gabber--Kac theorem}, or simply {\bf GK-simple} if every graded sub-$\UU_k$-module of $\g_k$ that is contained in $\nn^+_k$ is reduced to $\{0\}$. Similarly, the Kac--Moody group $\G^{pma}_A(k)$ is called {\bf GK-simple} if every normal subgroup of $\G^{pma}_A(k)$ that is contained in $\U_A^{ma+}(k)$ is reduced to $\{1\}$.

It is easy to see that the Lie algebra $\g_k$ is GK-simple if and only if for all $\delta\in\Delta_+^{\im}$, any homogeneous element $x\in\g_k$ of degree $\delta$ such that $(\ad f_i)^{(s)}x=0$ for all $i\in I$ and $s\in\NN$ must be zero. By the Gabber--Kac theorem, $\g_k$ is GK-simple when $A$ is symmetrisable and $\charact k=0$. When $\charact k=p>0$, this is not true anymore: for instance, the affine Kac--Moody algebra $\g_k={\mathfrak sl}_m(k)\otimes_kk[t,t\inv]$ is not GK-simple as soon as $p$ divides $m$. Note, however, that the corresponding Kac--Moody group $\G_A^{pma}(k)=\SL_m(k(\!(t)\!))$ is GK-simple (see \cite[Exemple~6.8]{Rousseau}).

Note that $\G^{pma}_A(k)$ is GK-simple if and only if its Gabber--Kac kernel $Z'_A\cap \U_A^{ma+}(k)$ is trivial, that is, if and only if $Z'_A=Z_A$. If $\g_k$ is GK-simple and $k$ is infinite, then $\G^{pma}_A(k)$ is GK-simple by \cite[Remarque~6.9.1]{Rousseau}. In particular, $\G^{pma}_A(k)$ is GK-simple as soon as $A$ is symmetrisable and $\charact k=0$.

\section{Functoriality}
In this section, given two GCM $A$ and $B$, we define a family of Lie algebra maps $\nn^+(A)\to\nn^+(B)$, which we call \emph{$\ZZ$-regular}, and which give rise to continuous group homomorphisms $\U_A^{ma+}(k)\to\U_B^{ma+}(k)$ over any field $k$. We then give concrete examples of such maps, respectively yielding surjective and injective exponentials, as in Theorems~\ref{thmintro:funct1} and \ref{thmintro:funct2}. Finally, we show how Theorem~\ref{thmintro:funct3} can be deduced using the same lines of proof.

\subsection{The exponential of a $\ZZ$-regular map}
\begin{definition}\label{definition:Zreg}
Let $A=(a_{ij})_{i,j\in I}$ and $B$ be two GCM. We call a map $\pi\co\nn^+(A)\to\nn^+(B)$ {\bf $\ZZ$-regular} if it is a Lie algebra morphism such that for each $i\in I$, there is some $\beta_i\in\Delta^{\re}_+(B)$ with $\pi(e_i)\in (\nn^+_{\ZZ}(B))_{\beta_i}$. In this case, we denote by $\overline{\pi}\co Q(A)\to Q(B)$ the $\ZZ$-linear map defined by 
$$\overline{\pi}(\alpha_i)=\beta_i\quad\forall i\in I.$$
\end{definition}

\begin{theorem}\label{thm:basic_functoriality}
Let $k$ be a field, and let $A=(a_{ij})_{i,j\in I}$ and $B$ be two GCM. Let $\pi\co\nn^+(A)\to\nn^+(B)$ be $\ZZ$-regular.
Then there is a continuous group homomorphism $$\widehat{\pi}\co \U_A^{ma+}(k)\to\U_B^{ma+}(k)$$ such that for any nonzero homogeneous $x\in\nn_{\ZZ}^+(A)$ and any choice of exponential sequence for $x$, there is a choice of exponential sequence $(\pi(x)^{[n]})_{n\in\NN}$ for $\pi(x)$ such that
\begin{equation}\label{eqn:basic_exp}
\widehat{\pi}([\exp]\lambda x)=\sum_{n\in\NN}\lambda^n\pi(x)^{[n]}\quad\textrm{for all $\lambda\in k$}.
\end{equation} 
\end{theorem}
\begin{proof}
By assumption, there exist for each $i\in I$ some real root $\beta_i\in\Delta_+^{\re}(B)$ and some $\lambda_i\in\ZZ$ such that 
$$\pi(e_i)=\lambda_i e_{\beta_i} \quad\textrm{for all $i\in I$.}$$
Since $e_{\beta_i}^{(n)}\in\UU^+(B)$ for all $n\in\NN$, the map $\UU_{\CC}(\nn^+(A))\to \UU_{\CC}(\nn^+(B))$ lifting $\pi$ at the level of the corresponding enveloping algebras restricts to an algebra morphism
$$\pi_1\co \UU^+(A)\to \UU^+(B).$$
Since $W(B)$ acts on $\UU^+(B)$ by bialgebra morphisms (see Lemma~\ref{lemma:W_bialgebra_morphism}), we get
$$\nabla_B\pi_1(e_i^{(m)})=\lambda_i^m\nabla_Be_{\beta_i}^{(m)}=\lambda_i^m\sum_{r+s=m}{e_{\beta_i}^{(r)}\otimes e_{\beta_i}^{(s)}}=(\pi_1\otimes\pi_1)\sum_{r+s=m}{e_i^{(r)}\otimes e_i^{(s)}}=(\pi_1\otimes\pi_1)\nabla_A e_i^{(m)}$$
for all $i\in I$ and $m\in\NN$, where $\nabla_X$ denotes the coproduct on $\UU^+(X)$, $X=A,B$. Hence $\nabla_B\pi_1=(\pi_1\otimes\pi_1)\nabla_A$. Similarly, denoting by $\epsilon_X$ the co-unit on $\UU^+(X)$, we have $\epsilon_B\pi_1=\epsilon_A$, and hence $\pi_1$ is a bialgebra morphism.

Note also that $\pi_1$ preserves the natural gradations on $\UU^+(A)$ and $\UU^+(B)$, in the sense that
\begin{equation}\label{eqn:gradations}
\pi_1(\UU^+_{\alpha}(A))\subseteq \UU^+_{\overline{\pi}(\alpha)}(B)\quad\textrm{for all $\alpha\in Q_+(A)$}.
\end{equation}
In particular, the map $$\UU^+(A)\otimes_{\ZZ}k\to\UU^+(B)\otimes_{\ZZ}k$$ obtained from $\pi_1$ by extension of scalars can be further extended 
to a bialgebra morphism 
$$\pi_2\co \widehat{\UU}_k^+(A)\to \widehat{\UU}_k^+(B)$$ between the corresponding completions.
Finally, since $\pi_2$ preserves the group-like elements of constant term $1$, it restricts to a group homomorphism
$$\widehat{\pi}\co \U_A^{ma+}(k)\to \U_B^{ma+}(k)$$
by Proposition~\ref{prop formal sum Rousseau}(1).

Let now $x\in\nn_{\ZZ}^+(A)$ be homogeneous of degree $\alpha\in\Delta_+(A)$, and choose an exponential sequence $(x^{[n]})_{n\in\NN}$ for $x$. Then $y:=\pi(x)\in\nn^+_{\ZZ}(B)$ is homogeneous of degree $\overline{\pi}(\alpha)\in Q_+(B)$. We claim that the sequence $(y^{[n]})_{n\in\NN}$ defined by 
\begin{equation}\label{eqn:piES}
y^{[n]}:=\pi_1(x^{[n]})\quad\textrm{for all $n\in\NN$}
\end{equation}
is an exponential sequence for $y$ (viewed as an element of degree $\overline{\pi}(\alpha)$ if $y=0$, cf. Remark~\ref{remark:ES_of_0}), so that 
\begin{equation}\label{eqn:exp_sequences}
\widehat{\pi}([\exp]\lambda x)=\sum_{n\in\NN}\lambda^n\pi(x)^{[n]}\quad\textrm{for all $\lambda\in k$.}
\end{equation}
Indeed, $y^{[0]}=1$ and $y^{[1]}=y$ by the corresponding properties for $x$. Since $\pi_1(\UU^+_{n\alpha}(A))\subseteq \UU^+_{n\overline{\pi}(\alpha)}(B)$ for all $n\in\NN$, we also have $y^{[n]}\in \UU^+_{n\overline{\pi}(\alpha)}(B)$ for all $n$, so that the condition (ES1) is satisfied. Similarly, $$y^{[n]}-y^{(n)}=\pi_1(x^{[n]})-\pi_1(x)^{(n)}=\pi_1(x^{[n]}-x^{(n)})$$ has filtration less than $n$ in $\UU_{\CC}(\g(B))$, because $\pi_1$ preserves the natural filtrations, yielding (ES2). Finally, (ES3) readily follows from the corresponding property for $x$ and the fact that $\pi_1$ is a bialgebra morphism. 

Note that (\ref{eqn:gradations}) and (\ref{eqn:exp_sequences}), together with Proposition~\ref{prop formal sum Rousseau}(2), imply that 
\begin{equation}\label{eqn:pi_root_groups}
\widehat{\pi}(\U^A_{(\alpha)}(k))\subseteq \U^B_{(\overline{\pi}(\alpha))}(k)\quad\textrm{for all $\alpha\in\Delta_+(A)$},
\end{equation}
where $\U^B_{(\overline{\pi}(\alpha))}(k):=\{1\}$ if $\overline{\pi}(\alpha)\notin\Delta_+(B)$ (see also Remark~\ref{remark:ES_of_0}). Since $\height(\overline{\pi}(\alpha))\to\infty$ as $\height(\alpha)\to\infty$, $\alpha\in\Delta_+(A)$, we deduce in particular that $\widehat{\pi}$ is continuous. This concludes the proof of the theorem.
\end{proof}

\begin{definition}\label{definition:exponential}
For a $\ZZ$-regular map $\pi\co\nn^+(A)\to\nn^+(B)$, we have just proved that the unique continuous map $\widehat{\pi}\co \U_A^{ma+}(k)\to\U_B^{ma+}(k)$ defined on the (topological) generators $[\exp]\lambda x$ of $\U_A^{ma+}(k)$ by the formulas (\ref{eqn:piES}) and (\ref{eqn:exp_sequences}) (where $\lambda\in k$, $x\in\nn_{\ZZ}^+(A)$ is a homogeneous element, and $[\exp]\lambda x=\sum_{n\geq 0}\lambda^nx^{[n]}$ a twisted exponential) is a group homomorphism, which we call the {\bf exponential} of $\pi$.
\end{definition}

\subsection{Surjective $\ZZ$-regular maps}\label{subsection:SZRM}

\begin{lemma}\label{lemma:Serre_relations}
Let $A=(a_{ij})_{i,j\in I}$ be a GCM, and let $\g(A)=\tilde{\g}(A)/\ii$ be the associated Kac--Moody algebra. Then $\ii$ decomposes as a direct sum of ideals $\ii=\ii^+\oplus\ii^-$, where $\ii^{\pm}\subseteq \tilde{\nn}^{\pm}$ is generated, as an ideal of the Lie algebra $\tilde{\nn}^{\pm}$, by the elements $x_{ij}^{\pm}$, $i,j\in I$. 
\end{lemma}
\begin{proof}
Let $\ii^{\pm}$ denote the ideal of $\tilde{\nn}^{\pm}$ generated by the elements $x_{ij}^{\pm}$, $i,j\in I$. We claim that $[f_k,x_{ij}^+]=0$ for all $i,j,k\in I$. If $k\neq i$, this is clear. For $k=i$, this follows from the formula
$$[f_i,(\ad e_i)^me_j]=m(m-1-|a_{ij}|)(\ad e_i)^{m-1}e_j,$$
obtained by an easy induction on $m\geq 0$. This implies that $\ii^{+}$ is in fact an ideal in $\tilde{\g}(A)$, and similarly for $\ii^{-}$. In particular, $\ii=\ii^+\oplus\ii^-$, as desired.
\end{proof}

\begin{lemma}\label{lemma:Zreg_ex1}
Let $A=(a_{ij})_{i,j\in I}$ and $B=(b_{ij})_{i,j\in J}$ be two GCM such that $B\leq A$. Then the assignment $e_i\mapsto e_i$ if $i\in J$ and $e_i\mapsto 0$ otherwise defines a surjective Lie algebra morphism $$\pi_{AB}\co \nn_+(A)\to \nn_+(B)$$
such that $\pi(\g(A)_{\alpha})=\g(B)_{\alpha}$ for all $\alpha\in Q_+(B)=\sum_{i\in J}{\NN\alpha_i}\subseteq Q_+(A)=\sum_{i\in I}{\NN\alpha_i}$. In particular, $\pi_{AB}$ is $\ZZ$-regular and $$\Delta_+(B)\subseteq \Delta_+(A).$$
\end{lemma}
\begin{proof}
The assignment $e_i\mapsto e_i$ if $i\in J$ and $e_i\mapsto 0$ otherwise defines a surjective Lie algebra morphism $\tilde{\pi}_{AB}\co\tilde{\nn}_+(A)\to \tilde{\nn}_+(B)$, which by hypothesis maps the ideal $\ii^+(A)$ inside the ideal $\ii^+(B)$. In particular, $\tilde{\pi}_{AB}$ factors through a surjective Lie algebra morphism $\pi_{AB}\co \nn_+(A)\to \nn_+(B)$ by Lemma~\ref{lemma:Serre_relations}. Since $\g(A)_{\alpha}$ is spanned by all iterated brackets $[e_{i_1},\dots,e_{i_s}]$ with $\alpha_{i_1}+\dots+\alpha_{i_s}=\alpha$ ($\alpha\in Q_+(A)$) and similarly for $\g(B)_{\alpha}$, the other claims follow.
\end{proof}

\begin{theorem}\label{thm:construction_pi}
Let $k$ be a field. Let $A$ and $B$ be two GCM such that $B\leq A$. Let $\pi_{AB}\co \nn_+(A)\to \nn_+(B)$ be the corresponding $\ZZ$-regular map, as in Lemma~\ref{lemma:Zreg_ex1}. Then the following hold:
\begin{enumerate}
\item
The exponential $\widehat{\pi}_{AB}\co \U_A^{ma+}(k)\to\U_B^{ma+}(k)$ of $\pi_{AB}$ is surjective, continuous and open.
\item
For any closed set of roots $\Psi_A\subseteq \Delta_+(A)$,
$$\widehat{\pi}(\U^{ma}_{\Psi_A}(k))=\U^{ma}_{\Psi_A\cap\Delta_+(B)}(k).$$
In particular, $\widehat{\pi}(\U_{(\alpha)}^A(k))=\U_{(\alpha)}^B(k)$ for all $\alpha\in\Delta^+(A)$, where $\U_{(\alpha)}^B(k):=\{1\}$ if $\alpha\notin \Delta^+(B)$.
\end{enumerate}
\end{theorem}
\begin{proof}
Recall that, for any closed set of roots $\Psi_A\subseteq \Delta_+(A)$, the subgroup $\U_{\Psi_A}^{ma}(k)$ of $\U_A^{ma+}(k)$ is topologically generated by the twisted exponentials $[\exp]\lambda x$ for $\lambda\in k$ and $x\in\nn_{\ZZ}^+(A)$ a homogeneous element of degree in $\Psi_A$ (and similarly for subgroups of $\U_B^{ma+}(k)$). The surjectivity of $\widehat{\pi}_{AB}$ as well as the second statement of the theorem thus readily follow from (\ref{eqn:basic_exp}). In particular, $\widehat{\pi}_{AB}(\U^{ma}_{A,n}(k))=\U^{ma}_{B,n}(k)$ for all $n\in\NN$, and hence $\pi_{AB}$ is also open, as desired.
\end{proof}

\begin{corollary}\label{corollary:functoriality_min}
Let $k$ be a field. Let $A$ and $B$ be two GCM such that $B\leq A$. Then the map $\widehat{\pi}_{AB}\co \U_A^{ma+}(k)\to\U_B^{ma+}(k)$ restricts to group homomorphisms
$$U^+_A(k)\to U^+_B(k)\quad\textrm{and}\quad\overline{U^+_A}(k)\to\overline{U^+_B}(k).$$
\end{corollary}
\begin{proof}
By Lemma~\ref{lemma:Zreg_ex1}, we may identify $\Delta_+(B)$ with a subset of $\Delta_+(A)$. Since a root $\alpha$ is real if and only if $2\alpha$ is not a root by \cite[Propositions~5.1 and 5.5]{Kac}, we deduce that $\Delta^{\re}_+(A)\cap\Delta_+(B)\subseteq \Delta^{\re}_+(B)$.
It then follows from Theorem~\ref{thm:construction_pi} that $\widehat{\pi}_{AB}(\U_{(\alpha)}^A(k))\subseteq U^+_B(k)$ for any $\alpha\in\Delta^{\re}_+(A)$, and hence that $\widehat{\pi}_{AB}$ restricts to a map $U^+_A(k)\to U^+_B(k)$. The corresponding statement for the completions follows from the continuity of $\widehat{\pi}_{AB}$. 
\end{proof}

\begin{remark}\label{remark:restriction_minimal_groups}
Let $A=(a_{ij})_{i,j\in I}$ and $B=(b_{ij})_{i,j\in I}$ be two GCM such that $B\leq A$. Then the restriction $U^+_A(k)\to U^+_B(k)$ of $\widehat{\pi}_{AB}$ provided by Corollary~\ref{corollary:functoriality_min} is, in general, not surjective anymore. Indeed, assume for instance that the matrices $A$ and $B$ are symmetric, and for $X\in\{A,B\}$, let $(\cdot,\cdot)_X$ denote the bilinear form on $Q_+(X)$ introduced in \cite[\S 2.1]{Kac}. Thus, given $\alpha=\sum_{i\in I}{n_i\alpha_i}\in Q_+(A)=Q_+(B)$ with support $J:=\{i\in I \ | \ n_i\neq 0\}$, we have
\begin{equation}\label{eqn:rk_restriction}(\alpha,\alpha)_A=\sum_{i,j\in I}{n_in_ja_{ij}}=2\sum_{i\in I}{n_i^2}-\sum_{i\neq j}{n_in_j|a_{ij}|}\leq 2\sum_{i\in I}{n_i^2}-\sum_{i\neq j}{n_in_j|b_{ij}|}=(\alpha,\alpha)_B.
\end{equation}
Moreover, if $\alpha\in\Delta_+(X)$, then $\alpha\in\Delta^{\re}_+(X)$ if and only if $(\alpha,\alpha)_X=2$, while $\alpha\in\Delta^{\im}_+(X)$ if and only if $(\alpha,\alpha)_X\leq 0$ (see \cite[Propositions~3.9 and 5.2]{Kac}). 
In particular, if $a_{ij}\neq b_{ij}$ for some $i,j$ in the support $J$ of the real root $\alpha\in\Delta_+^{\re}(A)$, then $\alpha\notin\Delta_+(B)$, because $(\alpha,\alpha)_A<(\alpha,\alpha)_B$ by (\ref{eqn:rk_restriction}). Hence in that case the real root group $\U_{(\alpha)}^A(k)$ is in the kernel of $\widehat{\pi}_{AB}$. 
For instance, if $A=(\begin{smallmatrix}2 & -a\\ -a & 2\end{smallmatrix})$ and $B=(\begin{smallmatrix}2 & -b\\ -b & 2\end{smallmatrix})$ with $b<a$, then $\widehat{\pi}_{AB}(\U_{\alpha}^A(k))=\{1\}$ for all $\alpha\in\Delta^{\re}_+(A)\setminus\{\alpha_1,\alpha_2\}$. Thus, in that case, $\widehat{\pi}_{AB}(U^+_A(k))$ is the subgroup of $U^+_B(k)$ generated by the real root groups $\U_{\alpha_1}^B(k)$ and $\U_{\alpha_2}^B(k)$ associated to the simple roots.
\end{remark}

\subsection{Injective $\ZZ$-regular maps}\label{subsection:IZRM}

\begin{lemma}\label{lemma:subalgebras}
Let $B$ be a GCM, and let $\{\beta_i \ | \ i\in I\}$ be a finite subset of $\Delta^{\re}_+(B)$ such that $\beta_i-\beta_j\notin\Delta(B)$ for all $i,j\in I$. Then the matrix $A:=(\beta_j(\beta_i^{\vee}))_{i,j\in I}$ is a GCM. Moreover, the assignment $e_i\mapsto e_{\beta_i}$, $f_i\mapsto e_{-\beta_i}$ for all $i\in I$ defines a Lie algebra morphism $\pi\co\g_A\to\g_B$.
\end{lemma}
\begin{proof}
Let $\widetilde{\pi}$ be the Lie algebra morphism from the free complex Lie algebra on the generators $\{e_i, f_i \ | i\in I\}$ to $\g_B$ defined by the assignment $e_i\mapsto e_{\beta_i}$, $f_i\mapsto e_{-\beta_i}$ for all $i\in I$.
Since $\beta_j-\beta_i\notin\Delta(B)$ for all $i,j\in I$, we deduce from \cite[Corollary~3.6]{Kac} that $\beta_j(\beta_i^{\vee})\leq 0$, so that $A$ is indeed a GCM. Moreover, $$[e_{\beta_i},e_{-\beta_j}]=0\quad\textrm{for all $i,j\in I$ with $i\neq j$.}$$
Similarly, since $s_i(\beta_j-\beta_i)=(|\beta_j(\beta_i^{\vee})|+1)\beta_i+\beta_j\notin\Delta(B)$, we have $$(\ad e_{\pm\beta_i})^{|\beta_j(\beta_i^{\vee})|+1}e_{\pm\beta_j}=0 \quad\textrm{for all $i,j\in I$ with $i\neq j$.}$$
Finally, the elements $\beta_i^{\vee}=[e_{-\beta_i},e_{\beta_i}]$ of $\g_{B}$ ($i\in I$) satisfy
$$[\beta_i^{\vee},\beta_j^{\vee}]=0 \quad\textrm{and}\quad [\beta_i^{\vee},e_{\pm\beta_j}]=\pm\beta_j(\beta_i^{\vee})e_{\pm\beta_j}$$
for all $i,j\in I$. Hence all the defining relations of $\g_A=[\g(A),\g(A)]$ (see \S\ref{subsection:KMA}) lie in the kernel of $\widetilde{\pi}$, so that $\widetilde{\pi}$ factors through a Lie algebra morphism $\pi\co\g_A\to\g_B$.
\end{proof}

\begin{theorem}\label{thm:injective_standard_maps}
Let $k$ be a field, $B$ a GCM, and let $\{\beta_i \ | \ i\in I\}$ be a linearly independent finite subset of $\Delta^{\re}_+(B)$ such that $\beta_i-\beta_j\notin\Delta(B)$ for all $i,j\in I$. Let $A:=(\beta_j(\beta_i^{\vee}))_{i,j\in I}$ be the corresponding GCM, and consider the $\ZZ$-regular map $\pi\co\nn^+(A)\to\nn^+(B):e_i\mapsto e_{\beta_i}$. Then the following holds:
\begin{enumerate}
\item
The kernel of the exponential $\widehat{\pi}\co \U_A^{ma+}(k)\to \U_B^{ma+}(k)$ of $\pi$ is a normal subgroup of $\G^{pma}_A(k)$. In particular, if $\G_A^{pma}(k)$ is GK-simple, then $\widehat{\pi}$ is injective. 
\item
The restriction of $\widehat{\pi}$ to $U^+_A(k)$ extends to continuous group homomorphisms
$$\G_A(k)\to\G_B(k)\quad\textrm{and}\quad \overline{\G_A}(k)\to\overline{\G_B}(k)$$ with kernels respectively contained in ${\mathcal Z}_A(k)$ and ${\mathcal Z}_A(k)\cdot (Z'_A\cap \overline{U^+_A}(k))$. Here, we view $\G_A(k)$ and $\G_B(k)$ as subgroups of $\G^{pma}_A(k)$ and $\G^{pma}_B(k)$ respectively, with the induced topology.
\end{enumerate}
\end{theorem}
\begin{proof}
By Lemma~\ref{lemma:subalgebras}, the $\ZZ$-regular map $\pi$ extends to a Lie algebra morphism $$\g_A\to \g_B:e_i\mapsto e_{\beta_i}, \ f_i\mapsto e_{-\beta_i}.$$
Since $e_{\pm\beta_i}^{(n)}\in\UU(B)$ for all $i\in I$ and $n\in\NN$, this then yields a $\ZZ$-algebra morphism $\UU(A)\to \UU(B)$, which in turn extends to a $k$-algebra morphism $$\pi_1\co \UU_k(A)\to \UU_k(B).$$  Note that the $\ZZ$-linear map $$\overline{\pi}\co Q(A)\to Q(B):\alpha_i\mapsto \beta_i$$ induced by $\pi$ is injective because the $\beta_i$ are linearly independent. Set $K:=\ker\widehat{\pi}$, where $\widehat{\pi}$ is the exponential of $\pi$ (see Definition~\ref{definition:exponential}).

Choose a $\ZZ$-basis $\BBBB$ of $\nn^+_{\ZZ}(A)$, as well as exponential sequences for the elements of $\BBBB$. Choose also exponential sequences for the elements $\pi(x)\in\nn^+_{\ZZ}(B)$, $x\in\BBBB$, as in Theorem~\ref{thm:basic_functoriality}, so that for any $y=\prod_{x\in\BBBB}{[\exp]\lambda_xx}\in\U^{ma+}_A(k)$ we have $$\widehat{\pi}(y)=\prod_{x\in\BBBB}{[\exp]\lambda_x\pi(x)}\in\U^{ma+}_B(k).$$ Thus, if $y\in K$, then for any $i\in I$ the component of $\prod_{x\in\BBBB}{[\exp]\lambda_x\pi(x)}\in\widehat{\UU}^+_k$ of degree $\beta_i$ must be zero. Since $\overline{\pi}$ is injective and $\pi(e_i)=e_{\beta_i}\neq 0$, this implies that $\lambda_{e_i}=0$. Hence $K\subseteq \U_{\Delta_+\setminus\{\alpha_i  |  i\in I\}}^{ma}(k)$.

For each real root $\alpha$ and each $r\in k^{\times}$, we set 
$$s_{\alpha}^*(r):=\exp(\ad re_{\alpha})\exp(\ad r\inv e_{-\alpha})\exp(\ad re_{\alpha})\in \Aut(\UU_k),$$
so that $s_{\alpha}^*=s_{\alpha}^*(1)$ (cf. \S\ref{subsection:IEA}). 
For any $i\in I$ and $r\in k^{\times}$, any homogeneous $x\in \nn^+_k(A)$ of degree $\alpha\neq\alpha_i$ and any choice of exponential sequence $(x^{[n]})_{n\in\NN}$ for $x$, we deduce from (\ref{eqn:ei_action}) that
\begin{equation*}
\begin{aligned}
\widehat{\pi}\big(\widetilde{s}_i(r)([\exp]x)\widetilde{s}_i(r)\inv\big)&=\widehat{\pi}\bigg(\sum_{n\geq 0}{s_{\alpha_i}^*(r)x^{[n]}}\bigg)=\sum_{n\geq 0}{\pi_1\big(s_{\alpha_i}^*(r)x^{[n]}\big)}\\
&=\sum_{n\geq 0}{s_{\beta_i}^*(r)\pi_1\big(x^{[n]}\big)} = s_{\beta_i}^*(r)\big(\widehat{\pi}([\exp]x)\big).
\end{aligned}
\end{equation*}
In particular, $\widetilde{s}_i(r)K \thinspace \widetilde{s}_i(r)\inv\subseteq K$ for any $i\in I$ and $r\in k^{\times}$.
Since the torus $\T(k)$ is generated by $$\big\{\widetilde{s}_{i}^{\thinspace -1} \widetilde{s}_{i}(r) \ | \ i\in I, \ r\in k^{\times}\big\}$$ (see \S\ref{subsection:MKG}), we deduce that $\N_A(k)\subseteq\G^{pma}_A(k)$ normalises $K$. As $\G^{pma}_A(k)$ is generated by $\U_A^{ma+}(k)$ and $\N_A(k)$, we conclude that $K$ is a normal subgroup of $G^{pma}_A(k)$, proving (1).

We now turn to the proof of (2).
Let $X\in\{A,B\}$, and let $I_X$ denote the indexing set of $X$. 
Given $w\in W(X)$ and a reduced decomposition $w=s_{i_1}s_{i_2}\dots s_{i_k}$ for $w$, we write $$w^*:=s_{i_1}^*s_{i_2}^*\dots s_{i_k}^*\in W^*\quad\textrm{and}\quad \widetilde{w}:=\widetilde{s}_{i_1}\widetilde{s}_{i_2}\dots \widetilde{s}_{i_k}\in\N_X(k)\subseteq \Stt_X(k).$$
We recall that $w^*$ depends only on $w$. Similarly, the coset $\widetilde{w}\thinspace\T_X(k)$ is uniquely determined by $w$.
The relations (\ref{R2}) and (\ref{R4}) in $\G_X(k)$ respectively imply that
\begin{equation}\label{eqn:A1}
\widetilde{w}\cdot t\cdot \widetilde{w}\inv=w(t)\quad\textrm{for any $t\in\T_X(k)$}
\end{equation}
and
\begin{equation}\label{eqn:A2}
\widetilde{w}\cdot u\cdot \widetilde{w}\inv=w^*(u)\quad\textrm{for any $u\in\Stt_X(k)$}.
\end{equation}
Moreover, in view of the relations (\ref{R3}), the torus $\T_X(k)$ is generated by the elements
\begin{equation}\label{eqn:A3}
r^{\alpha_i^{\vee}}=\widetilde{s}_i^{\thinspace -1}\widetilde{s}_i(r\inv)\quad\textrm{for all $r\in k^{\times}$ and $i\in I_X$}.
\end{equation}

For each positive real root $\gamma\in\Delta_+^{\re}(X)$, we fix some $w_{\gamma}\in W(X)$ and some $i_{\gamma}\in I_X$ such that $\gamma=w_{\gamma}\alpha_{i_{\gamma}}$ (with the choice $w_{\gamma}=1$ if $\gamma=\alpha_i$), and we choose the basis elements $e_{\gamma}\in E_{\gamma}$ and $e_{-\gamma}\in E_{-\gamma}$ so that $e_{\gamma}=w_{\gamma}^*e_{i_{\gamma}}$ and $e_{-\gamma}=w_{\gamma}^*f_{i_{\gamma}}$. To lighten the notation, we will also write $w_j:=w_{\beta_j}\in W(B)$ and $\sigma_j:=i_{\beta_j}$ for all $j\in I_A$, so that $$\beta_i=w_i\alpha_{\sigma_i}\quad\textrm{and}\quad e_{\pm\beta_i}=w_i^*e_{\pm\alpha_{\sigma_i}}\quad\textrm{for all $i\in I_A$.}$$ Defining for all $\gamma\in\Delta_+^{\re}(X)$ the reflection $$s_{\gamma}:Q(X)\to Q(X):\lambda\mapsto \lambda-\la\lambda,\gamma^{\vee}\ra \thinspace\gamma,$$ we then have $s_{\gamma}=w_{\gamma}s_{i_{\gamma}}w_{\gamma}\inv\in W(X)$. We will also view $s_{\gamma}$ as acting on the coroot lattice $Q^{\vee}(X)=\sum_{i\in I_X}{\ZZ\alpha_i^{\vee}}$ by  
$$s_{\gamma}:Q^{\vee}(X)\to Q^{\vee}(X):h\mapsto h-\la\gamma,h\ra \thinspace\gamma^{\vee}.$$

We define the map $$\widetilde{\pi}\co \T_A(k) * \bigg(\bigast_{\gamma\in\Delta^{\re}(A)}{U_{\gamma}(k)}\bigg)\to \G_B(k):x_{\pm\alpha_i}(r)\mapsto x_{\pm\beta_i}(r), \quad \left\{
\begin{array}{ll}
r^{\alpha_i^{\vee}}&\mapsto \widetilde{\pi}(\widetilde{s}_i^{\thinspace -1}\widetilde{s}_i(r\inv))\\
x_{\pm\gamma}(r)&\mapsto \widetilde{\pi}(\widetilde{w}_{\gamma}x_{\pm\alpha_{i_{\gamma}}}(r)\widetilde{w}_{\gamma}\inv)
\end{array}
\right.$$
on the free product of $\T_A(k)$ with all real root groups $U_{\gamma}(k)$, and we prove that $\widetilde{\pi}$ factors through a group homomorphism $\G_A(k)\to\G_B(k)$.
Note first that 
\begin{equation}\label{eqn:A4}
\widetilde{\pi}(\widetilde{s}_i(r))=\widetilde{\pi}(x_{\alpha_i}(r)x_{-\alpha_i}(r\inv)x_{\alpha_i}(r))=x_{\beta_i}(r)x_{-\beta_i}(r\inv)x_{\beta_i}(r)=w_i^*(\widetilde{s}_{\sigma_i}(r))
\end{equation}
for all $r\in k^{\times}$ and $i\in I_A$.
In particular, we deduce from (\ref{eqn:A2}) that
\begin{equation}\label{eqn:A5}
\widetilde{\pi}(\widetilde{s}_i)=w_i^*(\widetilde{s}_{\sigma_i})=\widetilde{w}_i\widetilde{s}_{\sigma_i}\widetilde{w}_i\inv\in \N_B(k)\quad\textrm{for all $i\in I_A$}.
\end{equation}
Hence for any $\gamma\in\Delta_+^{\re}$ and $r\in k$, we have
\begin{equation}\label{eqn:A6}
\widetilde{\pi}(x_{\pm\gamma}(r))=\widetilde{\pi}(\widetilde{w}_{\gamma}x_{\pm\alpha_{i_\gamma}}(r)\widetilde{w}_{\gamma}\inv) = \widetilde{w}^{\pi}_{\gamma}x_{\pm\beta_{i_\gamma}}(r)(\widetilde{w}^{\pi}_{\gamma})\inv=w_{\gamma}^{\pi *}(x_{\pm\beta_{i_\gamma}}(r)),
\end{equation}
where $$\widetilde{w}^{\pi}_{\gamma}:=w_{i_1}^*(\widetilde{s}_{\sigma_{i_1}})\dots w_{i_k}^*(\widetilde{s}_{\sigma_{i_k}})\in \N_B(k)\quad\textrm{and}\quad w_{\gamma}^{\pi *}:=s_{\beta_{i_1}}^*\dots s_{\beta_{i_k}}^*\in W^*(B)$$ for some prescribed reduced decomposition $w_{\gamma}=s_{i_1}\dots s_{i_k}$ of $w_{\gamma}\in W(A)$.
Finally, using (\ref{eqn:A1}), (\ref{eqn:A2}), (\ref{eqn:A3}) and (\ref{eqn:A4}), we see that the restriction of $\widetilde{\pi}$ to $\T_A(k)$ is given for all $r\in k^{\times}$ and $i\in I_A$ by 
\begin{equation}\label{eqn:A7}
\widetilde{\pi}(r^{\alpha_i^{\vee}})=\widetilde{\pi}(\widetilde{s}_i^{\thinspace -1}\widetilde{s}_i(r\inv))=w_{i}^*(\widetilde{s}_{\sigma_i}^{\thinspace -1}\widetilde{s}_{\sigma_i}(r\inv))= \widetilde{w}_{i}\cdot r^{\alpha_{\sigma_i}^{\vee}}\cdot\widetilde{w}_{i}\inv = w_{i}(r^{\alpha_{\sigma_i}^{\vee}})=r^{w_{i}\alpha_{\sigma_i}^{\vee}}=r^{\beta_i^{\vee}}.
\end{equation}

We are now ready to prove that the image by $\widetilde{\pi}$ of the relations (\ref{R0}), (\ref{R1}), (\ref{R2}), (\ref{R3}) and (\ref{R4}) defining $\G_A(k)$ are still satisfied in $\G_B(k)$. Observe first that $\widetilde{\pi}$ and $\widehat{\pi}$ coincide on $U^+_A(k)$. Indeed, this follows from (\ref{eqn:A6}) and the fact that for any $\gamma\in\Delta^{\re}_+(A)$ and any $r\in k$, 
$$\widehat{\pi}(x_{\gamma}(r))=\widehat{\pi}(\exp re_{\gamma})=\exp r\pi(e_{\gamma})=\exp r\pi_1(w_{\gamma}^*e_{i_{\gamma}})=\exp rw^{\pi *}_{\gamma}e_{\beta_{i_{\gamma}}}=w_{\gamma}^{\pi *}(x_{\beta_{i_\gamma}}(r)).$$
In particular, the image by $\widetilde{\pi}$ of the relations (\ref{R0}) are satisfied in $\G_B(k)$ for any prenilpotent pair $\{\alpha,\beta\}\subseteq \Delta^{\re}_+(A)$ of positive real roots (and hence also of negative real roots by symmetry). Let now $\{\alpha,\beta\}\subseteq \Delta^{\re}(A)$ be a prenilpotent pair of roots of opposite sign, say $\alpha\in\Delta^{\re}_+(A)$ and $\beta\in\Delta^{\re}_-(A)$. Then there exists some $w\in W$ such that $\{w\alpha,w\beta\}\subseteq \Delta^{\re}_+(A)$. Up to modifying $e_{w\alpha}$ and $e_{w\beta}$ by their opposite, we may then assume that $we_{\alpha}=e_{w\alpha}$ and $we_{\beta}=e_{w\beta}$ (note that $\{\alpha,\beta\}\neq\{w\alpha,w\beta\}\subseteq\Delta_+^{\re}(A)$). Hence $ww_{\alpha}e_{i_{\alpha}}=e_{w\alpha}$ and we may thus assume, up to modifying $w_{w\alpha}$, that $w_{w\alpha}w_{\alpha}\inv=w$. Set $$w^{\pi *}:=w_{w\alpha}^{\pi *}(w_{\alpha}^{\pi *})\inv.$$ Consider the relation
$$[x_{\alpha}(r),x_{\beta}(s)]=\prod_{\gamma}{x_{\gamma}(C^{\alpha\beta}_{ij}r^is^j)}$$ in $\G_A(k)$ for some $r,s\in k$, where $\gamma=i\alpha+j\beta$ runs, as in (\ref{R0}), through the interval $]\alpha,\beta[_{\NN}$. For each $\gamma\in ]\alpha,\beta[_{\NN}$, let $\epsilon_{\gamma}\in\{\pm 1\}$ be such that $e_{w\gamma}=\epsilon_{\gamma}w^*e_{\gamma}$. Note that $w\big(]\alpha,\beta[_{\NN}\big)=]w\alpha,w\beta[_{\NN}$. We have
$$[x_{w\alpha}(r),x_{w\beta}(s)]=w^*([x_{\alpha}(r),x_{\beta}(s)])=w^*\bigg(\prod_{\gamma}{x_{\gamma}(C^{\alpha\beta}_{ij}r^is^j)}\bigg)=\prod_{\gamma}{x_{w\gamma}(\epsilon_{\gamma}C^{\alpha\beta}_{ij}r^is^j)},$$
so that $C^{w\alpha,w\beta}_{ij}=\epsilon_{i\alpha+j\beta}C^{\alpha\beta}_{ij}$ for all $i,j$. It then follows from (\ref{eqn:A6}) that
\begin{equation*}
\begin{aligned}
\widetilde{\pi}([x_{\alpha}(r),x_{\beta}(s)])&=\widetilde{\pi}((w^*)\inv([x_{w\alpha}(r),x_{w\beta}(s)]))=(w^{\pi *})\inv\widetilde{\pi}([x_{w\alpha}(r),x_{w\beta}(s)])\\
&=(w^{\pi *})\inv\widehat{\pi}([x_{w\alpha}(r),x_{w\beta}(s)])=(w^{\pi *})\inv\widehat{\pi}\bigg(\prod_{\gamma}{x_{w\gamma}(C^{w\alpha,w\beta}_{ij}r^is^j)}\bigg)\\
&=(w^{\pi *})\inv\widetilde{\pi}\bigg(\prod_{\gamma}{x_{w\gamma}(\epsilon_{\gamma}C^{\alpha\beta}_{ij}r^is^j)}\bigg)=(w^{\pi *})\inv\widetilde{\pi}w^*\bigg(\prod_{\gamma}{x_{\gamma}(C^{\alpha\beta}_{ij}r^is^j)}\bigg)\\
&=\widetilde{\pi}\bigg(\prod_{\gamma}{x_{\gamma}(C^{\alpha\beta}_{ij}r^is^j)}\bigg),
\end{aligned}
\end{equation*}
so that the relations (\ref{R0}) are indeed satisfied.

We next check (\ref{R1}). Let $t=r^{\alpha_j^{\vee}}\in\T_A(k)$ for some $r\in k^{\times}$ and some $j\in I_A$, and let $s\in k$ and $i\in I_A$. We then deduce from (\ref{eqn:A7}) and the relations (\ref{R1}), (\ref{eqn:A1}) and (\ref{eqn:A2}) in $\G_B(k)$ that
\begin{equation*}
\begin{aligned}
\widetilde{\pi}(t\cdot x_{\alpha_i}(s)\cdot t\inv)&=r^{\beta_j^{\vee}} w_i^*(x_{\alpha_{\sigma_i}}(s)) r^{-\beta_j^{\vee}}=
w_i^*\big(r^{w_i\inv\beta_j^{\vee}} x_{\alpha_{\sigma_i}}(s) r^{-w_i\inv\beta_j^{\vee}}\big)\\
&=w_i^*\big(x_{\alpha_{\sigma_i}}(r^{\la w_i\alpha_{\sigma_i},\beta_j^{\vee}\ra}s)\big)=\widetilde{\pi}\big(x_{\alpha_i}(r^{\la \beta_i,\beta_j^{\vee}\ra}s)\big)\\
&= \widetilde{\pi}(x_{\alpha_i}(t(\alpha_i)s)).
\end{aligned}
\end{equation*}

To check (\ref{R2}), let again $t=r^{\alpha_j^{\vee}}\in\T_A(k)$ for some $r\in k^{\times}$ and some $j\in I_A$, and let $i\in I_A$. We then deduce from (\ref{eqn:A5}), (\ref{eqn:A7}) and the relations (\ref{R2}) in $\G_B(k)$ that
\begin{equation*}
\begin{aligned}
\widetilde{\pi}(\widetilde{s}_it\widetilde{s}_i^{\thinspace -1})&=w_i^*\big(\widetilde{s}_{\sigma_i}r^{w_i\inv\beta_j^{\vee}}\widetilde{s}_{\sigma_i}^{\thinspace -1}\big)=w_i^*\big(r^{s_{\sigma_i}w_i\inv\beta_j^{\vee}}\big)\\
&=r^{w_is_{\sigma_i}w_i\inv\beta_j^{\vee}}=r^{s_{\beta_i}\beta_j^{\vee}}=r^{\beta_j^{\vee}-\la \beta_i,\beta_j^{\vee}\ra \beta_i^{\vee}}\\
&=\widetilde{\pi}\big(r^{\alpha_j^{\vee}-\la \beta_i,\beta_j^{\vee}\ra\alpha_i^{\vee}}\big)=\widetilde{\pi}\big(r^{s_i\alpha_j^{\vee}}\big)\\
&=\widetilde{\pi}(s_i(t)).
\end{aligned}
\end{equation*}

Since (\ref{R3}) and (\ref{R4}) are an immediate consequence of the definition of $\widetilde{\pi}$, we conclude that $\widetilde{\pi}$ factors through a group homomorphism $$\widetilde{\pi}\co\G_A(k)\to\G_B(k),$$ which is continuous because it coincides with the continuous group homomorphism $\widehat{\pi}$ on $U_A^+(k)$. In particular, it extends to a continuous group homomorphism $\overline{\widetilde{\pi}}\co\overline{\G_A}(k)\to\overline{\G_B}(k)$ coinciding with $\widehat{\pi}$ on $\overline{U_A^+}(k)$.
It thus remains to show that $\ker\widetilde{\pi}\subseteq {\mathcal Z}_A(k)$ and $\ker \overline{\widetilde{\pi}}\subseteq Z'_A\cap \overline{\G_A}(k)={\mathcal Z}_A(k)\cdot (Z'_A\cap \overline{U^+_A}(k))$.

Note that $\widetilde{\pi}(U_A^+(k))=\widehat{\pi}(U_A^+(k))\subseteq U_B^+(k)$. Similarly, (\ref{eqn:A5}) and (\ref{eqn:A7}) respectively imply that
$$\widetilde{\pi}(\N_A(k))\subseteq \N_B(k)\quad\textrm{and}\quad \widetilde{\pi}(\T_A(k))\subseteq \T_B(k).$$
Let $g\in \ker\widetilde{\pi}$. The Bruhat decomposition $$\G_A(k)=\dot{\bigcup}_{w\in W(A)}{\B^+(k)\widetilde{w}\B^+(k)}$$ for $\G_A(k)$ implies that $g=b_1\widetilde{w}b_2$ for some $w\in W(A)$ and some $b_1,b_2\in \B^+(k)$. Hence $$\widetilde{\pi}(g)=\widetilde{\pi}(b_1)\widetilde{\pi}(\widetilde{w})\widetilde{\pi}(b_2)=1,$$ so that the Bruhat decomposition for $\G_B(k)$ implies that $\widetilde{\pi}(\widetilde{w})=1$. We claim that for any reduced decomposition $w=s_{i_1}\dots s_{i_k}$ with $k\geq 1$, the element $w^{\pi}:=s_{\beta_{i_1}}\dots s_{\beta_{i_k}}\in W(B)$ is nontrivial. Indeed, for any $i\in I_A$ and $\lambda\in Q(A)$, we have
$$\overline{\pi}(s_i(\lambda))=\overline{\pi}(\lambda)-\la\lambda,\alpha_i^{\vee}\ra\beta_i=\overline{\pi}(\lambda)-\la\overline{\pi}(\lambda),\beta_i^{\vee}\ra\beta_i=s_{\beta_i}(\overline{\pi}(\lambda)).$$ In particular, $\overline{\pi}(w(\lambda))=w^{\pi}(\overline{\pi}(\lambda))$ for all $\lambda\in Q(A)$. Since $\overline{\pi}$ is injective, the claim follows.

This shows that $\widetilde{w}\in\T_A(k)$, and hence that $\ker\widetilde{\pi}\subseteq \B^+(k)$. Therefore,
$$\ker\widetilde{\pi}\subseteq \bigcap_{h\in\G_A(k)}{h\B^+(k)h\inv}={\mathcal Z}_A(k).$$
The same argument (using the Bruhat decompositions in $\G^{pma}_A(k)$ and $\G^{pma}_B(k)$) yields $\ker \overline{\widetilde{\pi}}\subseteq Z'_A$, as desired. This concludes the proof of the theorem.
\end{proof}

\begin{remark}\label{remark:passage_quotient}
Note that the map $\widetilde{\pi}\co \G_A(k)\to\G_B(k)$ provided by Theorem~\ref{thm:injective_standard_maps} maps ${\mathcal Z}_A(k)$ into ${\mathcal Z}_B(k)$. Indeed, recall from (\ref{eqn:center}) that ${\mathcal Z}_A(k)=\{t\in \T_A(k) \ | \ t(\alpha_j)=1 \ \forall j\in I\}$ (and similarly for ${\mathcal Z}_B(k)$). Hence, if we write $t\in \T_A(k)$ as a product $t=\prod_{i\in I}r_i^{\alpha_i^{\vee}}$ for some $r_i\in k^{\times}$, then $t\in {\mathcal Z}_A(k)$ if and only if $\prod_{i\in I}r_i^{\la\alpha_j,\alpha_i^{\vee}\ra}=1$ for all $j\in I$ (and similarly for $t\in \T_B(k)$, with $\alpha_i$ replaced by $\beta_i$). Since $\la\alpha_j,\alpha_i^{\vee}\ra=\la\beta_j,\beta_i^{\vee}\ra$ for all $i,j\in I$, the claim then follows from (\ref{eqn:A7}).

In particular, $\widetilde{\pi}$ induces a continuous injective group homomorphism $$\G_A(k)/{\mathcal Z}_A(k)\to\G_B(k)/{\mathcal Z}_B(k).$$
\end{remark}

\begin{example}\label{example:funny}
Let $k$ be a field and let $a\in\NN$ with $a\geq 2$. We define recursively the sequence $(a_n)_{n\in\NN}$ by $a_0:=a$ and $a_{n+1}:=a_n(a_n^2-3)$. For each $n\in\NN$, consider the GCM $A_n=(\begin{smallmatrix}2 & -a_n\\ -a_n & 2\end{smallmatrix})$. 
By Theorem~\ref{thm:construction_pi}, the assignment $e_i\mapsto e_i$, $i=1,2$, defines surjective group homomorphisms $$\pi_n\co \U_{A_{n+1}}^{ma+}(k)\to \U_{A_n}^{ma+}(k).$$ Similarly, by Theorem~\ref{thm:injective_standard_maps}, the assignment $e_i\mapsto e_{\beta_i}$, $i=1,2$, where $\beta_1=s_1\alpha_2$ and $\beta_2=s_2\alpha_1$, defines group homomorphisms $$\iota_n\co \U_{A_{n+1}}^{ma+}(k)\to \U_{A_n}^{ma+}(k),$$ which are moreover injective if the corresponding Kac--Moody groups are GK-simple. Indeed, this follows from the fact that $$\beta_1(\beta_2^{\vee})=\la s_1\alpha_2,s_2\alpha_1^{\vee}\ra=\la a_n\alpha_1+\alpha_2,\alpha_1^{\vee}+a_n\alpha_2^{\vee}\ra=3a_n-a_n^3=-a_{n+1},$$
and similarly for $\beta_2(\beta_1^{\vee})$. Thus, we get two projective systems
\begin{equation*}
\begin{aligned}
&\dots\stackrel{\pi_{n+1}}{\longrightarrow} \U_{A_{n+1}}^{ma+}(k)\stackrel{\pi_n}{\longrightarrow} \U_{A_n}^{ma+}(k)\stackrel{\pi_{n-1}}{\longrightarrow}\dots \stackrel{\pi_1}{\longrightarrow}\U_{A_1}^{ma+}(k)\stackrel{\pi_0}{\longrightarrow}\U_{A_0}^{ma+}(k)\\
&\dots\stackrel{\iota_{n+1}}{\longrightarrow} \U_{A_{n+1}}^{ma+}(k)\stackrel{\iota_n}{\longrightarrow} \U_{A_n}^{ma+}(k)\stackrel{\iota_{n-1}}{\longrightarrow}\dots \stackrel{\iota_1}{\longrightarrow}\U_{A_1}^{ma+}(k)\stackrel{\iota_0}{\longrightarrow}\U_{A_0}^{ma+}(k).
\end{aligned}
\end{equation*}
The projective limit of the first system should be, in some sense to be made precise, the group $\U_{A_{\infty}}^{ma+}(k)$ associated to the matrix $A_{\infty}=(\begin{smallmatrix}2 & -\infty\\ -\infty & 2\end{smallmatrix})$ and with corresponding Lie algebra $\nn^+(A_{\infty})=\tilde{\nn}^+$ freely generated by $e_1,e_2$ (see also \cite[Remark on page 55]{KM95}). The projective limit of the second system is trivial. 
\end{example}

\begin{remark}
If $B$ is a GCM of affine type, then every subsystem $\{\beta_i \ | \ i\in I\}\subseteq\Delta(B)$ as in Theorem~\ref{thm:injective_standard_maps} yields a GCM $A=(\beta_i(\beta_j^{\vee}))_{i,j\in I}$ all whose factors are of finite or affine type. For instance, the case $a=2$ in Example~\ref{example:funny} together with Theorem~\ref{thm:injective_standard_maps} show that for the affine matrix $B=(\begin{smallmatrix}2 & -2\\ -2 & 2\end{smallmatrix})$, the Kac--Moody group $\G_B(k)/{\mathcal Z}_B(k)$ embeds properly into itself. Note that, at the algebraic level, $\G_B(k)=\SL_2(k[t,t\inv])$ and the maps $k[t,t\inv]\to k[t,t\inv]:t\mapsto t^m$ ($m\geq 2$) provide examples of such embeddings.

By constrast, as soon as $B$ is of indefinite type, Example~\ref{example:funny} shows that there exist GCM $A=(\begin{smallmatrix}2 & -m\\ -n & 2\end{smallmatrix})$ with $m,n$ arbitrarily large such that $\G_A(k)/{\mathcal Z}_A(k)$ embeds into $\G_B(k)/{\mathcal Z}_B(k)$.
\end{remark}

\subsection{Simply laced covers}\label{subsection:SLC}

A GCM $A$ is called \emph{simply laced} if every off-diagonal entry of $A$ is either $0$ or $-1$. Equivalently, $A$ is simply laced if its Dynkin diagram $D(A)$ is a graph with only simple (unoriented, unlabelled) edges (see \cite[\S 4.7]{Kac}).

Let $A=(a_{ij})_{i,j\in I}$ be a symmetrisable GCM. A \emph{simply laced cover} of $A$ is a simply-laced GCM $B$ whose Dynkin diagram $D(B)$ has $n_i$ vertices $\alpha_{(i,1)},\dots,\alpha_{(i,n_i)}$ for each simple root $\alpha_i\in\Delta(A)$ (where the $n_i$ are some positive integers), and such that each $\alpha_{(i,r)}$ is connected in $D(B)$ to exactly $|a_{ji}|$ of the vertices $\alpha_{(j,1)},\dots,\alpha_{(j,n_j)}$ for $j\neq i$, and to none of the other vertices $\alpha_{(i,s)}$.
Such simply laced covers $B$ of $A$ always exist, but are in general non-unique (if one restricts to those of minimal rank). For more details about simply laced covers, we refer to \cite[\S 2.4]{HKL15}. 

Given a simply laced cover $B$ of $A$ as above, we write the indexing set $J$ of $B$ as the set of couples $$J=\{(i,j) \ | \ i\in I, \ 1\leq j\leq n_i\}.$$
In particular, we denote by $e_{(i,j)}$ and $e_{-(i,j)}:=f_{(i,j)}$ the Chevalley generators of $\g_B$, by $s_{(i,j)}$ the simple reflections generating $W(B)$, and so on. For a field $k$, and elements $i\in I$ and $r\in k$, we also set for short
$$x_{\pm(i,\cdot)}(r):=\prod_{j=1}^{n_i}{x_{\pm\alpha_{(i,j)}}(r)}\in \G_B(k), \quad \widetilde{s}_{(i,\cdot)}(r):=x_{(i,\cdot)}(r)x_{-(i,\cdot)}(r\inv)x_{(i,\cdot)}(r)=\prod_{j=1}^{n_i}{\widetilde{s}_{(i,j)}(r)}\in \G_B(k),$$
as well as $$s_{(i,\cdot)}:=\prod_{j=1}^{n_i}{s_{(i,j)}}\in W(B), \quad \widetilde{s}_{(i,\cdot)}:=\widetilde{s}_{(i,\cdot)}(1)\in\N_B(k) \quad\textrm{and}\quad s_{(i,\cdot)}^*:=\prod_{j=1}^{n_i}{s^*_{(i,j)}}\in W^*(B).$$
Note that each of the above four products (indexed by $j$) consists of pairwise commuting factors. For $i\in I$, we also set
$$e_{\pm(i,\cdot)}:=\sum_{j=1}^{n_i}{e_{\pm(i,j)}}\in\g_B, \quad \alpha_{(i,\cdot)}:=\sum_{j=1}^{n_i}{\alpha_{(i,j)}}\in Q(B)\quad\textrm{and}\quad \alpha_{(i,\cdot)}^{\vee}:=\sum_{j=1}^{n_i}{\alpha_{(i,j)}^{\vee}}\in Q^{\vee}(B).$$ Then for all $i,j\in I$ and $m\in\{1,\dots,n_j\}$,
$$\la \alpha_{(j,m)},\alpha_{(i,\cdot)}^{\vee}\ra=a_{ij}.$$

The following lemma is extracted from \cite[\S 2.4]{HKL15}; we give here a more detailed proof.
\begin{lemma}\label{lemma:SLC}
Let $A=(a_{ij})_{i,j\in I}$ be a symmetrisable GCM, and let $B$ be a simply laced cover of $A$ as above. Then the assignment $e_{\pm\alpha_i}\mapsto e_{\pm(i,\cdot)}$ for $i\in I$ defines an injective Lie algebra morphism $\pi\co\g_A\to\g_B$.
\end{lemma}
\begin{proof}
We proceed as in the proof of Lemma~\ref{lemma:subalgebras}. Let $\widetilde{\pi}$ be the Lie algebra morphism from the free complex Lie algebra on the generators $\{e_{\pm\alpha_i} \ | i\in I\}$ to $\g_B$ defined by the assignment $e_{\pm\alpha_i}\mapsto e_{\pm(i,\cdot)}$ for $i\in I$.
Since $\alpha_{(j,\cdot)}-\alpha_{(i,\cdot)}\notin\Delta(B)$ for all $i,j\in I$, we have $$[e_{(i,\cdot)},e_{-(j,\cdot)}]=0\quad\textrm{for all $i,j\in I$ with $i\neq j$.}$$
Similarly, $$(\ad e_{\pm(i,\cdot)})^{|a_{ij}|+1}e_{\pm(j,\cdot)}=\sum_{r_1+\dots+r_{n_i}=|a_{ij}|+1}\binom{|a_{ij}|+1}{r_1,\dots,r_{n_i}}(\ad e_{\pm(i,1)})^{r_1}\dots (\ad e_{\pm (i,n_i)})^{r_{n_i}}e_{\pm(j,\cdot)}=0$$
for all $i,j\in I$ with $i\neq j$. Indeed, each homogeneous component of $(\ad e_{\pm(i,\cdot)})^{|a_{ij}|+1}e_{\pm(j,\cdot)}$ has degree of the form $\alpha:=\pm(\alpha_{j,m}+\sum_{s=1}^{n_i}r_s\alpha_{(i,s)})$ for some $m\in\{1,\dots,n_j\}$ and some $r_s\in\NN$ with $\sum_{s=1}^{n_i}r_s=|a_{ij}|+1$. On the other hand, since
$$s_{(i,\cdot)}\alpha_{(i,s)}=-\alpha_{(i,s)}\quad\textrm{and}\quad s_{(i,\cdot)}\alpha_{(j,m)}=\alpha_{(j,m)}-\sum_{s=1}^{n_i}\la \alpha_{(j,m)},\alpha_{(i,s)}^{\vee}\ra \alpha_{(i,s)}=\alpha_{(j,m)}+\alpha_{[i]}$$
for some $\alpha_{[i]}\in\sum_{s=1}^{n_i}\NN\alpha_{(i,s)}$ of height $|a_{ij}|$, we have
$$s_{(i,\cdot)}\alpha=\pm(\alpha_{j,m}+\alpha_{[i]}-\sum_{s=1}^{n_i}r_s\alpha_{(i,s)})=\pm(\alpha_{j,m}+\alpha'_{[i]})$$
for some $\alpha'_{[i]}\in\sum_{s=1}^{n_i}\ZZ\alpha_{(i,s)}$ of height $-1$. Hence $s_{(i,\cdot)}\alpha$ (and thus also $\alpha$) cannot be a root, yielding the claim.

Finally, the elements $\alpha_{(i,\cdot)}^{\vee}=[e_{-(i,\cdot)},e_{(i,\cdot)}]$ of $\g_{B}$ ($i\in I$) satisfy
$$[\alpha_{(i,\cdot)}^{\vee},\alpha_{(j,\cdot)}^{\vee}]=0 \quad\textrm{and}\quad [\alpha_{(i,\cdot)}^{\vee},e_{\pm(j,\cdot)}]=\sum_{m=1}^{n_j}[\alpha_{(i,\cdot)}^{\vee},e_{\pm(j,m)}]=\sum_{m=1}^{n_j}\pm a_{ij}e_{\pm(j,m)}=\pm a_{ij}e_{\pm(j,\cdot)}$$
for all $i,j\in I$. Hence all the defining relations of $\g_A=[\g(A),\g(A)]$ (see \S\ref{subsection:KMA}) lie in the kernel of $\widetilde{\pi}$, so that $\widetilde{\pi}$ factors through a Lie algebra morphism $\pi\co\g_A\to\g_B$.

For the injectivity, note that $\ker\pi$ intersects the Cartan subalgebra of $\g_A$ trivially. Hence $\ker\pi=\{0\}$ by the Gabber--Kac theorem (see \S\ref{subsection:GKS}), as desired.
\end{proof}

The proof of the following theorem follows the lines of the proof of Theorems~\ref{thm:basic_functoriality} and \ref{thm:injective_standard_maps}. We prefer, however, to repeat the arguments, as a common treatment of these results would necessitate very cumbersome notation. 

\begin{theorem}\label{thm:SLC}
Let $k$ be a field and $A=(a_{ij})_{i,j\in I}$ be a symmetrisable GCM. Let $B$ be a simply laced cover of $A$, and let $\pi\co \g_A\to\g_B$ be the embedding provided by Lemma~\ref{lemma:SLC}. Then the following holds:
\begin{enumerate}
\item
There exists a continuous group homomorphism $\widehat{\pi}\co\U_A^{ma+}(k)\to\U_B^{ma+}(k)$ such that for all $r\in k$, $i\in I$ and $\gamma\in\Delta_+^{\re}(A)\setminus\{\alpha_i\}$,
$$\widehat{\pi}(x_{\alpha_i}(r))=x_{(i,\cdot)}(r)\quad\textrm{and}\quad \widehat{\pi}(\widetilde{s}_i\cdot x_{\gamma}(r)\cdot\widetilde{s}_i^{\thinspace -1})=\widetilde{s}_{(i,\cdot)}\cdot\widehat{\pi}(x_{\gamma}(r))\cdot\widetilde{s}_{(i,\cdot)}^{\thinspace -1}.$$
\item
The restriction of $\widehat{\pi}$ to $U^+_A(k)$ extends to continuous group homomorphisms
$$\G_A(k)\to\G_B(k)\quad\textrm{and}\quad \overline{\G_A}(k)\to\overline{\G_B}(k)$$ with kernels respectively contained in ${\mathcal Z}_A(k)$ and ${\mathcal Z}_A(k)\cdot (Z'_A\cap \overline{U^+_A}(k))$. Here, we view $\G_A(k)$ and $\G_B(k)$ as subgroups of $\G^{pma}_A(k)$ and $\G^{pma}_B(k)$ respectively, with the induced topology.
\end{enumerate}
\end{theorem}
\begin{proof}
For $i\in I$ and a multi-index $\mm=(m_1,\dots,m_{n_i})\in\NN^{n_i}$, we write
$$| \negthinspace \mm \negthinspace |:=\sum_{j=1}^{n_i}{m_j}\quad\textrm{and}\quad e_{\pm(i,\cdot)}^{(\mm)}:=\prod_{j=1}^{n_i}{e_{\pm(i,j)}^{(m_j)}}\in\UU(B).$$
Note that the $e_{\pm(i,j)}$ pairwise commute (for $i$ fixed). Since for any $i\in I$ and $n\in\NN$, 
$$e_{\pm(i,\cdot)}^{(n)}=\bigg(\sum_{j=1}^{n_i}{e_{\pm(i,j)}}\bigg)^{(n)}=\sum_{| \negthinspace \mm \negthinspace |=n}{e_{\pm(i,\cdot)}^{(\mm)}}\in\UU(B)$$
and since $\pi(\sum_{i\in I}{\ZZ\alpha_i^{\vee}})\subseteq \sum_{i\in I}\sum_{j=1}^{n_i}{\ZZ\alpha_{(i,j)}^{\vee}}$,
the map $\UU_{\CC}(\g_A)\to \UU_{\CC}(\g_B)$ lifting $\pi$ at the level of the corresponding enveloping algebras restricts to an algebra morphism
$$\pi_1\co \UU(A)\to \UU(B).$$
Moreover, for any $i\in I$ and $n\in\NN$,
\begin{equation*}
\begin{aligned}
\nabla_B\pi_1\big(e_i^{(n)}\big)&=\nabla_Be_{(i,\cdot)}^{(n)}=\sum_{| \negthinspace \mm \negthinspace |=n}{\nabla_Be_{(i,\cdot)}^{(\mm)}}=\sum_{r+s=n}{\sum_{| \negthinspace \rr \negthinspace |=r}{\sum_{| \negthinspace \s \negthinspace |=s}{e_{(i,\cdot)}^{(\rr)}\otimes e_{(i,\cdot)}^{(\s)}}}}=\sum_{r+s=n}{e_{(i,\cdot)}^{(r)}\otimes e_{(i,\cdot)}^{(s)}}\\ &=(\pi_1\otimes\pi_1)\nabla_A e_i^{(n)}.
\end{aligned}
\end{equation*}
Since clearly $\epsilon_B\pi_1=\epsilon_A$, we deduce that the restriction of $\pi_1$ to $\UU^+(A)$ is a bialgebra morphism.

Note also that $\pi_1$ preserves the $\NN$-gradations on $\UU^+(A)$ and $\UU^+(B)$ induced by $\height\co Q_+\to\NN$.
In particular, the map $$\UU^+(A)\otimes_{\ZZ}k\to\UU^+(B)\otimes_{\ZZ}k$$ obtained from $\pi_1$ by extension of scalars can be further extended 
to a bialgebra morphism 
$$\pi_2\co \widehat{\UU}_k^+(A)\to \widehat{\UU}_k^+(B)$$ between the corresponding completions.
Finally, since $\pi_2$ preserves the group-like elements of constant term $1$, it restricts to a group homomorphism
$$\widehat{\pi}\co \U_A^{ma+}(k)\to \U_B^{ma+}(k)$$
by Proposition~\ref{prop formal sum Rousseau}(1), which is moreover continuous because $\pi_2$ preserves the $\NN$-gradations on $\widehat{\UU}_k^+(A)$ and $\widehat{\UU}_k^+(B)$.

Let now $i\in I$ and $r\in k$. By definition, 
$$\widehat{\pi}(x_{\alpha_i}(r))=\pi_2\bigg(\sum_{n\geq 0}{r^ne_i^{(n)}}\bigg)=\sum_{n\geq 0}{r^ne_{(i,\cdot)}^{(n)}}=\sum_{n\geq 0}\sum_{| \negthinspace \mm \negthinspace |=n}{r^ne_{(i,\cdot)}^{(\mm)}}=\prod_{j=1}^{n_i}{\exp re_{(i,j)}}=x_{(i,\cdot)}(r).$$
Moreover, since for any $u\in\UU(A)$,
\begin{equation*}
\begin{aligned}
\pi_1\big((\exp \ad e_{\pm\alpha_i})(u)\big)&=\pi_1\bigg(\sum_{n\geq 0}\sum_{r+s=n}{(-1)^re_{\pm\alpha_i}^{(r)}ue_{\pm\alpha_i}^{(s)}}\bigg)=\sum_{n\geq 0}\sum_{r+s=n}\sum_{| \negthinspace \rr \negthinspace |=r}\sum_{| \negthinspace \s \negthinspace |=s}{(-1)^re_{\pm(i,\cdot)}^{(\rr)}\pi_1(u)e_{\pm(i,\cdot)}^{(\s)}}\\
&=\bigg(\prod_{j=1}^{n_i}{\exp \ad e_{\pm(i,j)}}\bigg)(\pi_1(u)),
\end{aligned}
\end{equation*}
so that
\begin{equation*}
\pi_1(s_i^*u)=s_{(i,\cdot)}^*\pi_1(u),
\end{equation*}
we deduce from the relations (\ref{R4}) that for any $\gamma\in\Delta_+^{\re}(A)\setminus\{\alpha_i\}$,  
\begin{equation*}
\begin{aligned}
\widehat{\pi}(\widetilde{s}_i\cdot x_{\gamma}(r)\cdot\widetilde{s}_i^{\thinspace -1})&=\widehat{\pi}\bigg(\sum_{n\geq 0}{r^ns_i^*e_{\gamma}^{(n)}}\bigg)=\sum_{n\geq 0}{r^n\pi_1\big(s_i^*e_{\gamma}^{(n)}\big)}=\sum_{n\geq 0}{r^ns_{(i,\cdot)}^*\pi_1\big(e_{\gamma}^{(n)}\big)}\\
&=s_{(i,\cdot)}^*\widehat{\pi}(x_{\gamma}(r))=\widetilde{s}_{(i,\cdot)}\cdot\widehat{\pi}(x_{\gamma}(r))\cdot\widetilde{s}_{(i,\cdot)}^{\thinspace -1}.
\end{aligned}
\end{equation*}
This concludes the proof of (1).

We now turn to the proof of (2).
Let $X\in\{A,B\}$, and let $I_X$ denote the indexing set of $X$. 
Given $w\in W(X)$ and a reduced decomposition $w=s_{i_1}s_{i_2}\dots s_{i_k}$ for $w$, we write 
\begin{equation}\label{eqn:wstar_wtilde}
w^*:=s_{i_1}^*s_{i_2}^*\dots s_{i_k}^*\in W^*\quad\textrm{and}\quad \widetilde{w}:=\widetilde{s}_{i_1}\widetilde{s}_{i_2}\dots \widetilde{s}_{i_k}\in\N_X(k)\subseteq \Stt_X(k).
\end{equation}

For each positive real root $\gamma\in\Delta_+^{\re}(X)$, we fix some $w_{\gamma}\in W(X)$ and some $i_{\gamma}\in I_X$ such that $\gamma=w_{\gamma}\alpha_{i_{\gamma}}$ (with the choice $w_{\gamma}=1$ if $\gamma=\alpha_i$), and we choose the basis elements $e_{\gamma}\in E_{\gamma}$ and $e_{-\gamma}\in E_{-\gamma}$ so that $e_{\gamma}=w_{\gamma}^*e_{i_{\gamma}}$ and $e_{-\gamma}=w_{\gamma}^*f_{i_{\gamma}}$. 

We define the map $$\widetilde{\pi}\co \T_A(k) * \bigg(\bigast_{\gamma\in\Delta^{\re}(A)}{U_{\gamma}(k)}\bigg)\to \G_B(k):x_{\pm\alpha_i}(r)\mapsto x_{\pm(i,\cdot)}(r), \quad \left\{
\begin{array}{ll}
r^{\alpha_i^{\vee}}&\mapsto \widetilde{\pi}(\widetilde{s}_i^{\thinspace -1}\widetilde{s}_i(r\inv))\\
x_{\pm\gamma}(r)&\mapsto \widetilde{\pi}(\widetilde{w}_{\gamma}x_{\pm\alpha_{i_{\gamma}}}(r)\widetilde{w}_{\gamma}\inv)
\end{array}
\right.$$
on the free product of $\T_A(k)$ with all real root groups $U_{\gamma}(k)$, and we prove that $\widetilde{\pi}$ factors through a group homomorphism $\G_A(k)\to\G_B(k)$.
Note first that 
\begin{equation}\label{eqn:A11}
\widetilde{\pi}(\widetilde{s}_i(r))=\widetilde{\pi}(x_{\alpha_i}(r)x_{-\alpha_i}(r\inv)x_{\alpha_i}(r))=x_{(i,\cdot)}(r)x_{-(i,\cdot)}(r\inv)x_{(i,\cdot)}(r)=\widetilde{s}_{(i,\cdot)}(r)
\end{equation}
for all $r\in k^{\times}$ and $i\in I_A$.
In particular, 
\begin{equation}\label{eqn:A12}
\widetilde{\pi}(\widetilde{s}_i)=\widetilde{s}_{(i,\cdot)}\in \N_B(k)\quad\textrm{for all $i\in I_A$}.
\end{equation}
Hence for any $\gamma\in\Delta_+^{\re}$ and $r\in k$, we have
\begin{equation}\label{eqn:A13}
\widetilde{\pi}(x_{\pm\gamma}(r))=\widetilde{\pi}(\widetilde{w}_{\gamma}x_{\pm\alpha_{i_\gamma}}(r)\widetilde{w}_{\gamma}\inv) = \widetilde{w}^{\pi}_{\gamma}x_{\pm(i_{\gamma},\cdot)}(r)(\widetilde{w}^{\pi}_{\gamma})\inv=w_{\gamma}^{\pi *}(x_{\pm(i_{\gamma},\cdot)}(r)),
\end{equation}
where $$\widetilde{w}^{\pi}_{\gamma}:=\widetilde{s}_{(i_1,\cdot)}\dots \widetilde{s}_{(i_k,\cdot)}\in \N_B(k)\quad\textrm{and}\quad w_{\gamma}^{\pi *}:=s_{(i_1,\cdot)}^*\dots s_{(i_k,\cdot)}^*\in W^*(B)$$ for some prescribed reduced decomposition $w_{\gamma}=s_{i_1}\dots s_{i_k}$ of $w_{\gamma}\in W(A)$.
Finally, using (\ref{eqn:A11}) and the relations (\ref{R3}) in $\G_B(k)$, we see that the restriction of $\widetilde{\pi}$ to $\T_A(k)$ is given for all $r\in k^{\times}$ and $i\in I_A$ by 
\begin{equation}\label{eqn:A14}
\widetilde{\pi}(r^{\alpha_i^{\vee}})=\widetilde{\pi}(\widetilde{s}_i^{\thinspace -1}\widetilde{s}_i(r\inv))=\widetilde{s}_{(i,\cdot)}^{\thinspace -1}\widetilde{s}_{(i,\cdot)}(r\inv)= \prod_{j=1}^{n_i}{\big(\widetilde{s}_{(i,j)}^{\thinspace -1}\widetilde{s}_{(i,j)}(r\inv)\big)}=\prod_{j=1}^{n_i}{r^{\alpha_{(i,j)}^{\vee}}}=r^{\alpha_{(i,\cdot)}^{\vee}}.
\end{equation}

We are now ready to prove that the image by $\widetilde{\pi}$ of the relations (\ref{R0}), (\ref{R1}), (\ref{R2}), (\ref{R3}) and (\ref{R4}) defining $\G_A(k)$ are still satisfied in $\G_B(k)$. Observe first that $\widetilde{\pi}$ and $\widehat{\pi}$ coincide on $U^+_A(k)$ by (\ref{eqn:A12}) and the first statement of the theorem.
In particular, the image by $\widetilde{\pi}$ of the relations (\ref{R0}) are satisfied in $\G_B(k)$ for any prenilpotent pair $\{\alpha,\beta\}\subseteq \Delta^{\re}_+(A)$ of positive real roots (and hence also of negative real roots by symmetry). Let now $\{\alpha,\beta\}\subseteq \Delta^{\re}(A)$ be a prenilpotent pair of roots of opposite sign, say $\alpha\in\Delta^{\re}_+(A)$ and $\beta\in\Delta^{\re}_-(A)$. Then there exists some $w\in W$ such that $\{w\alpha,w\beta\}\subseteq \Delta^{\re}_+(A)$. Up to modifying $e_{w\alpha}$ and $e_{w\beta}$ by their opposite, we may then assume that $we_{\alpha}=e_{w\alpha}$ and $we_{\beta}=e_{w\beta}$ (note that $\{\alpha,\beta\}\neq\{w\alpha,w\beta\}\subseteq\Delta_+^{\re}(A)$). Hence $ww_{\alpha}e_{i_{\alpha}}=e_{w\alpha}$ and we may thus assume, up to modifying $w_{w\alpha}$, that $w_{w\alpha}w_{\alpha}\inv=w$. Set $$w^{\pi *}:=w_{w\alpha}^{\pi *}(w_{\alpha}^{\pi *})\inv.$$ Consider the relation
$$[x_{\alpha}(r),x_{\beta}(s)]=\prod_{\gamma}{x_{\gamma}(C^{\alpha\beta}_{ij}r^is^j)}$$ in $\G_A(k)$ for some $r,s\in k$, where $\gamma=i\alpha+j\beta$ runs, as in (\ref{R0}), through the interval $]\alpha,\beta[_{\NN}$. For each $\gamma\in ]\alpha,\beta[_{\NN}$, let $\epsilon_{\gamma}\in\{\pm 1\}$ be such that $e_{w\gamma}=\epsilon_{\gamma}w^*e_{\gamma}$. Note that $w\big(]\alpha,\beta[_{\NN}\big)=]w\alpha,w\beta[_{\NN}$. We have
$$[x_{w\alpha}(r),x_{w\beta}(s)]=w^*([x_{\alpha}(r),x_{\beta}(s)])=w^*\bigg(\prod_{\gamma}{x_{\gamma}(C^{\alpha\beta}_{ij}r^is^j)}\bigg)=\prod_{\gamma}{x_{w\gamma}(\epsilon_{\gamma}C^{\alpha\beta}_{ij}r^is^j)},$$
so that $C^{w\alpha,w\beta}_{ij}=\epsilon_{i\alpha+j\beta}C^{\alpha\beta}_{ij}$ for all $i,j$. It then follows from (\ref{eqn:A13}) that
\begin{equation*}
\begin{aligned}
\widetilde{\pi}([x_{\alpha}(r),x_{\beta}(s)])&=\widetilde{\pi}((w^*)\inv([x_{w\alpha}(r),x_{w\beta}(s)]))=(w^{\pi *})\inv\widetilde{\pi}([x_{w\alpha}(r),x_{w\beta}(s)])\\
&=(w^{\pi *})\inv\widehat{\pi}([x_{w\alpha}(r),x_{w\beta}(s)])=(w^{\pi *})\inv\widehat{\pi}\bigg(\prod_{\gamma}{x_{w\gamma}(C^{w\alpha,w\beta}_{ij}r^is^j)}\bigg)\\
&=(w^{\pi *})\inv\widetilde{\pi}\bigg(\prod_{\gamma}{x_{w\gamma}(\epsilon_{\gamma}C^{\alpha\beta}_{ij}r^is^j)}\bigg)=(w^{\pi *})\inv\widetilde{\pi}w^*\bigg(\prod_{\gamma}{x_{\gamma}(C^{\alpha\beta}_{ij}r^is^j)}\bigg)\\
&=\widetilde{\pi}\bigg(\prod_{\gamma}{x_{\gamma}(C^{\alpha\beta}_{ij}r^is^j)}\bigg),
\end{aligned}
\end{equation*}
so that the relations (\ref{R0}) are indeed satisfied.

We next check (\ref{R1}). Let $t=r^{\alpha_j^{\vee}}\in\T_A(k)$ for some $r\in k^{\times}$ and some $j\in I_A$, and let $s\in k$ and $i\in I_A$. We then deduce from (\ref{eqn:A14}) and the relations (\ref{R1}) in $\G_B(k)$ that
\begin{equation*}
\begin{aligned}
\widetilde{\pi}(t\cdot x_{\alpha_i}(s)\cdot t\inv)&=r^{\alpha_{(j,\cdot)}^{\vee}} x_{(i,\cdot)}(s) r^{-\alpha_{(j,\cdot)}^{\vee}}= \prod_{m=1}^{n_i}{\big(r^{\alpha_{(j,\cdot)}^{\vee}} x_{(i,m)}(s) r^{-\alpha_{(j,\cdot)}^{\vee}}\big)}=\prod_{m=1}^{n_i}{x_{(i,m)}\big(r^{\la\alpha_{(i,m)},\alpha_{(j,\cdot)}^{\vee}\ra} s\big)}\\
&=\prod_{m=1}^{n_i}{x_{(i,m)}(r^{a_{ji}} s)}= x_{(i,\cdot)}(t(\alpha_i) s)\\
&=\widetilde{\pi}(x_{\alpha_i}(t(\alpha_i)s)).
\end{aligned}
\end{equation*}

To check (\ref{R2}), let again $t=r^{\alpha_j^{\vee}}\in\T_A(k)$ for some $r\in k^{\times}$ and some $j\in I_A$, and let $i\in I_A$. We then deduce from (\ref{eqn:A12}), (\ref{eqn:A14}) and the relations (\ref{R2}) in $\G_B(k)$ that
\begin{equation*}
\begin{aligned}
\widetilde{\pi}(\widetilde{s}_it\widetilde{s}_i^{\thinspace -1})&=\widetilde{s}_{(i,\cdot)}r^{\alpha_{(j,\cdot)}^{\vee}}\widetilde{s}_{(i,\cdot)}^{\thinspace -1}=s_{(i,\cdot)}(r^{\alpha_{(j,\cdot)}^{\vee}})=r^{s_{(i,\cdot)}(\alpha_{(j,\cdot)}^{\vee})}\\
&=r^{\alpha_{(j,\cdot)}^{\vee}-\sum_{m=1}^{n_i}{\la\alpha_{(i,m)},\alpha_{(j,\cdot)}^{\vee}\ra\alpha_{(i,m)}^{\vee}}}=r^{\alpha_{(j,\cdot)}^{\vee}-a_{ji}\alpha_{(i,\cdot)}^{\vee}}\\
&=\widetilde{\pi}\big(r^{\alpha_j^{\vee}-a_{ji}\alpha_i^{\vee}}\big)=\widetilde{\pi}(r^{s_i(\alpha_j^{\vee})})\\
&=\widetilde{\pi}(s_i(t)).
\end{aligned}
\end{equation*}

Since (\ref{R3}) and (\ref{R4}) are an immediate consequence of the definition of $\widetilde{\pi}$, we conclude that $\widetilde{\pi}$ factors through a group homomorphism $$\widetilde{\pi}\co\G_A(k)\to\G_B(k),$$ which is continuous because it coincides with the continuous group homomorphism $\widehat{\pi}$ on $U_A^+(k)$. In particular, it extends to a continuous group homomorphism $\overline{\widetilde{\pi}}\co\overline{\G_A}(k)\to\overline{\G_B}(k)$ coinciding with $\widehat{\pi}$ on $\overline{U_A^+}(k)$.
It thus remains to show that $\ker\widetilde{\pi}\subseteq {\mathcal Z}_A(k)$ and $\ker \overline{\widetilde{\pi}}\subseteq Z'_A\cap \overline{\G_A}(k)={\mathcal Z}_A(k)\cdot (Z'_A\cap \overline{U^+_A}(k))$.

Note that $\widetilde{\pi}(U_A^+(k))=\widehat{\pi}(U_A^+(k))\subseteq U_B^+(k)$. Similarly, (\ref{eqn:A12}) and (\ref{eqn:A14}) respectively imply that
$$\widetilde{\pi}(\N_A(k))\subseteq \N_B(k)\quad\textrm{and}\quad \widetilde{\pi}(\T_A(k))\subseteq \T_B(k).$$
Let $g\in \ker\widetilde{\pi}$. The Bruhat decomposition $$\G_A(k)=\dot{\bigcup}_{w\in W(A)}{\B^+(k)\widetilde{w}\B^+(k)}$$ for $\G_A(k)$ implies that $g=b_1\widetilde{w}b_2$ for some $w\in W(A)$ and some $b_1,b_2\in \B^+(k)$. Hence $$\widetilde{\pi}(g)=\widetilde{\pi}(b_1)\widetilde{\pi}(\widetilde{w})\widetilde{\pi}(b_2)=1,$$ so that the Bruhat decomposition for $\G_B(k)$ implies that $\widetilde{\pi}(\widetilde{w})=1$. We claim that for any reduced decomposition $w=s_{i_1}\dots s_{i_k}$ with $k\geq 1$, the element $w^{\pi}:=s_{(i_1,\cdot)}\dots s_{(i_k,\cdot)}\in W(B)$ is nontrivial. Indeed, define the $\ZZ$-linear map
$$\overline{\pi}\co Q^{\vee}(A)\to Q^{\vee}(B):\alpha_i^{\vee}\mapsto\alpha^{\vee}_{(i,\cdot)}.$$ Then for any $i,j\in I_A$, we have
$$s_{(i,\cdot)}(\overline{\pi}(\alpha_j^{\vee}))=s_{(i,\cdot)}(\alpha^{\vee}_{(j,\cdot)})=\alpha^{\vee}_{(j,\cdot)}-\sum_{m=1}^{n_i}{\la \alpha_{(i,m)},\alpha_{(j,\cdot)}^{\vee}\ra\alpha_{(i,m)}^{\vee}}=\alpha_{(j,\cdot)}^{\vee}-a_{ji}\alpha_{(i,\cdot)}^{\vee}=\overline{\pi}(s_i(\alpha_j^{\vee}))$$
and hence $\overline{\pi}(s_i(h))=s_{(i,\cdot)}(\overline{\pi}(h))$ for any $i\in I$ and $h\in Q^{\vee}(A)$. In particular, $\overline{\pi}(w(h))=w^{\pi}(\overline{\pi}(h))$ for all $h\in Q^{\vee}(A)$. Since $\overline{\pi}$ is injective, the claim follows.

This shows that $\widetilde{w}\in\T_A(k)$, and hence that $\ker\widetilde{\pi}\subseteq \B^+(k)$. Therefore,
$$\ker\widetilde{\pi}\subseteq \bigcap_{h\in\G_A(k)}{h\B^+(k)h\inv}={\mathcal Z}_A(k).$$
The same argument (using the Bruhat decompositions in $\G^{pma}_A(k)$ and $\G^{pma}_B(k)$) yields $\ker \overline{\widetilde{\pi}}\subseteq Z'_A$, as desired. This concludes the proof of the theorem.
\end{proof}

\begin{remark}
Proceeding exactly as in Remark~\ref{remark:passage_quotient}, we see that the map $\widetilde{\pi}\co \G_A(k)\to\G_B(k)$ provided by Theorem~\ref{thm:SLC} maps ${\mathcal Z}_A(k)$ into ${\mathcal Z}_B(k)$ by 
(\ref{eqn:A14}), and hence induces a continuous injective group homomorphism $$\G_A(k)/{\mathcal Z}_A(k)\to\G_B(k)/{\mathcal Z}_B(k).$$
\end{remark}

\section{Non-density and Gabber--Kac simplicity}\label{section:ND}
This section is devoted to the proof of Propositions~\ref{thmintro:nondensity} and \ref{thmintro:nonGK-simple}. 

\begin{prop}\label{proposition:non-density_induction}
Let $k$ be a field and let $B$ be a GCM. Assume that $U^+_B(k)$ is not dense in $\U_B^{ma+}(k)$. Then $U^+_A(k)$ is not dense in $\U_A^{ma+}(k)$ for all GCM $A$ such that $B\leq A$.
\end{prop}
\begin{proof}
By Corollary~\ref{corollary:functoriality_min}, the surjective map $\widehat{\pi}_{AB}\co \U_A^{ma+}(k)\to \U_B^{ma+}(k)$ provided by Theorem~\ref{thm:construction_pi} restricts to a group homomorphism $\overline{U^+_A}(k)\to\overline{U^+_B}(k)$. Thus, if $U^+_A(k)$ were dense in $\U_A^{ma+}(k)$, we would conclude that $$\overline{U^+_B}(k)\supseteq \widehat{\pi}_{AB}(\overline{U^+_A}(k))=\widehat{\pi}_{AB}(\U_A^{ma+}(k))=\U_B^{ma+}(k),$$
and hence that $\overline{U^+_B}(k)=\U_B^{ma+}(k)$, yielding the desired contradiction.
\end{proof}

\begin{lemma}\label{lemma:density_U_G}
Let $k$ be a field and $A$ be a GCM. Then $U^+_A(k)$ is dense in $\U_A^{ma+}(k)$ if and only if the minimal Kac--Moody group $\G_A(k)$ is dense in its Mathieu--Rousseau completion $\G_A^{pma}(k)$.
\end{lemma}
\begin{proof}
This follows from the fact that $\G_A^{pma}(k)$ is generated by $\U_A^{ma+}(k)$ and $\G_A(k)$ and that $U^+_A(k)=\U_A^{ma+}(k)\cap \G_A(k)$ (see \cite[3.16]{Rousseau}).
\end{proof}

The following lemma is a slight generalisation of \cite[Lemma~5.4]{simpleKM}.
\begin{lemma}\label{lemma:gen_Lemma5.4}
Let $k$ be a field, and let $A=(\begin{smallmatrix}2 & -m\\ -n & 2\end{smallmatrix})$ be a GCM such that $mn>4$. If $\charact k=2$, we moreover assume that at least one of $m$ and $n$ is odd and $\geq 3$. Then the imaginary subgroup $U^{im+}$ of $\U_A^{ma+}(k)$ is not contained in $Z'_A$.
\end{lemma}
\begin{proof}
Assume for a contradiction that $U^{im+}$ is contained in $Z'_A$. 

Note first that 
\begin{equation}
U^{im+}=\bigcap_{w\in W}\widetilde{w}\U^{ma+}_A(k)\widetilde{w}\inv,
\end{equation}
 where $\widetilde{w}$ is as in (\ref{eqn:wstar_wtilde}) (see also \cite[Definition~7.58]{KMGbook}): indeed, the inclusion $\subseteq$ readily follows from Lemma~\ref{lemma:Winv_twisted_exp} and the fact that $W$ stabilises $\Delta^{im}_+$ (see \cite[Theorem~5.4]{Kac}). Conversely, if $g\in\U^{ma+}_A(k)\setminus U^{im+}$, then by Proposition~\ref{prop formal sum Rousseau}(2) we can write $g$ as a product $g=\prod_{x\in\BBBB_{\Delta_+}}{[\exp]\lambda_xx}$ for some $\lambda_x\in k$ such that $\lambda_y\neq 0$ for some $y$ with $\deg(y)\in\Delta^{\re}_+$. In particular, by Lemma~\ref{lemma:Winv_twisted_exp}, we find some $v\in W$ such that $\widetilde{v}g\widetilde{v}\inv=x_{\alpha_i}(r)h$ for some $i\in I$, some nonzero $r\in k$, and some $h\in\U^{ma}_2(k)$. Hence $\widetilde{w}g\widetilde{w}\inv\notin \U^{ma+}_A(k)$ for $w:=s_iv\in W$, proving the reverse inclusion.
 
 As $Z'_A=Z_A\cdot (Z'_A\cap \U_A^{ma+}(k))$ and as $Z'_A\cap \U_A^{ma+}(k)$ is normal in $\G_A^{pma}(k)$ by \cite[Proposition 6.4]{Rousseau}, we deduce that $U^{im+}=Z'_A\cap \U_A^{ma+}(k)$ is a normal subgroup of $\G_A^{pma}(k)$. We now exhibit some imaginary root $\delta\in\Delta_+^{\im}$, some simple root $\alpha_i$, and some element $x\in(\nn^+_k)_{\delta}$ such that $\delta-\alpha_i\in\Delta_+^{\re}$ and such that $\ad(f_i)x$ is nonzero in $\nn^+_k$. This will show that the element $\exp(f_i)\in \U_{(-\alpha_i)}(k)\subseteq \G_A^{pma}(k)$ conjugates the element $[\exp]x\in \U_{(\delta)}(k)\subseteq U^{im+}$ outside $U^{im+}$ (see (\ref{eqn:ei_action})), yielding the desired contradiction.

Set $p=\charact k$. By hypothesis, $mn>4$. Up to interchanging $m$ and $n$, we may then assume that $n\geq 3$. If $p=2$, we may moreover assume that $n$ is odd. Set $\beta:=s_1(\alpha_2)=\alpha_2+m\alpha_1\in\Delta^{\re}_+$ and $\gamma:=s_2(\alpha_1)=\alpha_1+n\alpha_2\in\Delta^{\re}_+$, so that 
$$\la\gamma,\alpha_1^{\vee}\ra=\la\beta,\alpha_2^{\vee}\ra=2-mn, \quad \la\gamma,\alpha_2^{\vee}\ra=n,\quad \la\beta,\alpha_1^{\vee}\ra=m,\quad\textrm{and}\quad \la\gamma,\beta^{\vee}\ra=n(3-mn).$$

Assume first that $p$ does not divide $2-mn$ or that $p=0$. Set $\delta:=\alpha_1+\gamma$. Then $\delta\in\Delta_+^{\im}$ because $\delta(\alpha_1^{\vee})=4-mn<0$ and $\delta(\alpha_2^{\vee})=0$ (see \cite[Lemma~5.3]{Kac}). Set also $x:=[e_1,e_{\gamma}]\in\nn^+_{k}$. Since $\gamma-\alpha_1=n\alpha_2\notin\Delta$, we deduce that 
$$[f_1,x]= \la \gamma,\alpha_1^{\vee}\ra \cdot e_{\gamma}=(2-mn)\cdot e_{\gamma}\neq 0 \quad\textrm{in $\nn^+_{k}$},$$
as desired. 

Assume next that $p$ divides $2-mn$. Since $n$ is odd if $p=2$, this implies that $p$ does not divide $n(3-mn)$. Set $\delta:=s_1(\beta+\gamma)=\alpha_2+s_1(\gamma)$. Note that if $m\geq 2$, then
$$\la s_1(\delta),\alpha_1^{\vee}\ra=m+2-mn\leq 2-2m<0 \quad\textrm{and}\quad \la s_1(\delta),\alpha_2^{\vee}\ra=2-mn+n\leq 2-n<0,$$ 
while if $m=1$, so that $n\geq 5$, then 
$$\la s_2s_1(\delta),\alpha_1^{\vee}\ra=\la s_1(\delta),\alpha_1^{\vee}+\alpha_2^{\vee}\ra= 5-n\leq 0 \quad\textrm{and}\quad \la s_2s_1(\delta),\alpha_2^{\vee}\ra=-\la s_1(\delta),\alpha_2^{\vee}\ra= -2<0.$$  
Hence $\delta\in\Delta_+^{\im}$ by \cite[Theorem~5.4]{Kac}. Set $x:=[e_2,e_{\gamma'}]\in\nn^+_{k}$, where $\gamma'=s_1(\gamma)\in\Delta^{\re}_+$. Since $\gamma-\beta=-(m-1)\alpha_1+(n-1)\alpha_2\notin\Delta$ and hence also $\gamma'-\alpha_2=s_1(\gamma-\beta)\notin\Delta$, we deduce that 
$$[f_2,x]=\la \gamma',\alpha_2^{\vee}\ra \cdot e_{\gamma'}= \la \gamma,\beta^{\vee}\ra \cdot e_{\gamma'}=n(3-mn)\cdot e_{\gamma'}\neq 0 \quad\textrm{in $\nn^+_{k}$},$$
as desired.
\end{proof}

We record the following more precise version of \cite[Theorem~E]{simpleKM}.
\begin{prop}\label{prop:gen_TheoremE}
Let $k=\FF_q$ be a finite field. Consider the GCM $A_1=(\begin{smallmatrix}2 & -2\\ -2 & 2\end{smallmatrix})$ and $A_2=(\begin{smallmatrix}2 & -m\\ -n & 2\end{smallmatrix})$ with $m,n\geq 2$ and $mn>4$. Assume that $m\equiv n\equiv 2 \ (\modulo q-1)$. If $\charact k=2$, we moreover assume that at least one of $m$ and $n$ is odd. Then the minimal Kac--Moody groups $\G_{A_1}(\FF_q)$ and $\G_{A_2}(\FF_q)$ are isomorphic as abstract groups, but the simple quotients $\G^{pma}_{A_1}(\FF_q)/Z'_{A_1}$ and $\G^{pma}_{A_2}(\FF_q)/Z'_{A_2}$ of the corresponding Mathieu--Rousseau completions are not isomorphic as topological groups.
\end{prop}
\begin{proof}
The proof of \cite[Theorem~E]{simpleKM} on p.~725 of \emph{loc. cit.} applies \emph{verbatim}, with the same notation [note: that proof uses Lemma~5.3 in \cite{simpleKM}; for the convenience of the reader, we provide below (see Lemma~\ref{lemma:exceptional_isom}) a more detailed proof of that lemma]. The only difference is that \cite[Lemma~5.4]{simpleKM}, which is used to conclude the proof, must be replaced by its generalisation, Lemma~\ref{lemma:gen_Lemma5.4} above (hence the extra assumption in characteristic $2$).
\end{proof}

\begin{lemma}\label{lemma:exceptional_isom}
Let $k=\FF_q$ be a finite field. Consider the GCM $A=(\begin{smallmatrix}2&-m \\ -n&2\end{smallmatrix})$ and $B=(\begin{smallmatrix}2&-m' \\ -n'&2\end{smallmatrix})$ with $m,m',n,n'\geq 2$. Assume moreover that $m\equiv m' \ (\modulo q-1)$ and $n\equiv n' \ (\modulo q-1)$. Then the minimal Kac--Moody groups $\G_A(k)$ and $\G_{B}(k)$ are isomorphic as abstract groups, and the corresponding R\'emy-Ronan completions $\G_{A}^{rr}(\FF_q)$ and $\G_{B}^{rr}(\FF_q)$ are isomorphic as topological groups.
\end{lemma}
\begin{proof}
We can identify the Weyl groups $W(A)$ and $W(B)$ (both isomorphic to the infinite dihedral group), and hence also the corresponding sets of real roots $\Delta^{\re}(A)$ and $\Delta^{re}(B)$. Moreover, since $A$ and $B$ do not have any $-1$ entry, it follows from \cite[\S 3]{Morita88} that the commutation relations (\ref{R0}) are all trivial, and hence one can identify the Steinberg functors $\Stt_A$ and $\Stt_B$. Let us fix this identification $\Stt_A\stackrel{\sim}{\to}\Stt_B$ as follows (we add a superscript $A$ or $B$ to the usual notations, to distinguish between the objects related to the GCM $A$ or $B$). For $X\in\{A,B\}$, each real root $\alpha\in\Delta^{\re}_+(X)$ can be uniquely written as $\alpha=w_{\alpha}\alpha_i$ for some $w_{\alpha}\in W(X)$ and $i\in I=\{1,2\}$. We then choose the sign of $e_{\pm\alpha}^X$ in the double basis $E^X_{\alpha}$ by setting $$e_{\alpha}^X:=w_{\alpha}^*e_i^X\quad\textrm{and}\quad e_{-\alpha}^X:=w_{\alpha}^*f_i^X,$$
where $w_{\alpha}^*$ is as in (\ref{eqn:wstar_wtilde}) (see also \cite[Definition~7.58]{KMGbook}), and we define the corresponding parametrisations $x^X_{\pm\alpha}\co k\to U^X_{\pm\alpha}:r\mapsto\exp(re_{\pm\alpha})$ of the real root groups accordingly. The identification $\Stt_A\stackrel{\sim}{\to}\Stt_B$ is now obtained by mapping $x^A_{\alpha}$ to $x^B_{\alpha}$ for each $\alpha\in\Delta^{\re}(A)=\Delta^{\re}(B)$. 

Similarly, identifying the coroots associated to $A$ and $B$, we obtain an identification of the tori $\T_A(k)\stackrel{\sim}{\to}\T_B(k)$ mapping $r^{\alpha_i^{\vee}}\in\T_A(k)$ ($r\in k^{\times}$, $i\in I$) to the corresponding element of $\T_B(k)$. This yields an isomorphism $\varphi\co \Stt_A(k)*\T_A(k)\to\Stt_B(k)*\T_B(k)$, and to see it induces an isomorphism $\G_A(k)\to\G_B(k)$, we only have to show that the relations (\ref{R1})--(\ref{R4}) are the same in $\Stt_A(k)*\T_A(k)$ and $\Stt_B(k)*\T_B(k)$.

For the relations (\ref{R3}), this is clear by construction. For the relations (\ref{R1}) and (\ref{R2}), this follows from the fact that
\begin{equation}\label{eqn:congruence}
r^m=r^{m'}\quad\textrm{and}\quad r^n=r^{n'}\quad\textrm{for all $r\in k$.}
\end{equation}
Finally, for the relations (\ref{R4}), let $i\in I$, $\alpha\in\Delta^{\re}$ and $r\in k$, and let us check that $\widetilde{s}_i\cdot x^A_{\alpha}(r)\cdot \widetilde{s}_i^{\thinspace -1}\cdot (s_i^*x^A_{\alpha}(r))\inv$ is mapped to $\widetilde{s}_i\cdot x^B_{\alpha}(r)\cdot \widetilde{s}_i^{\thinspace -1}\cdot (s_i^*x^B_{\alpha}(r))\inv$ under $\varphi$, or else that
\begin{equation}\label{eqn:TPR4}
\varphi(s_i^*x^A_{\alpha}(r))=s_i^*x^B_{\alpha}(r).
\end{equation}
 We may assume that $\alpha\in\Delta^{\re}_+$ (the case $\alpha\in\Delta^{\re}_-$ being symmetric). Let $X\in\{A,B\}$.  If $\alpha=\alpha_i$, then $s_i^*x_{\alpha}^X(r)=x_{-\alpha_i}^X(r)$, yielding (\ref{eqn:TPR4}) in that case. Assume now that $\alpha\in\Delta^{\re}_+\setminus\{\alpha_i\}$. By definition, $x_{\alpha}^X(r)=\exp(re^X_{\alpha})=\exp(r\cdot w_{\alpha}^*e^X_j)=w_{\alpha}^*x^X_{\alpha_j}(r)$ for some $j\in I$ (determined by $\alpha$). If $\ell(s_iw_{\alpha})=\ell(w_{\alpha})+1$ (where $\ell\co W\to \NN$ is the word length on $W=W(X)$ with respect to the generating set $\{s_1,s_2\}$), then $s_i^*w_{\alpha}^*=(s_iw_{\alpha})^*=w_{s_i\alpha}^*$, and hence $s_i^*x^X_{\alpha}(r)=x^X_{s_i\alpha}(r)$, yielding (\ref{eqn:TPR4}) in that case. Finally, suppose $\ell(s_iw_{\alpha})=\ell(w_{\alpha})-1$. Then $w^*_{\alpha}=s^*_i\cdot (s_iw_{\alpha})^*$, and hence 
 $s_i^*x^X_{\alpha}(r)=(s_i^*)^2\cdot x^X_{s_i\alpha}(r)$. On the other hand, by \cite[Proposition~4.18(6)]{KMGbook}, we have $$(s_i^*)^2\cdot x^X_{s_i\alpha}(r)=x^X_{s_i\alpha}((-1)^{\langle s_i\alpha,\alpha_i^{\vee}\rangle}r)=x^X_{s_i\alpha}((-1)^{\langle \alpha,\alpha_i^{\vee}\rangle}r).$$
It thus remains to check that $(-1)^{\langle \alpha,\alpha_i^{\vee}\rangle}$ yields the same element of $k$, regardless of whether $\alpha$ is viewed as a root of $\Delta^{\re}_+(A)$ or of $\Delta^{\re}_+(B)$. But if $\alpha=\alpha_j$ is a simple root, this follows from (\ref{eqn:congruence}), and in general, this follows from an easy induction on $\ell(w_{\alpha})$ using the fact that $\la s_j\alpha, \alpha_i^{\vee}\ra= \la\alpha,\alpha_i^{\vee}\ra -\la\alpha,\alpha_j^{\vee}\ra \cdot \la \alpha_j,\alpha_i^{\vee}\ra$. 

We have thus shown that $\varphi$ induces an isomorphism $\phi\co\G_A(k)\to\G_B(k)$. On the other hand, note that $\phi$ identifies the (positive) BN-pairs of $\G_A(k)$ and $\G_B(k)$ (see \S\ref{subsection:GKKND}), and hence also induces an isomorphism of topological groups between the corresponding R\'emy-Ronan completions, yielding the lemma.
\end{proof}

Finally, we prove a slight generalisation of \cite[Corollary~F]{simpleKM}.
\begin{lemma}\label{lemma:nondensity_initial}
Let $k=\FF_q$ be a finite field. Consider the GCM $A=(\begin{smallmatrix}2 & -m\\ -n & 2\end{smallmatrix})$ with $m,n\geq 2$ and $mn>4$. Assume that $m\equiv n\equiv 2 \ (\modulo q-1)$. If $\charact k=2$, we moreover assume that at least one of $m$ and $n$ is odd. Then $U^+_{A}(\FF_q)$ is not dense in $\U_{A}^{ma+}(\FF_q)$.
\end{lemma}
\begin{proof}
Set $A_1=(\begin{smallmatrix}2 & -2\\ -2 & 2\end{smallmatrix})$ and $A_2=A$. For $i=1,2$, we also set $G_i:=\G_{A_i}(\FF_q)$, $\widehat{G}_i:=\G^{pma}_{A_i}(\FF_q)$, and $Z'_i:=Z'_{A_i}$. It follows from Proposition~\ref{prop:gen_TheoremE} that $G_1$ and $G_2$ are isomorphic as abstract groups, whereas $\widehat{G}_1/Z'_1$ and $\widehat{G}_2/Z'_2$ are not isomorphic as topological groups. Note also that the R\'emy-Ronan completions $\G_{A_1}^{rr}(\FF_q)$ of $G_1$ and $\G_{A_2}^{rr}(\FF_q)$ of $G_2$ are isomorphic as topological groups by Lemma~\ref{lemma:exceptional_isom}.
Finally, we may assume without loss of generality that $G_1$ is dense in $\widehat{G}_1$, for otherwise $U^+_{A_1}(\FF_q)$ would not be dense in $\U_{A_1}^{ma+}(\FF_q)$ by Lemma~\ref{lemma:density_U_G}, so that the conclusion of the lemma would immediately follow from Proposition~\ref{proposition:non-density_induction}.

Assume now for a contradiction that $U^+_{A}(\FF_q)$ is dense in $\U_{A}^{ma+}(\FF_q)$. Then $G_2$ is dense in $\widehat{G}_2$ by Lemma~\ref{lemma:density_U_G}. Hence the continuous surjective group homomorphisms $\varphi_{A_i}\co \widehat{G}_i\to \G_{A_i}^{rr}(\FF_q)$, $i=1,2$, induce isomorphisms
$$\widehat{G}_1/Z'_1\cong \G_{A_1}^{rr}(\FF_q)\cong \G_{A_2}^{rr}(\FF_q)\cong \widehat{G}_2/Z'_2$$ of topological groups,
yielding the desired contradiction. 
\end{proof}

\begin{theorem}\label{thm:non-density}
Let $k=\FF_q$ be a finite field, and let $A=(a_{ij})_{i,j\in I}$ be a GCM. Assume that there exist indices $i,j\in I$ such that $|a_{ij}|\geq q+1$ and $|a_{ji}|\geq 2$. 
Then $U^+_A(k)$ is not dense in $\U_A^{ma+}(k)$.
\end{theorem}
\begin{proof}
Consider the GCM $B=(\begin{smallmatrix}2 & -m\\ -n & 2\end{smallmatrix})$ with $m=q+1$ and $n=2$. Then $U^+_B(k)$ is not dense in $\U_B^{ma+}(k)$ by Lemma~\ref{lemma:nondensity_initial}. Since $B\leq (\begin{smallmatrix}2 & a_{ij}\\ a_{ji} & 2\end{smallmatrix})$ or $B\leq (\begin{smallmatrix}2 & a_{ji}\\ a_{ij} & 2\end{smallmatrix})$, the conclusion then follows from Proposition~\ref{proposition:non-density_induction}.
\end{proof}

We now give a completely different proof of Theorem~\ref{thm:non-density}, which provides another perspective on this non-density phenomenon.

\begin{prop}\label{prop:element_not_dense}
Let $k=\FF_q$ be a finite field, and let $A=(a_{ij})_{i,j\in I}$ be a GCM. Fix distinct $i,j\in I$, and let $g\in \U_A^{ma+}(k)\subseteq \widehat{\UU}_k^+$ be one of the twisted exponentials $[\exp][e_i,e_j]$ or $[\exp](\ad e_i)^{(q)}e_j$.
\begin{enumerate}
\item
If $|a_{ij}|\geq q$, then $g\notin \overline{[\U_A^{ma+}(k),\U_A^{ma+}(k)]}$. 
\item
If moreover $|a_{ij}|\geq q+1$ and $|a_{ji}|\geq 2$, then $g\notin\overline{U_A^+}(k)$.
\end{enumerate}
\end{prop}
\begin{proof}
As usual, we realise $\U_A^{ma+}(k)$ inside $\widehat{\UU}_k^+$. Assume that $|a_{ij}|\geq q$. We claim that for any element $h=\sum_{\alpha\in Q_+}{h_{\alpha}}\in V:=[\U_A^{ma+}(k),\U_A^{ma+}(k)]$, where $h_{\alpha}\in \UU^+_{\alpha}\otimes_{\ZZ}k$ for all $\alpha\in Q_+$, the homogeneous components $h_{\alpha_i+\alpha_j}$ and $h_{q\alpha_i+\alpha_j}$ are either both zero or both nonzero. 

Set
$$\Psi=Q_+\setminus\{m\alpha_i+n\alpha_j\in Q_+ \ | \ 0\leq m\leq q, \ 0\leq n\leq 1\}\quad\textrm{and}\quad \widehat{\UU}^+_{\Psi}:=\prod_{\alpha\in\Psi}{(\UU^+_{\alpha}\otimes_{\ZZ}k)}\subseteq \widehat{\UU}_k^+.$$
Note that $\widehat{\UU}^+_{\Psi}$ is an ideal of the $k$-algebra $\widehat{\UU}_k^+$. To prove the claim, we will compute modulo $\widehat{\UU}^+_{\Psi}$. 

Any element of $\U_A^{ma+}(k)$ is congruent modulo $\widehat{\UU}^+_{\Psi}$ to an element of the form
$$g_{\underline{\lambda}}:=\exp\lambda e_i\cdot \prod_{s=0}^q{[\exp]\lambda_s(\ad e_i)^{(s)}e_j}\equiv \exp\lambda e_i\cdot (1+x(\underline{\lambda}))\mod \widehat{\UU}^+_{\Psi}$$
for some tuple $\underline{\lambda}:=(\lambda,\lambda_0,\dots,\lambda_q)\in k^{q+2}$, where
$$x(\underline{\lambda}):=\sum_{s=0}^q{\lambda_s(\ad e_i)^{(s)}e_j}.$$
Using the identity (see for instance \cite[Lemma~4.9]{thesemoi})
\begin{equation*}
\begin{aligned}
\exp\mu e_i\cdot x(\underline{\lambda})\cdot \exp (-\mu e_i) &= (\exp \ad\mu e_i)x(\underline{\lambda})=\sum_{s=0}^q{\sum_{t\geq 0}{\lambda_s\mu^t\binom{s+t}{t}(\ad e_i)^{(s+t)}e_j}}
\end{aligned}
\end{equation*}
and the fact that 
$$\big([\exp]\lambda_s(\ad e_i)^{(s)}e_j\big)\inv \equiv 1-\lambda_s(\ad e_i)^{(s)}e_j \mod \widehat{\UU}^+_{\Psi},$$
we may now compute, for two tuples $\underline{\lambda}$ and $\underline{\mu}$ in $k^{q+2}$ as above, that
\begin{equation*}
\begin{aligned}
{[}g_{\underline{\lambda}},g_{\underline{\mu}}]&\equiv \exp\lambda e_i\cdot (1+x(\underline{\lambda}))\cdot\exp\mu e_i \cdot (1+x(\underline{\mu}))\cdot (1-x(\underline{\lambda}))\cdot\exp(-\lambda e_i) \cdot (1-x(\underline{\mu}))\cdot\exp(-\mu e_i)\\
&\equiv  1+(\exp \ad\lambda e_i)x(\underline{\lambda})+(\exp\ad (\lambda+\mu)e_i)(x(\underline{\mu})-x(\underline{\lambda}))-(\exp \ad\mu e_i)x(\underline{\mu})\\
&\equiv 1+\sum_{s=1}^q{C_s(\underline{\lambda},\underline{\mu})\cdot(\ad e_i)^{(s)}e_j} \mod \widehat{\UU}^+_{\Psi}
\end{aligned}
\end{equation*}
for some polynomials $C_s\in k[\lambda,\lambda_0,\dots,\lambda_q,\mu,\mu_0,\dots,\mu_q]$ satisfying 
$$C_1(\underline{\lambda},\underline{\mu})=\lambda\mu_0-\mu\lambda_0= \lambda^q\mu_0-\mu^q\lambda_0=C_q(\underline{\lambda},\underline{\mu}).$$
Here we used the fact that $\binom{q}{t}=0$ in $k$ unless $t=0$ or $t=q$. 

Let now $h=\sum_{\alpha\in Q_+}{h_{\alpha}}\in V$. Then $h$ is congruent modulo $\widehat{\UU}^+_{\Psi}$ to a (finite) product of elements of the form $[g_{\underline{\lambda}},g_{\underline{\mu}}]$ as above, say
$$h\equiv \prod_r{[g_{\underline{\lambda}^r},g_{\underline{\mu}^r}]}\equiv 1+\sum_r{([g_{\underline{\lambda}^r},g_{\underline{\mu}^r}]-1)}\mod \widehat{\UU}^+_{\Psi}$$ for some tuples $\underline{\lambda}^r$, $\underline{\mu}^r$ in $k^{q+2}$. The above discussion then implies that there is some $c\in k$ such that
$$h_{\alpha_i+\alpha_j}=c[e_i,e_j]\quad\textrm{and}\quad h_{q\alpha_i+\alpha_j}=c(\ad e_i)^{(q)}e_j,$$
proving our claim. This shows in particular that $g\notin V+ \widehat{\UU}^+_{\Psi}$. Since $1+\widehat{\UU}^+_{\Psi}$ contains the open subgroup $\U^{ma}_{q+2}(k)$, so that $V\U^{ma}_{q+2}(k)\subseteq V+\widehat{\UU}^+_{\Psi}$, we deduce that $g\notin \overline{V}$, proving (1).

Assume now that $|a_{ij}|\geq q+1$ and $|a_{ji}|\geq 2$. In particular, the only real roots not in $\Psi$ are the simple roots $\alpha_i$ and $\alpha_j$ (see \cite[Chapter~5]{Kac}). Assume for a contradiction that $g\in\overline{U_A^+}(k)$. Then 
$$g\equiv \exp(\lambda e_i)\exp(\mu e_j) \mod V+ \widehat{\UU}^+_{\Psi}$$
for some $\lambda,\mu\in k$. Since $V+ \widehat{\UU}^+_{\Psi}\subseteq 1+\widehat{\UU}^+_{\geq 2}$, where $\widehat{\UU}^+_{\geq 2}:=\prod_{\height(\alpha)\geq 2}{(\UU^+_{\alpha}\otimes_{\ZZ}k)}\subseteq\widehat{\UU}^+_k$, the components of degree $\alpha_i$ and $\alpha_j$ of $\exp(\lambda e_i)\exp(\mu e_j)$ must be zero, so that $\lambda=\mu=0$. Hence $g\in V+ \widehat{\UU}^+_{\Psi}$. But this contradicts the first part of the proof, yielding (2).
\end{proof}

As pointed out to us by Pierre-Emmanuel Caprace, the methods of this section can also be used to show that Kac--Moody groups $G^{pma}_A(k)$ (or even $\overline{\G_A}(k)$) are in general not GK-simple if $\charact k\leq M_A$.
\begin{prop}\label{prop:noGK}
Let $k=\FF_q$ be a finite field. Consider the GCM $A=(\begin{smallmatrix}2 & -m\\ -n & 2\end{smallmatrix})$ with $m,n\geq 2$ and $mn>4$. Assume that $m\equiv n\equiv 2 \ (\modulo q-1)$. If $\charact k=2$, we moreover assume that at least one of $m$ and $n$ is odd. Then $\G^{pma}_{A}(k)$ and $\overline{\G_{A}}(k)$ are not GK-simple, that is, $Z'_A\cap\overline{U_A^+}(k)\neq\{1\}$.
\end{prop}
\begin{proof}
Consider the (affine) GCM $B=(\begin{smallmatrix}2 & -2\\ -2 & 2\end{smallmatrix})$. Note first that the hypotheses of Lemma~\ref{lemma:nondensity_initial} are satisfied. Hence, as noted in the proof of this lemma, there is an isomorphism of the R\'emy-Ronan completions $\G_{A}^{rr}(k)$ of $\G_A(k)$ and $\G_{B}^{rr}(k)$ of $\G_B(k)$  preserving the corresponding BN-pair structures. In particular, the R\'emy-Ronan completions $U^{rr+}_A(k)\subseteq \G_{A}^{rr}(k)$ of $U_A^+(k)$ and $U^{rr+}_B(k)\subseteq \G_{B}^{rr}(k)$ of $U_B^+(k)$ are isomorphic. 

Assume for a contradiction that $Z'_A\cap\overline{U_A^+}(k)=\{1\}$. Then the surjective homomorphism $\varphi_A\co \overline{U_A^+}(k)\to U^{rr+}_A(k)$ (see \S\ref{subsection:GKKND}) is an isomorphism, so that $\overline{U_A^+}(k)\cong U^{rr+}_A(k)\cong U^{rr+}_B(k)$. On the other hand, it follows from \cite{Riehm70} (and the fact that $\G_{B}^{rr}(k)\cong \mathrm{PSL}_2(k(\!(t)\!))$) that $U^{rr+}_B(k)$ is just-infinite: every proper quotient of $U^{rr+}_B(k)$ is finite. But Corollary~\ref{corollary:functoriality_min} provides a map $\pi_{AB}\co \overline{U_A^+}(k)\to \overline{U_B^+}(k)$ with nontrivial kernel: in fact, $\ker\pi_{AB}$ is even infinite, as it contains all real root groups in $U_A^+(k)$ associated to positive real roots $\alpha=x\alpha_1+y\alpha_2$ with $x,y\geq 2$ (i.e., by \cite[Exercises~5.25--5.27]{Kac}, the element $\alpha$ is a positive real root in both $\Delta_+^{\re}(A)$ and $\Delta_+^{\re}(B)$ if and only if $nx^2-mnxy+my^2\in\{m,n\}$ and $|x-y|=1$, which is easily seen to have no positive integral solutions $(x,y)$ other than $(x,y)=(1,2)$ if $n=2$ and $(x,y)=(2,1)$ if $m=2$. One then concludes as in Remark~\ref{remark:restriction_minimal_groups}). Moreover, $\pi_{AB}$ has infinite image, as $\pi_{AB}(U_A^+(k))$ contains the subgroup of $U_B^+(k)$ generated by the simple root groups. Hence $\overline{U_A^+}(k)$ cannot be just-infinite, a contradiction.
\end{proof}

\section{Non-linearity}\label{section:nonlinearity}
This section is devoted to the proof of Theorem~\ref{thmintro:nonlinearity}. For earlier contributions to the linearity problem for the group $\U_A^{ma+}(k)$ over a finite field $k$, we refer to \cite[\S 4.2]{RCap} (see also \cite{CS14}).

We recall that a GCM $A=(a_{ij})_{i,j\in I}$ is called \emph{indecomposable} if, up to a permutation of the index set $I$, it does not admit any nontrivial block-diagonal decomposition $A=(\begin{smallmatrix}A_1 & 0\\ 0 & A_2\end{smallmatrix})$. Indecomposable GCM are either of finite, affine or indefinite type (see \cite[Chapter~4]{Kac}). If $A$ is of indefinite type and all proper submatrices of $A$ (corresponding to proper subdiagrams of the Dynkin diagram of $A$) are of finite type, then $A$ is moreover said to be of \emph{compact hyperbolic type}.

\begin{lemma}\label{lemma:GDA}
Let $A$ be a GCM of compact hyperbolic type. Then there exists some $B\leq A$ such that $B$ is of affine type.
\end{lemma}
\begin{proof}
We use the notation of \cite[\S 4.8]{Kac} for the parametrisation of affine GCM. If $A$ is of rank $2$, then one can take for $B$ the GCM of affine type $A_1^{(1)}$ or $A_2^{(2)}$. If the Dynkin diagram of $A$ is a cycle of length $\ell+1$ for some $\ell\geq 2$, then one can take for $B$ the GCM of affine type $A^{(1)}_{\ell}$. Assume now that the Dynkin diagram of $A$ is not a cycle and that $A$ is of rank at least $3$. Then $A$ must correspond to one of the $7$ Dynkin diagrams $H^{(3)}_{100}$, $H^{(3)}_{106}$, $H^{(3)}_{101}$, $H^{(3)}_{105}$, $H^{(3)}_{114}$, $H^{(3)}_{115}$ and $H^{(3)}_{116}$ from \cite[Section~7]{MR2608277}. One can then respectively choose $B$ to be affine of type $D^{(3)}_4$, $G^{(1)}_2$, $D^{(3)}_4$, $G^{(1)}_2$, $D^{(3)}_4$, $D^{(3)}_4$ and $G^{(1)}_2$.
\end{proof}

Using the results of \cite{CS14}, we can now prove our non-linearity theorem.

\begin{theorem}\label{thm:non-linearity}
Let $A$ be an indecomposable GCM of non-finite type and let $k$ be a finite field. Assume that $\G_A^{pma}(k)$ is GK-simple and set $G:=\G_A^{pma}(k)/Z_A'$. Then the following assertions are equivalent:
\begin{enumerate}
\item
Every compact open subgroup of $G$ is just-infinite (i.e. possesses only finite proper quotients).
\item
$\U_A^{ma+}(k)$ is linear over a local field.
\item
$G$ is a simple algebraic group over a local field.
\item
The matrix $A$ is of affine type.
\end{enumerate}
\end{theorem}
\begin{proof}
Note that the GK-simplicity assumption on $\G_A^{pma}(k)$ allows to view $\U_A^{ma+}(k)$ (rather than a quotient of $\U_A^{ma+}(k)$) as a subgroup of the simple group $G$.

The implications $(4)\Rightarrow (3)\Rightarrow (2)$ are clear. Since $G$ is a non-discrete, compactly generated, topologically simple, totally disconnected locally compact group (see, for instance, \cite[Appendix~A]{CRWpart2}) and since $\U_A^{ma+}(k)$ is an open compact subgroup of $G$, the implication $(2)\Rightarrow (3)$ follows from \cite[Corollary~1.4]{CS14}, while the implication $(3)\Rightarrow (1)$ follows from \cite[Theorem~2.6]{CS14}. We are thus left with the proof of $(1)\Rightarrow (4)$.

Assume thus that $\U_A^{ma+}(k)$ is just-infinite, and suppose for a contradiction that $A$ is of indefinite type. Assume first that $A=(a_{ij})_{i,j\in I}$ has a proper submatrix $(a_{ij})_{i,j\in J}$ of non-finite type. Consider the closed sets of positive roots $$\Psi_J:=\Delta_+(A)\cap\bigoplus_{j\in J}{\NN\alpha_j}\quad\textrm{and}\quad\Psi_{I\setminus J}:=\Delta_+(A)\setminus\Psi_J.$$
Note that $\Psi_{I\setminus J}$ is an ideal in $\Delta_+(A)$, in the sense that $\alpha+\beta\in \Psi_{I\setminus J}$ for all $\alpha\in\Delta_+(A)$ and $\beta\in \Psi_{I\setminus J}$ such that $\alpha+\beta\in\Delta_+(A)$. It then follows from \cite[Lemme~3.3(c)]{Rousseau} that $\U_{\Psi_{I\setminus J}}^{ma}(k)$ is normal in $\U_A^{ma+}(k)$ and that 
$$\U_A^{ma+}(k)/\U_{\Psi_{I\setminus J}}^{ma}(k)\cong \U_{\Psi_{J}}^{ma}(k)$$
is infinite, contradicting (1).

We may thus assume that $A$ is of compact hyperbolic type. By Lemma~\ref{lemma:GDA}, there exists a matrix $B$ of affine type such that $B\leq A$. It then follows from Theorem~\ref{thm:construction_pi} that there is a surjective map $\widehat{\pi}_{AB}\co \U_A^{ma+}(k)\to \U_B^{ma+}(k)$. This again yields an infinite quotient $$\U_A^{ma+}(k)/K\cong \U_B^{ma+}(k)$$ of $\U_A^{ma+}(k)$ for $K:=\ker\widehat{\pi}_{AB}$, in contradiction with (1). This concludes the proof of the theorem.
\end{proof}

\begin{remark}\label{remark:KZ'}
Note that, up to replacing $\U_A^{ma+}(k)$ by $\U_A^{ma+}(k)/Z$ where $Z:=Z'_A\cap \U_A^{ma+}(k)$ in the statement of Theorem~\ref{thm:non-linearity}, the GK-simplicity assumption on $\G_A^{pma}(k)$ can be substantially weakened. Indeed, the only issue that may arise in the above proof of Theorem~\ref{thm:non-linearity} if we replace $\U_A^{ma+}(k)$ by its quotient $\U_A^{ma+}(k)/Z$ is that for $A$ of compact hyperbolic type, the implication $(1)\Rightarrow (4)$ would require to ensure that the map $$\U_A^{ma+}(k)/Z\to \U_B^{ma+}(k)/\widehat{\pi}_{AB}(Z)$$ induced by $\widehat{\pi}_{AB}$ has still infinite image. In other words, we need to know that $KZ$ is not open in $\U_A^{ma+}(k)$ where $K:=\ker\widehat{\pi}_{AB}$, which is \emph{a priori} much weaker than the GK-simplicity assumption $Z=\{1\}$.
\end{remark}

\section{On the isomorphism problem}\label{section:isomproblem}
This section is devoted to the proof of Proposition~\ref{thmintro:isomproblem}.
Let $A=(a_{ij})_{i,j\in I}$ be a GCM, and let $k$ be a field. Set $\gamma_1(\U_A^{ma+}(k)):=\U_A^{ma+}(k)$, and for each $n\geq 1$, define recursively $$\gamma_{n+1}(\U_A^{ma+}(k)):=\overline{[\U_A^{ma+}(k),\gamma_n(\U_A^{ma+}(k))]},$$ 
that is, $\gamma_{n+1}(\U_A^{ma+}(k))$ is the closure in $\U_A^{ma+}(k)$ of the commutator subgroup $[\U_A^{ma+}(k),\gamma_n(\U_A^{ma+}(k))]$.

\begin{remark}\label{remark:LCS}
If $k$ is a finite field of characteristic $p>M_A$, then $\U_A^{ma+}(k)$ is a finitely generated pro-$p$ group by \cite[\S 2.2]{RCap}. It then follows from \cite[Exercise~1.17]{padicanalytic} that $\big(\gamma_n(\U_A^{ma+}(k))\big)_{n\geq 1}$ coincides with the lower central series of $\U_A^{ma+}(k)$.
\end{remark}

The proof of the following proposition is an adaptation of the proof of \cite[Proposition~6.11]{Rousseau} (see also \cite[\S 2.2]{RCap}).
\begin{prop}\label{prop:suite_centrale}
Let $A=(a_{ij})_{i,j\in I}$ be a GCM and let $k$ be a field. Assume that $\charact k=0$ or that $\charact k>M_A$. Then $\gamma_n(\U_A^{ma+}(k))=\U_{A,n}^{ma}(k)$ for all $n\geq 1$.
\end{prop}
\begin{proof}
To lighten the notation, we set $U^{ma+}=\U_A^{ma+}(k)$ and $U^{ma}_n=\U_{A,n}^{ma}(k)$. 
Given some $n\geq 1$, it follows from \cite[proof of Proposition~6.11]{Rousseau} that $$U^{ma}_m\subseteq [U^{ma+},U^{ma}_n]\cdot U^{ma}_{m+1}\quad\textrm{for all $m\geq n+1$.}$$ Indeed, in the notation of \emph{loc. cit.}, G.~Rousseau proves that for any given $g\in U^{ma}_m$, there exists some $i\in I$ and some $h\in U^{ma}_{m-1}$ such that $g\equiv [\exp e_i,h] \mod U^{ma}_{m+1}$, yielding the claim. 
By definition of the topology on $U^{ma+}$, we deduce that $U^{ma}_{n+1}\subseteq \overline{[U^{ma+},U^{ma}_n]}$ for all $n\geq 1$.
Since the reverse inclusion holds as well by \cite[Lemme~3.3]{Rousseau}, so that
$$U^{ma}_{n+1}= \overline{[U^{ma+},U^{ma}_n]}\quad\textrm{for all $n\geq 1$},$$
the proposition follows from an easy induction on $n$.
\end{proof}

\begin{remark}
If $k=\FF_q$ is finite and such that $|a_{ij}|\geq q$ for some $i,j\in I$, Proposition~\ref{prop:element_not_dense} shows that the conclusion of Proposition~\ref{prop:suite_centrale} does not hold anymore, i.e. $\gamma_2(\U_A^{ma+}(k))$ is properly contained in $\U_{A,2}^{ma}(k)$.
\end{remark}

We now apply the above observations to the study of the isomorphism problem for Mathieu--Rousseau completions of Kac--Moody groups over finite fields. 
We first record some known facts about complete Kac--Moody groups allowing to recognise specific subgroups from the topological group structure. 

For this, we will need to define Kac--Moody groups in a slightly more general context, namely by considering arbitrary Kac--Moody root data (see for instance \cite[\S 1.1]{Rousseau} or \cite[\S 7.3]{KMGbook}). 

To simplify the notation, we have so far considered Kac--Moody root data $\DDD$ of simply connected type, as we are mainly interested in the structure of the subgroup $\U_A^{ma+}$, which only depends on the GCM $A$ and not on a specific choice of $\DDD$. For $\DDD=(I,A,\Lambda,(c_i)_{i\in I},(h_i)_{i\in I})$ arbitrary with associated GCM $A=(a_{ij})_{i,j\in I}$, we denote by $\G_{\DDD}^{pma}$ the Mathieu--Rousseau completion of the Tits functor $\G_{\DDD}$ of type $\DDD$ (see \cite[\S 3.19]{Rousseau}), and by $Z'_{\DDD}$ the kernel of the action of $\G_{\DDD}^{pma}$ on its associated building.

The additional information provided by $\DDD$ is encoded in the torus scheme $\T_{\DDD}$. We denote as before by $\B_{\DDD}^{ma+}=\T_{\DDD}\ltimes\U_A^{ma+}$ the standard Borel subgroup of $\G_{\DDD}^{pma}$. Given a subset $J\subseteq I$, we let $\PP_{\DDD}^{ma+}(J)$ denote the standard parabolic subgroup of $\G_{\DDD}^{pma}$ of type $J$ (see \cite[\S 3.10]{Rousseau}). We also set $\DDD(J):=(J,A|_J,\Lambda,(c_i)_{i\in J},(h_i)_{i\in J})$ where $A|_J=(a_{ij})_{i,j\in J}$ and $\Delta_+(J):=\Delta_+\cap \bigoplus_{j\in J}{\ZZ\alpha_j}$.

\begin{lemma}\label{lemma:basic_facts_KMpro-p}
Let $\DDD$ be a Kac--Moody root datum with associated GCM $A=(a_{ij})_{i,j\in I}$ and let $k$ be a finite field of characteristic $p$. Then the following hold:
\begin{enumerate}
\item
If $\G_{\DDD}^{pma}(k)$ contains an open pro-$q$ subgroup for some prime $q$, then $q=p$.
\item
Every maximal pro-$p$ subgroup of $\G_{\DDD}^{pma}(k)$ is conjugate to $\U_A^{ma+}(k)$. 
\item
The normaliser of $\U_A^{ma+}(k)$ in $\G_{\DDD}^{pma}(k)$ is the standard Borel subgroup $\B_{\DDD}^{ma+}(k)$.
\item
The subgroups of $\G_{\DDD}^{pma}(k)$ containing $\B_{\DDD}^{ma+}(k)$ are precisely the standard parabolic subgroups of $\G_{\DDD}^{pma}(k)$.
\item
For any subset $J\subset I$, one has a Levi decomposition $\PP_{\DDD}^{ma+}(J)=\G^{pma}_{\DDD(J)}\ltimes\U^{ma}_{\Delta_+\setminus\Delta_+(J)}$. Moreover, $$\bigcap_{g\in \PP_{\DDD}^{ma+}(J)}{g\U_A^{ma+}(k)g\inv}=(Z'_{\DDD(J)}\cap \U^{ma}_{\Delta_+(J)}(k))\ltimes \U^{ma}_{\Delta_+\setminus\Delta_+(J)}(k).$$
\end{enumerate}
\end{lemma}
\begin{proof}
To prove (1), let $V$ be an open pro-$q$ subgroup of $\G_{\DDD}^{pma}(k)$. Then $V':=V\cap \U_A^{ma+}(k)$ is open in $V$, hence an open pro-$q$ subgroup of $\U_A^{ma+}(k)$ (see e.g. \cite[Proposition~1.11(i)]{padicanalytic}). Since $\U_A^{ma+}(k)$ is pro-$p$, the same argument implies that $V'$ is pro-$p$, and hence $q=p$.

The second statement follows from \cite[1.B.2]{Remytopolsimple} (see also \cite[Section~2]{RCap}). The statements (3) and (4) are standard (see e.g. \cite[Theorem~6.43]{BrownAbr}). The Levi decomposition in (5) follows from \cite[3.10]{Rousseau}. 

Let us now prove the identity in (5). Since $\U^{ma}_{\Delta_+\setminus\Delta_+(J)}(k)$ is normal in $\PP_{\DDD}^{ma+}(J)$ (see the above Levi decomposition) and since $Z'_{\DDD(J)}\cap \U^{ma}_{\Delta_+(J)}(k)$ is the Gabber--Kac kernel of $\G^{pma}_{\DDD(J)}$ (hence is conjugate under any element of $\PP_{\DDD}^{ma+}(J)=\G^{pma}_{\DDD(J)}\ltimes\U^{ma}_{\Delta_+\setminus\Delta_+(J)}$ to an element of $\U^{ma+}_A(k)$), the inclusion from right to left is clear. Conversely, since the Gabber--Kac kernel $Z'_{\DDD(J)}\cap \U^{ma}_{\Delta_+(J)}(k)$ is the largest normal subgroup of $\G^{pma}_{\DDD(J)}$ that is contained in $\U^{ma}_{\Delta_+(J)}(k)$, the image of $\bigcap_{g\in \PP_{\DDD}^{ma+}(J)}{g\U_A^{ma+}(k)g\inv}$ under the quotient map $$\U^{ma+}_A(k)=\U^{ma}_{\Delta_+(J)}(k)\ltimes \U^{ma}_{\Delta_+\setminus\Delta_+(J)}(k)\to \U^{ma}_{\Delta_+(J)}(k)$$ (see \cite[Lemme~3.3(c)]{Rousseau}) is contained in $Z'_{\DDD(J)}\cap \U^{ma}_{\Delta_+(J)}(k)$, as desired.
\end{proof}

To lighten the notation, we will write $H/Z'_A:=H/(H\cap Z'_A)$ for any subgroup $H$ of $\G^{pma}_A(k)$.

\begin{lemma}\label{lemma:isom1}
Let $\DDD,\DDD'$ be Kac--Moody root data with associated GCM $A=(a_{ij})_{i,j\in I}$ and $A'=(a'_{ij})_{i,j\in I'}$, respectively. Let also $k,k'$ be finite fields. If $\alpha\co\G_{\DDD}^{pma}(k)/Z'_{\DDD}\to\G_{\DDD'}^{pma}(k')/Z'_{\DDD'}$ is an isomorphism of topological groups, then $k\cong k'$ and there exist an inner automorphism $\gamma$ of $\G_{\DDD'}^{pma}(k')/Z'_{\DDD'}$ and a bijection $\sigma\co I\to I'$ such that 
$$\gamma\alpha(\U_{A|_{\{i,j\}}}^{ma+}(k)/Z'_{\DDD})=\U_{A'|_{\{\sigma(i),\sigma(j)\}}}^{ma+}(k')/Z'_{\DDD'}\quad\textrm{for all distinct $i,j\in I$.}$$ 
\end{lemma}
\begin{proof}
By Lemma~\ref{lemma:basic_facts_KMpro-p}(1) and (2), there exists an inner automorphism $\gamma$ of $\G_{\DDD'}^{pma}(k')/Z'_{\DDD'}$ such that $\gamma\alpha$ maps $\U_A^{ma+}(k)/Z'_{\DDD}$ to $\U_{A'}^{ma+}(k')/Z'_{\DDD'}$. Then $\gamma\alpha$ maps $\B_{\DDD}^{ma+}(k)/Z'_{\DDD}$ to $\B_{\DDD'}^{ma+}(k')/Z'_{\DDD'}$ by Lemma~\ref{lemma:basic_facts_KMpro-p}(3). Hence Lemma~\ref{lemma:basic_facts_KMpro-p}(4) implies that $\gamma\alpha$ maps maximal chains of standard parabolic subgroups in $\G_{\DDD}^{pma}(k)/Z'_{\DDD}$ to maximal chains of standard parabolic subgroups in $\G_{\DDD'}^{pma}(k')/Z'_{\DDD'}$. In particular, $|I|=|I'|$ and there exists a bijection $\sigma\co I\to I'$ such that $$\gamma\alpha(\PP_{\DDD}^{ma+}(\{i\})/Z'_{\DDD})=\PP_{\DDD'}^{ma+}(\{\sigma(i)\})/Z'_{\DDD'}\quad\textrm{for all $i\in I$.}$$ Hence $$\gamma\alpha(\PP_{\DDD}^{ma+}(\{i,j\})/Z'_{\DDD})=\PP_{\DDD'}^{ma+}(\{\sigma(i),\sigma(j)\})/Z'_{\DDD'}\quad\textrm{for all $i,j\in I$.}$$ It then follows from Lemma~\ref{lemma:basic_facts_KMpro-p}(5) that $$\gamma\alpha(\U^{ma}_{\Delta_+\setminus\Delta_+(\{i,j\})}(k)/Z'_{\DDD})=\U^{ma}_{\Delta_+\setminus\Delta_+(\{\sigma(i),\sigma(j)\})}(k')/Z'_{\DDD'}$$ and hence that $$\gamma\alpha(\U^{ma}_{\Delta_+(\{i,j\})}(k)/Z'_{\DDD})=\U^{ma}_{\Delta_+(\{\sigma(i),\sigma(j)\})}(k')/Z'_{\DDD'}\quad\textrm{for all $i,j\in I$}$$ because 
$$\U_A^{ma+}=\U^{ma}_{\Delta_+(J)}(k)\ltimes \U^{ma}_{\Delta_+\setminus\Delta_+(J)}(k)\quad\textrm{for all $J\subseteq I$},$$
and similarly for $\U_{A'}^{ma+}$. As $\U^{ma}_{\Delta_+(\{i,j\})}(k)=\U_{A|_{\{i,j\}}}^{ma+}(k)$, it thus remains to prove that $k\cong k'$.

Since each panel of the building $X_+$ of $\G_{\DDD}^{pma}(k)/Z'_{\DDD}$ (respectively, $X_+'$ of $\G_{\DDD'}^{pma}(k')/Z'_{\DDD'}$) is of cardinality $|k|+1$ (respectively, $|k'|+1$)(see for instance \cite[Chapter~7]{BrownAbr}), and since $X_+=X_+'$ (as simplicial complexes) by the above discussion, we deduce that $|k|=|k'|=:q$, and hence that $k\cong \FF_q\cong k'$. This concludes the proof of the lemma. 
\end{proof}

\begin{remark}\label{remark:alpha_lifts}
In the notation of Lemma~\ref{lemma:isom1}, if $\alpha$ lifts to an isomorphism $\alpha\co\G_{\DDD}^{pma}(k)\to\G_{\DDD'}^{pma}(k')$ and if $\G_{\DDD}^{pma}(k)$ is of rank $2$ (that is, $|I|=2$), then Lemma~\ref{lemma:basic_facts_KMpro-p}(1) and (2) implies that $$\gamma\alpha(\U_A^{ma+}(k))=\U_{A'}^{ma+}(k')$$ for some inner automorphism $\gamma$ of $\G_{\DDD'}^{pma}(k')$.
\end{remark}

\begin{lemma}\label{lemma:isom2}
Let $A=(a_{ij})_{i,j\in I}$ and $B=(b_{ij})_{i,j\in I}$ be GCM indexed by $I$ and let $k$ be a finite field with $p=\charact k>M_A,M_B$. Assume that the groups $\U_A^{ma+}(k)$ and $\U_B^{ma+}(k)$ are isomorphic. Then the following hold:
\begin{enumerate}
\item
$\sum_{\height(\alpha)=n}\dim\g(A)_{\alpha}=\sum_{\height(\alpha)=n}\dim\g(B)_{\alpha}$ for all $n\geq 1$.
\item
If $I=\{i,j\}$, then $B=(\begin{smallmatrix}2 & a_{ij}\\ a_{ji} & 2\end{smallmatrix})$ or $B=(\begin{smallmatrix}2 & a_{ji}\\ a_{ij} & 2\end{smallmatrix})$.
\end{enumerate}
\end{lemma}
\begin{proof}
Let $\alpha\co \U_A^{ma+}(k)\to \U_B^{ma+}(k)$ be an isomorphism. Then $\alpha$ maps $\U_{A,n}^{ma}(k)$ to $\U_{B,n}^{ma}(k)$ for each $n\geq 1$ by Proposition~\ref{prop:suite_centrale}, and hence induces isomorphisms of the quotients $$\U_{A,n}^{ma}(k)/\U_{A,n+1}^{ma}(k)\cong \U_{B,n}^{ma}(k)/\U_{B,n+1}^{ma}(k)\quad\textrm{for all $n\geq 1$}.$$
In turn, this yields isomorphisms of the additive groups $\bigoplus_{\height(\alpha)=n}\nn^+_k(A)_{\alpha}\cong\bigoplus_{\height(\alpha)=n}\nn^+_k(B)_{\alpha}$ by \cite[Lemma~3.3(e)]{Rousseau}. Hence (1) follows from the fact that if $d_n(A)=\sum_{\height(\alpha)=n}\dim\g(A)_{\alpha}$, then $|k|^{d_n(A)}$ is the cardinality of $\bigoplus_{\height(\alpha)=n}\nn^+_k(A)_{\alpha}$.

Assume now that $I=\{i,j\}$. For $X\in\{A,B\}$, let $\ii^+(X)$ be the ideal of the free Lie algebra $\tilde{\nn}^{+}(X)$ generated by the Serre relations $x_{ij}^{+}(X)=\ad (e_i)^{1+|X_{ij}|}e_j$ and $x_{ji}^{+}(X)=\ad (e_j)^{1+|X_{ji}|}e_i$. For each $n\geq 1$, let also $\tilde{\nn}_n^{+}(X)$ denote the subspace of elements of $\tilde{\nn}^{+}(X)$ of total degree $n$, that is, the linear span of all brackets of the form $[e_{i_1},\dots,e_{i_n}]$ ($i_s\in I$). In particular, since $\ii^+(X)$ is graded, 
$$\tilde{\nn}_n^{+}(X)/\ii_n^+(X)=\nn_n^+(X)\quad\textrm{for all $n\geq 1$}$$
as vector spaces, where $\ii_n^+(X):=\ii^+(X)\cap \tilde{\nn}_n^{+}(X)$ and $\nn_n^+(X):=\bigoplus_{\height(\alpha)=n}\nn^+(X)_{\alpha}$.
The above discussion now implies that 
$$\dim \ii_n^+(A)=\dim \tilde{\nn}_n^{+}(A)-\dim \nn_n^+(A)=\dim \tilde{\nn}_n^{+}(B)-\dim \nn_n^+(B)=\dim \ii_n^+(B)\quad\textrm{for all $n\geq 1$}.$$
If $|a_{ij}|=|a_{ji}|=m$, then $\dim \ii_n^+(A)=0$ for all $n\leq m+1$, while $\dim \ii_{m+2}^+(A)=2$. The corresponding assertion for $B$ then implies that $|b_{ij}|=|b_{ji}|=m$, proving (2) in this case.

Assume now that $a_{ij}\neq a_{ji}$, say $m=|a_{ij}|<|a_{ji}|=m'$. Then $\dim \ii_n^+(A)=0$ for all $n\leq m+1$, while $\dim \ii_{m+2}^+(A)=1$. Again, the corresponding assertion for $B$ implies that $m=|b_{ij}|<|b_{ji}|$ or that $m=|b_{ji}|<|b_{ij}|$. Say $m=|b_{ij}|<|b_{ji}|=m''$. 
For $X\in\{A,B\}$, let $\ii^+_{ij}(X)$ denote the ideal of $\tilde{\nn}^{+}(X)$ generated by $x_{ij}^{+}(X)=\ad (e_i)^{1+m}e_j$. Assume for a contradiction that $m'\neq m''$, say $m'<m''$ (the case $m'>m''$ being similar). Then
$$\dim \ii^+_{m'+2}(A)=\dim (\ii^+_{ij}(A)\cap \tilde{\nn}_{m'+2}^{+}(A))+1=\dim (\ii^+_{ij}(B)\cap \tilde{\nn}_{m'+2}^{+}(B))+1=\dim\ii^+_{m'+2}(B)+1,$$
yielding the desired contradiction. This concludes the proof of (2).
\end{proof}

\begin{theorem}\label{thm:isomorphism}
Let $k,k'$ be finite fields, and let $A=(a_{ij})_{i,j\in I}$ and $B=(b_{ij})_{i,j\in J}$ be GCM. Assume that $p=\charact k>M_A,M_B$ and that all rank $2$ subgroups of $\G_{A}^{pma}(k)$ and $\G_{B}^{pma}(k')$ are GK-simple. 

If $\alpha\co\G_{A}^{pma}(k)/Z'_{A}\to\G_{B}^{pma}(k')/Z'_{B}$ is an isomorphism of topological groups, then $k\cong k'$, and there exist an inner automorphism $\gamma$ of $\G_{B}^{pma}(k')/Z'_{B}$ and a bijection $\sigma\co I\to J$ such that 
\begin{enumerate}
\item
$\gamma\alpha(\U_{A|_{\{i,j\}}}^{ma+}(k))=\U_{B|_{\{\sigma(i),\sigma(j)\}}}^{ma+}(k')$ for all distinct $i,j\in I$. 
\item
$B|_{\{\sigma(i),\sigma(j)\}}\in \big\{(\begin{smallmatrix}2 & a_{ij} \\ a_{ji} & 2 \end{smallmatrix}),(\begin{smallmatrix}2 & a_{ji} \\ a_{ij} & 2 \end{smallmatrix})\big\}$ for all distinct $i,j\in I$.
\end{enumerate}
\end{theorem}
\begin{proof}
Since all rank $2$ subgroups of $\G_{A}^{pma}(k)$ and $\G_{B}^{pma}(k')$ are GK-simple by assumption, (1) follows from Lemma~\ref{lemma:isom1} and (2) follows from Lemma~\ref{lemma:isom2}.
\end{proof}

\begin{remark}\label{remark:rank2isom}
In the notation of Theorem~\ref{thm:isomorphism}, if $\alpha$ lifts to an isomorphism $\alpha\co\G_{A}^{pma}(k)\to\G_{B}^{pma}(k')$ and if $\G_{A}^{pma}(k)$ is of rank $2$, then the conclusion of Theorem~\ref{thm:isomorphism} holds without any GK-simplicity assumption using Remark~\ref{remark:alpha_lifts} and Lemma~\ref{lemma:isom2}.
\end{remark}

We conclude this section with two further observations on the isomorphism problem, using the results from the previous sections.

\begin{lemma}
Let $A=(a_{ij})_{i,j\in I}$ and $B=(b_{ij})_{i,j\in I}$ be GCM, and let $k=\FF_q$ with $\charact k=p$. If $M_A<p$ and $M_B\geq q$, then $\U_A^{ma+}(k)$ and $\U_B^{ma+}(k)$ are not isomorphic.
\end{lemma}
\begin{proof}
By Proposition~\ref{prop:suite_centrale}, the quotient of $\U_A^{ma+}(k)$ by its commutator subgroup has cardinality $q^{|I|}$. On the other hand, it follows from Proposition~\ref{prop:element_not_dense} that the quotient of $\U_A^{ma+}(k)$ by its commutator subgroup has cardinality strictly larger than $q^{|I|}$. This proves the claim.
\end{proof}

\begin{prop}
Let $A=(a_{ij})_{i,j\in I}$ and $B=(b_{ij})_{i,j\in I}$ be GCM with $B\leq A$, and let $k$ be a finite field with $\charact k>M_A$.
If $\U_A^{ma+}(k)$ and $\U_B^{ma+}(k)$ are isomorphic, then $B=A$.
\end{prop}
\begin{proof}
Since $\U_A^{ma+}(k)$ is a finitely generated residually finite prop-$p$ group by \cite[\S 2.2]{RCap}, it is Hopfian, in the sense that every surjective homomorphism from $\U_A^{ma+}(k)$ to itself is an isomorphism (see Lemma~\ref{lemma:Schupp} below). Assume now for a contradiction that $B\neq A$. Then by Theorem~\ref{thm:construction_pi}, there is a surjective group homomorphism $\widehat{\pi}_{AB}\co \U_A^{ma+}(k)\to \U_B^{ma+}(k)$ with nontrivial kernel. Hence $\U_A^{ma+}(k)$ and $\U_B^{ma+}(k)$ cannot be isomorphic, for this would contradict the fact that $\U_A^{ma+}(k)$ is Hopfian.
\end{proof}

The following lemma and its proof are a straightforward adaptation of \cite[Theorem~4.10]{LyndonSchupp}.
\begin{lemma}\label{lemma:Schupp}
Let $G$ be a finitely generated residually finite pro-$p$ group. Then $G$ is Hopfian, i.e. every surjective homomorphism $G\to G$ is an isomorphism.
\end{lemma}
\begin{proof}
Let $\theta\co G\to G$ be a surjective homomorphism, and let $K$ be the kernel of $\theta$. Let $n\in\NN^*$. By \cite[Proposition~1.6 and Theorem~1.17]{padicanalytic}, there are only finitely many subgroups of $G$ of index $n$, say $M_1,\dots,M_r$. Then the subgroups $L_i:=\theta\inv(M_i)$ ($i=1,\dots,r$) are pairwise distinct and of index $n$ in $G$. Thus $\{M_1,\dots,M_r\}=\{L_1,\dots,L_r\}$. In particular, $$K\subseteq \bigcap_{i=1}^rL_i=\bigcap_{i=1}^rM_i,$$
and since $n$ was arbitrary, we deduce that $K$ is contained in the intersection of all finite-index subgroups of $G$. Since $G$ is residually finite, this implies that $K=\{1\}$, as desired.
\end{proof}

\begin{remark}
Lemma~\ref{lemma:Schupp} also holds when $G$ is a finitely generated residually finite profinite group. Indeed, the main result of \cite{NS07} (which relies on the classification of finite simple groups) asserts that finite-index subgroups of a finitely generated profinite group $G$ are automatically open, and hence $G$ has only finitely many subgroups of index $n$ for any given $n\in\NN^*$ by \cite[Proposition~1.6]{padicanalytic}. The proof of Lemma~\ref{lemma:Schupp} thus also holds in that case.
\end{remark}

\section{Zassenhaus--Jennings--Lazard series}\label{section:MDS}
This section is devoted to the proof of Proposition~\ref{thmintro:ZJLseries}. The general reference for this section is \cite[Chapter~11]{padicanalytic}.

Given a group $G$, as well as some positive natural number $n$, we write $G^n$ for the subgroup of $G$ generated by the elements of the form $g^n$, $g\in G$. We also let $\gamma_n(G)$ denote the lower central series of $G$: $$\gamma_1(G)=G\quad\textrm{and}\quad \gamma_{n+1}(G)=[G,\gamma_n(G)]\quad\textrm{for all $n\geq 1$}.$$
[Here, we consider lower central series in the category of abstract groups; as noticed in Remark~\ref{remark:LCS}, when $G$ is a finitely generated pro-$p$ group, this coincides with the lower central series defined at the beginning of \S\ref{section:isomproblem}.]

Let $k=\FF_q$ be a finite field of characteristic $p$, let $A$ be a GCM, and set $G:=\U_A^{ma+}(k)$. Then $G$ is a prop-$p$ group. Set $\Gamma_n=\gamma_n(G)$, and let $D_n=D_n(G)$ be the series of characteristic subgroups of $G$ defined by $D_1:=G$ and for $n>1$, 
$$D_n:=D_{n^{*}}^p\cdot\prod_{i+j=n}{\big[D_i,D_j\big]},$$
where $n^*:=\lceil n/p \rceil$ is the least integer $r$ such that $pr\geq n$. The series $(D_n)_{n\geq 1}$ is called the \emph{Zassenhaus--Jennings--Lazard series} of $G$. The subgroups $D_n$ are also called the \emph{dimension subgroups} of $G$.

For each $n\geq 1$, the quotient $L_n:=D_n/D_{n+1}$ is an elementary abelian $p$-group. We view it as a vector space over $\FF_p$ and write the group operation additively. Then 
$$L:=\bigoplus_{n= 1}^{\infty}{L_n}$$
is a graded Lie algebra over $\FF_p$ for the Lie bracket
$$(\overline{x},\overline{y}):=[x,y]D_{i+j+1}\in L_{i+j},$$
where $\overline{x}=xD_{i+1}\in L_i$ and $\overline{y}=yD_{j+1}\in L_j$ (see \cite[p.280]{padicanalytic}). It is called the \emph{Zassenhaus--Jennings--Lazard Lie algebra} of $G$.
Note that the $p$-operation
$$[p]\co L_i\to L_{pi}:\overline{x}=xD_{i+1}\mapsto \overline{x}^{[p]}:=x^pD_{pi+1}$$
extends to a $p$-operation on $L$, turning $L$ into a restricted Lie algebra (\cite[Theorem~12.8]{padicanalytic}).

\begin{lemma}\label{lemma:ppower}
$\U^{ma}_n(k)^p\subseteq \U^{ma}_{np}(k)$ for all $n\geq 1$.
\end{lemma}
\begin{proof}
We realise as usual $\U_A^{ma+}(k)$ inside $\widehat{\UU}_k^+$.
For each $m\geq 1$, we set
$$\widehat{\UU}^+_{\geq m}:=\prod_{\height(\alpha)\geq m}{(\UU^+_{\alpha}\otimes_{\ZZ}k)}\subseteq \widehat{\UU}_k^+.$$
Let $g\in \U^{ma}_n(k)$. Then $g=1+x$ for some $x\in \widehat{\UU}^+_{\geq n}$, and hence 
$$g^p=(1+x)^p=1+x^p\in 1+ \widehat{\UU}^+_{\geq np}.$$
In particular, $g^p\in \U^{ma}_{np}(k)$, as desired.
\end{proof}

\begin{lemma}\label{lemma:comparison_series}
$\Gamma_n\leq D_n\leq \U^{ma}_n(k)$ for all $n\geq 1$.
\end{lemma}
\begin{proof}
The first inclusion follows by induction on $n$, since $\Gamma_1=G=D_1$ and since if $\Gamma_n\subseteq D_n$, then $$\Gamma_{n+1}=[G,\Gamma_n]\subseteq [D_1,D_n]\subseteq D_{n+1}.$$

Since $\big[\U^{ma}_i(k),\U^{ma}_j(k)\big]\subseteq \U^{ma}_{i+j}(k)$ for all $i,j\geq 1$ by \cite[Lemme~3.3]{Rousseau}, the second inclusion follows from Lemma~\ref{lemma:ppower} and the fact that $(D_n)_{n\geq 1}$ is the fastest descending series with $D_1=G$ such that $D_i^p\leq D_{pi}$ and $[D_i,D_j]\leq D_{i+j}$ for all $i,j\geq 1$.
\end{proof}

\begin{corollary}\label{corollary:ZJL}
Assume that $p>M_A$. Then $\Gamma_n=D_n=\U^{ma}_n(k)$ for all $n\geq 1$.
\end{corollary}
\begin{proof}
The equality $\Gamma_n=\U^{ma}_n(k)$ follows from Remark~\ref{remark:LCS} and Proposition~\ref{prop:suite_centrale}. The lemma then follows from Lemma~\ref{lemma:comparison_series}.
\end{proof}

For each $n\geq 1$, set $(\nn^+_k)_n:=\bigoplus_{\height(\alpha)=n}{(\nn^+_k)_{\alpha}}$. Then $L_n(\U_A^{ma+}(k)):=\U^{ma}_n(k)/\U^{ma}_{n+1}(k)$ is isomorphic to the additive group of $(\nn^+_k)_n$ by \cite[Lemme~3.3(e)]{Rousseau}. We view it as an $\FF_p$-vector space and write the group operation additively. Set $$L(\U_A^{ma+}(k)):=\bigoplus_{n=1}^{\infty}{L_n(\U_A^{ma+}(k))},$$ which we endow with the graded Lie algebra structure given by the Lie bracket $$(\overline{x},\overline{y}):=[x,y]\U^{ma}_{i+j+1}(k)$$
for $\overline{x}=x\U^{ma}_{i+1}(k)\in L_i(\U_A^{ma+}(k))$ and $\overline{y}=y\U^{ma}_{j+1}(k)\in L_j(\U_A^{ma+}(k))$.
\begin{lemma}\label{lemma:ZJL}
Let $k$ be a finite field of characteristic $p$. The map $\nn^+_k\to L(\U_A^{ma+}(k))$ mapping a homogeneous element $x\in\nn^+_k$ with $\height(\deg(x))=n$ to $([\exp]x)\U^{ma}_{n+1}(k)$ defines an isomorphism of Lie algebras over $\FF_p$.
\end{lemma}
\begin{proof}
This readily follows from the fact that if $x,y\in\nn^+_k$ are homogeneous with $\height(\deg(x))=i$ and $\height(\deg(y))=j$, then
$$\big[[\exp]x,[\exp]y\big]\equiv [\exp][x,y] \quad \modulo \U^{ma}_{i+j+1}(k).$$
\end{proof}

\begin{corollary}
Assume that $p>M_A$. Then $L=L(\U_A^{ma+}(k))\cong \nn^+_k$ as Lie algebras over $\FF_p$.
\end{corollary}
\begin{proof}
This readily follows from Corollary~\ref{corollary:ZJL} and Lemma~\ref{lemma:ZJL}.
\end{proof}

\bibliographystyle{amsalpha} 
\bibliography{these}

\end{document}